\documentclass[english]{amsart}

\usepackage{enumerate}
\usepackage{amssymb}
\usepackage{array}
\usepackage{graphicx}
\usepackage{lscape}
\usepackage{mathrsfs}
\usepackage{mathtools}
\usepackage{afterpage}
\usepackage{pdflscape}

\usepackage{color}
\usepackage[usenames,dvipsnames,svgnames,table]{xcolor}

\usepackage{pstricks-add}
\usepackage{pgfplots}
\usepackage{tkz-fct}  

\makeatletter
\@namedef{subjclassname@2020}{%
  \textup{2020} Mathematics Subject Classification}
\makeatother

\usepackage[curve,all]{xy}
\xyoption{matrix}
 \xyoption{curve}
 \xyoption{color}
 \xyoption{line}
\xyoption{arc}

\usepackage[shortlabels]{enumitem}
\setlist[enumerate]{label=\rm{(\arabic*)}}
\setlist[enumerate,2]{label=\rm({\it\roman*})}
\setlist[itemize]{label=\raisebox{0.25ex}{\tiny$\bullet$}}
\usepackage[backref=false, colorlinks, linktocpage, citecolor = blue, linkcolor = blue]{hyperref}
\usepackage[toc,page]{appendix}

\theoremstyle{plain}
\newtheorem{theorem}{Theorem}[section]
\newtheorem*{theoremaux}{Theorem \theoremauxnum}
\gdef\theoremauxnum{1}

\newtheorem{proposition}[theorem]{Proposition}
\newtheorem*{propositionaux}{Proposition \propositionauxnum}
\gdef\propositionauxnum{1}

\newtheorem{lemma}[theorem]{Lemma}
\newtheorem*{lemmaaux}{Lemma \lemmaauxnum}
\gdef\lemmaauxnum{1}

\newtheorem{corollary}[theorem]{Corollary}

\theoremstyle{definition}
\newtheorem{definition}[theorem]{Definition}
\newtheorem*{definition*}{Definition}

\newtheorem{example}[theorem]{Example}

\theoremstyle{remark}
\newtheorem{remark}[theorem]{Remark}
\newtheorem*{remark*}{Remark}

\newcommand{\incl}[1][r]{\ar@<-0.2pc>@{^(-}[#1] \ar@<+0.2pc>@{-}[#1]}

\newcommand{\hs}{\kern 0.8pt}

\newcommand{\FF}{{\mathbb{F}}}

\renewcommand{\P}{\mathbb{P}}

\newcommand{\Hom}{\mathrm{Hom}}

\renewcommand{\AA}{\mathcal{A}}
\newcommand{\BB}{\mathcal{B}}
\newcommand{\CC}{\mathcal{C}}
\newcommand{\AB}{\mathcal{AB}}

\newcommand{\Mon}{D_\infty}
\newcommand{\Spec}{\mathrm{Spec}}
\newcommand{\inn}{\mathrm{inn}}

\newcommand{\V}{\mathcal{V}}
\renewcommand{\SS}{\mathcal{S}}

\newcommand{\p}{\mathbb{P}}

\newcommand{\Q}{\mathbb{Q}}
\newcommand{\HH}{\mathrm{H}}

\newcommand{\Fk}{F_{\k}}
\newcommand{\RR}{\mathcal{R}}

\renewcommand{\k}{\mathrm{k}}

\newcommand{\Inn}{\mathrm{Inn}}

\newcommand{\A}{\mathbb{A}}

\newcommand{\C}{\mathbb{C}}
\newcommand{\F}{\mathbb{F}}
\newcommand{\N}{\mathbb{N}}

\newcommand{\Z}{\mathbb{Z}}
\newcommand{\R}{\mathbb{R}}
\renewcommand{\RR}{\mathcal{R}}
\newcommand{\Gm}{\mathbb{G}_m}

\renewcommand{\O}{\mathcal{O}}

\newcommand{\Vg}{\mathcal{V}^G(G/H)}
\newcommand{\Vgt}{\mathcal{V}^G(G)}

\newcommand{\Vone}{\mathcal{V}_1^G(G/H)}

\newcommand{\W}{\mathcal W}
\renewcommand{\k}{\overline{k}}
\newcommand{\DD}{\mathcal{D}}

\def\w{\omega}
\def\modH{\overline}

\DeclareMathOperator{\Card}{Card}
\DeclareMathOperator{\SL}{SL}
\DeclareMathOperator{\Def}{Def}
\DeclareMathOperator{\SU}{SU}
\DeclareMathOperator{\PSU}{PSU}

\DeclareMathOperator{\Aut}{Aut}

\DeclareMathOperator{\PGL}{PGL}
\DeclareMathOperator{\PSL}{PSL}

\DeclareMathOperator{\Quot}{Quot}

\DeclareMathOperator{\Gal}{Gal}

\DeclareMathOperator{\SO}{SO}
\DeclareMathOperator{\Bir}{Bir}

\DeclareMathOperator{\Stab}{Stab}

\title[Forms of almost homogeneous varieties over perfect fields]{Forms of almost homogeneous varieties \\ over perfect fields}

\author[Lucy Moser-Jauslin and Ronan Terpereau]{Lucy Moser-Jauslin and Ronan Terpereau}
\thanks{The second-named author is supported by the ANR Project FIBALGA ANR-18-CE40-0003-01.
This work received partial support from the French "Investissements d\textquoteright Avenir" program and from project ISITE-BFC (contract ANR-lS-IDEX-OOOB). The IMB receives support from  the EIPHI Graduate School (contract ANR-17-EURE-0002).
}

\address{Institut de Math\'{e}matiques de Bourgogne, UMR 5584 CNRS, Universit\'{e} de Bourgogne, F-21000 Dijon, France}
\email{lucy.moser-jauslin@u-bourgogne.fr}

\address{Univ. Lille, CNRS, UMR 8524 - Laboratoire Paul Painlev\'e, F-59000 Lille, France}
\email{ronan.terpereau@univ-lille.fr}

\begin{document}

\begin{abstract}
We study the $k$-forms of almost homogeneous varieties over perfect base fields $k$. 
First, we discuss criteria for the existence of $k$-forms in the homogeneous case. 
Then, we extend the Luna-Vust theory from algebraically closed fields to perfect fields to determine when a given $k$-form of the open orbit of an almost homogeneous variety extends to a $k$-form of the entire variety. 
Finally, in the last section, we apply these results to determine the real forms of complex almost homogeneous $\mathrm{SL}_2$-threefolds.
\end{abstract}

\maketitle

\tableofcontents

\vspace{-7mm}

\section{Introduction}

\subsection{Setting}
The goal of this article is to contribute to the study of forms of algebraic varieties by treating the case of \emph{almost homogeneous varieties}. These are the (normal) varieties on which a reductive algebraic group acts with a dense open orbit. Toric varieties, determinantal varieties, flag varieties, symmetric varieties (and more generally spherical varieties), normalizations of nilpotent orbit closures in Lie algebras, and prehomogeneous vector spaces are classical examples of almost homogeneous varieties.

Let $k$ be a perfect field, let $\k$ be a fixed algebraic closure of $k$, and let $\Gamma=\Gal(\k/k)$ be the absolute Galois group of $k$. In this article, we call \emph{$k$-form} of a variety $X$ over $\k$ any variety $Z$ over $k$ such that 
\[Z_{\k}=Z \times_{\Spec(k)} \Spec(\k) \simeq X.\] Giving a $k$-form of $X$ is equivalent to giving an \emph{effective descent datum} on $X$, i.e.~an algebraic semilinear action $\mu\colon \Gamma \to \Aut(X)$ stabilizing an affine covering of $X$, in which case the corresponding $k$-form is $X/\Gamma$, where $\Gamma$ acts on $X$ via $\mu$.

Since, in this article, we are mostly interested in varieties endowed with an algebraic group action (in particular, almost homogeneous varieties), we will only consider $k$-forms compatible with this algebraic group action. 
More precisely, given a connected linear algebraic group $G$ over $\k$ and a $G$-variety $X$ over $\k$, we fix a $k$-form $F$ of $G$ in the category of algebraic groups over $k$, and then consider the \emph{$(k,F)$-forms} of $X$, i.e.~the $F$-varieties $Z$ over $k$ such that $Z_{\k} \simeq X$ as $G$-varieties over $\k$. Any such form corresponds to an  \emph{effective $(G,\rho)$-equivariant descent datum} on $X$ (here $\rho$ stands for the descent datum on $G$ associated with the $k$-form $F$) through the map $X \mapsto X/\Gamma$; see \S\S~\ref{sec:Forms of algebraic groups and inner twists}-\ref{sec:descent data for algebraic varieties with group action} for details.

One of the main goals of this article is to give a complete classification of real forms of almost homogeneous $\SL_2$-threefolds. This is achieved in section \S~\ref{sec:real forms of SL2-threefolds}. Before doing this, we develop two independent sections with results on $k$-forms for general almost homogeneous varieties, which are used in the last section to obtain the classification.

\subsection{Overview of the article}

\subsubsection*{Forms of homogeneous spaces \emph{(\S~\ref{sec:forms of qh varieties2})}} 
In \S~\ref{sec:forms of qh varieties2} we focus on the problem of the existence of forms, and the parametrization of their isomorphism classes, for arbitrary homogeneous spaces under a connected linear algebraic group. 

Two descent data $\rho_1,\rho_2\colon \Gamma \to \Aut(G)$ on an algebraic group $G$ are called \emph{equivalent} if there exists an algebraic group automorphism  $\psi \in \Aut_{\k}(G)$ such that $\rho_{2,\gamma}=\psi \circ \rho_{1,\gamma} \circ \psi^{-1}$, for all $\gamma \in \Gamma$, in which case they correspond to isomorphic $k$-forms of $G$. If furthermore $\psi=\inn_g$ is an inner algebraic group automorphism, then $\rho_1$ and $\rho_2$ are called \emph{strongly equivalent}. Also, if $\rho$ is a descent datum on $G$ and if there exists a locally constant map $c\colon\Gamma \to G(\k)$ (for the Krull topology on $\Gamma$) such that $\rho_c:=\inn_c \circ \rho$ is a descent datum on $G$, then $\rho_c$ is called an \emph{inner twist} of $\rho$. Note that the relation of being inner twists defines an equivalence relation on the set of descent data on $G$.

The two main results of this section are Proposition \ref{prop:A}, which gives a criterion to determine when a given homogeneous space $G/H$ admits a $(k,F)$-form, and Theorem \ref{th:B}, which describes the possible obstruction to the existence of a  $(k,F)$-form for $G/H$ when changing the $k$-form $F$ of $G$ inside its equivalence class for the relation of being inner twists.
The proofs of these two results are  detailed in \S~\ref{sec:Existence of forms}.  The first result is proven by a straightforward calculation. As for the Theorem~\ref{th:B}, the first part is a special case of a result of  Borovoi and Gagliardi (see \cite[Theorem~1.6]{BG21}). We include a direct proof in the case of homogeneous spaces, relying only on Proposition~\ref{prop:A}.  First of all, this gives an explicit instance of Borovoi and Gagliardi's result, and it also allows us to give a self-contained presentation.

\begin{proposition} \label{prop:A} \emph{(\S~\ref{sec:Existence of forms})}
Let $G$ be a connected linear algebraic group over $\k$, and let $F$ be a $k$-form of $G$ corresponding to the descent datum $\rho$ on $G$. Let $H \subseteq G$ be an algebraic subgroup.
The homogeneous space $X=G/H$ admits a $(k,F)$-form if and only if there exists a locally constant map $t\colon \Gamma \to G(\k)$ such that
\begin{enumerate}
\item\label{item1-prop:A}
$\rho_\gamma(H)=t_\gamma H t_\gamma^{-1}$ for all $\gamma\in\Gamma$; and
\item\label{item2-prop:A}
$t_{\gamma_1 \gamma_2}\in \rho_{\gamma_1}(t_{\gamma_2}) t_{\gamma_1} H$ for all $\gamma_1,\gamma_2 \in\Gamma$. 
\end{enumerate}
If \ref{item1-prop:A}-\ref{item2-prop:A} are satisfied, then a $(G,\rho)$-equivariant descent datum on $X$ is given by 
\begin{equation*}
\mu\colon \Gamma \to \Aut(X),\ \gamma \mapsto (gH \mapsto \rho_\gamma(g)t_\gamma H).
\end{equation*}
Moreover, if $\rho_1$ and $\rho_2$ are two strongly equivalent descent data on $G$, with corresponding $k$-forms $F_1$ and $F_2$, then there is a bijection between the isomorphism classes of $(k,F_1)$-forms and of $(k,F_2)$-forms of $X$. 
\end{proposition}

\begin{remark}
If $\rho_1$ and $\rho_2$ are two equivalent descent data on $G$, which are not strongly equivalent, with corresponding $k$-forms $F_1$ and $F_2$, then the existence of a $(k,F_1)$-form for $X$ does not imply the existence of a $(k,F_2)$-form for $X$; see Example~\ref{ex:counter-ex not strongly eq}. 
\end{remark}

\begin{theorem}\label{th:B} \emph{(\S~\ref{sec:Existence of forms})}
We keep the notation of Proposition~\ref{prop:A}.
Let $\rho_c$ be an inner twist of $\rho$, and  let $F_c$ be the corresponding $k$-form of $G$.
We assume that the homogeneous space $X=G/H$ admits a $(k,F)$-form. 
\begin{enumerate}[leftmargin=6mm]
\item\label{item part 1 th B} The homogeneous space $X$ admits a $(k,F_c)$-form if and only if the cohomology class $\Delta_H([\rho_c]) \in \HH^2(\Gamma,\Aut_{\k}^{G}(X))$ is neutral (see \S~\ref{sec:cohomological criterion} for details). 
In particular, if $Z(G) \subseteq H$ or $H=N_G(H)$, then $X$ admits a $(k,F_c)$-form.
\item\label{item part 2 th B} If $X$ admits a $(k,F_c)$-form
and either $\Aut_{\k}^{G}(X)$ is abelian or $Z(G) \subseteq H$, then there is a bijection between the isomorphism classes of $(k,F)$-forms and of $(k,F_c)$-forms of $X$.
\end{enumerate}
\end{theorem}

\begin{remark}
With the notation of Theorem \ref{th:B}, there are homogeneous spaces $G/H$ admitting a $(k,F)$-form and a $(k,F_c)$-form for which the numbers of isomorphism classes of $(k,F)$-forms and of $(k,F_c)$-forms are distinct; see Remark~\ref{rk:rk sec 3.2}.
This shows that the condition that the automorphism group $\Aut_{\k}^{G}(G/H)$ is abelian or that  $Z(G) \subseteq H$, in the second part of Theorem~\ref{th:B}, is necessary. 
\end{remark}

\subsubsection*{Luna-Vust theory over perfect fields \emph{(\S~\ref{sec:LV over perfect fields})}}
From now on, we assume that the connected linear algebraic group $G$ over $\k$ is reductive. 
In \S~\ref{sec:LV over perfect fields}, which is the most technical part of the article, we are mostly concerned with the Luna-Vust theory whose major goal is to classify with combinatorial objects, in a sense that is precised in \S~\ref{sec:LV over alg closed fields}, the (normal) $G$-varieties that are $G$-equivariantly birational to a given (normal) $G$-variety $X_0$. 
In particular, when $X_0=G/H$ is a homogeneous space, Luna-Vust theory provides a combinatorial description of all the $G$-equivariant embeddings $X_0 \hookrightarrow X$ of $X_0$ (up to $G$-equivariant automorphism of $X$). 
It was established by Luna-Vust in \cite{LV83} for almost homogeneous $G$-varieties (case where $X_0=G/H$) over an algebraically closed field of characteristic zero, and then later extended by Timashev in \cite{Tim97,Tim11}  for arbitrary (normal) $G$-varieties over any algebraically closed field (see \S~\ref{sec:LV over alg closed fields} for a brief summary of Luna-Vust theory in this setting). 

In this article, inspired by the work of Huruguen in \cite{Hur11} on toric and spherical varieties, we extend the Luna-Vust theory over an arbitrary perfect base field $k$. As in Huruguen's case for spherical varieties, this is done by defining an action of the Galois group on the combinatorial Luna-Vust data; see \S\S~\ref{sec:Gamma action on colored data}-\ref{sec:k-forms for models} for details. 
The main result in \S~\ref{sec:LV over perfect fields} (Theorem \ref{th:C}) specializes for certain families of well-known $G$-varieties and allows us to recover some known results obtained by Huruguen in \cite{Hur11} and Wedhorn in \cite{Wed18} (for toric and spherical varieties) and by Langlois in \cite{Lan15} (for affine complexity-one $T$-varieties); see \S~\ref{sec:caseof complexity one} for a more detailed account. 

Our principal interest in generalizing the Luna-Vust theory from algebraically closed fields to arbitrary perfect fields comes from the observation that, combined with Proposition~\ref{prop:A} and Theorem~\ref{th:B}, it provides a strategy to determine the $(k,F)$-forms of a given almost homogeneous $G$-variety (as before $F$ denotes a fixed $k$-form of $G$); this strategy is detailed in \S~\ref{sec:strategy} and then applied in \S~\ref{sec:real forms of SL2-threefolds} to determine the real forms of complex almost homogeneous $\SL_2$-threefolds.

\subsubsection*{Real forms of complex almost homogeneous $\SL_2$-threefolds \emph{(\S~\ref{sec:real forms of SL2-threefolds})}}
In \S~\ref{sec:real forms of SL2-threefolds} we fix $k=\R$ and focus on the case of the complex threefolds on which $G=\SL_2(\C)$ acts with a dense open orbit. Let us emphasize that such varieties are never spherical.

First note that, since $\Gamma=\Gal(\C/\R) \simeq \Z/2\Z$, a descent datum $\rho$ on $G$ corresponds to a \emph{real group structure} on $G$, i.e.~an antiregular group involution $\sigma\colon G \to G$. Also, any real group structure on $G$ is strongly equivalent either to $\sigma_s\colon g \mapsto \modH{g}$ or to $\sigma_c\colon g \mapsto {}^{t}\modH{g}^{-1}$ (here $\modH{g}$ indicates the complex conjugate of $g$), with corresponding real loci $\SL_2(\R)$ and $\SU_2(\C)$ respectively.
Similarly, a $(G,\rho)$-equivariant descent datum on a $G$-variety $X$ corresponds to a \emph{$(G,\sigma)$-equivariant real structure} on $X$, i.e.~an antiregular  involution $\mu\colon X \to X$ such that 
\[
\forall g \in G(\C), \forall x \in X(\C),\ \mu(g \cdot x)=\sigma(g) \cdot \mu(x).
\]
Moreover, two $(G,\sigma)$-equivariant real structures on $X$ are called \emph{equivalent} if they are conjugate by a $G$-equivariant automorphism of $X$, in which case the corresponding $(\R,F)$-forms of $X$ are isomorphic as real $F$-varieties. (Here $F$ denotes the real form of $G=\SL_2$ corresponding to $\sigma$.) See \S~\ref{sec:real forms SL2} for details.

If $X$ is an almost homogeneous $G$-threefold, then it contains a dense open orbit isomorphic to $G/H$, with $H \subseteq G$ a finite subgroup, and these are well-known (see Remark~\ref{prop:finite subgroups of SL2}): there are the cyclic groups of order $n$ (conjugate to $A_n$), the binary dihedral groups of order $4n-8$ (conjugate to $D_n$ with $n \geq 4$), and the binary polyhedral groups (conjugate to $E_n$ with $n \in \{6,7,8\}$).

Our first main result in \S~\ref{sec:real forms of SL2-threefolds} is a complete description of the $(\R,F)$-forms for homogeneous $\SL_2$-threefolds.
  
\begin{theorem}\label{th:D}\emph{(\S~\ref{sec:SL2 homogeneous case})}
Let $H$ be a finite subgroup of $G=\SL_2(\C)$, and let $\sigma$ be a real group structure on $G$ corresponding to the real form $F$ or $G$.
Then the set of equivalence classes of the $(G,\sigma)$-equivariant real structures $\mu$ on $X=G/H$, which is in bijection with the set of isomorphism classes of the $(\R,F)$-forms of $X$ through the map $X \mapsto X/ \left\langle \mu \right \rangle$, and their real loci $X(\C)^\mu$, are given in Appendix~\ref{ap:table2}.
\end{theorem}

We then move on to the almost homogeneous case and give a combinatorial criterion for a $(G,\sigma)$-equivariant real structure on the open orbit $X_0=G/H$ to extend to a given $G$-equivariant embedding $X_0 \hookrightarrow X$ (see Theorem \ref{th:E}); this is done by specializing the Luna-Vust theory over perfect fields (developed in \S~\ref{sec:LV over perfect fields}) in this particular setting.
Several examples where Theorem~\ref{th:E} is applied to determine the real forms of certain almost homogeneous $\SL_2$-threefolds can be found in \S~\ref{sec:almost homogeneous SL2 threefolds}; they illustrate various pathological behaviors (see Remarks \ref{rk:spherical behavior} and \ref{rk:Gamma map}), some of which cannot occur in the spherical case and motivate a systematic study  of (normal) varieties of complexity $\geq 1$ (see Definition~\ref{def:complexity}).

Our last main result, which is mostly a byproduct of Theorems~\ref{th:D} and ~\ref{th:E}, concerns the real forms of the minimal smooth completions of $X_0=G/H$ when $H$ is non-cyclic (see Remark~\ref{rk:cyclic case} for a comment on the cyclic case). 
Here we call \emph{minimal smooth completion} of $X_0$ a smooth complete almost homogeneous $G$-threefold $X$, with open orbit $X_0$, such that any $G$-equivariant birational morphism $X \to X'$, with $X'$ smooth, is an isomorphism. (The list of the minimal smooth completions of $X_0=G/H$ when $H$ is non-cyclic is given by Propositions \ref{prop:embeddings in type E} and \ref{prop:embeddings in type D}; it includes the Fano threefolds $\P^3$, $Q_3$, $V_5$, and $V_{22}^{MU}$.)

\begin{corollary}\label{cor:F}\emph{(\S~\ref{sec:minimal models})}
Let $H$ be a non-cyclic finite subgroup of $G=\SL_2(\C)$, and let $\sigma$ be a real group structure on $G$ corresponding to a real form $F$ of $G$. 
\begin{itemize}[leftmargin=5mm]
\item If $H$ is conjugate to $D_4$, then the unique minimal smooth completion $X$ of $X_0=G/H$ admits exactly two inequivalent $(G,\sigma)$-equivariant real structures.
\item If $H$ is conjugate to $E_6$, $E_7$, $E_8$, or $D_n$ with $n \geq 5$, then  the unique minimal smooth completion $X$ of $X_0=G/H$ admits exactly one $(G,\sigma)$-equivariant real structure.
\end{itemize}
Hence, there is always a unique $(\R,F)$-form for $X$, except when $H$ is conjugate to $D_4$, in which case there are exactly two non-isomorphic $(\R,F)$-forms for $X$. 
\end{corollary}

\begin{remark}
In the setting of Corollary~\ref{cor:F}, a table with the list of the $(G,\sigma)$-equivariant real structures on $X_0=G/H$ that extend to the minimal smooth completions $X$ of $X_0$ can be found at the end of \S~\ref{sec:minimal models}.
\end{remark}

\subsection{Foreword}
The three main sections (\S\S~\ref{sec:forms of qh varieties2}-\ref{sec:LV over perfect fields}-\ref{sec:real forms of SL2-threefolds}) are largely  independent, yet closely interconnected. 
Within sections~\ref{sec:forms of qh varieties2} and \ref{sec:LV over perfect fields},  we ensure the accessibility of proofs by providing two self-contained presentations on each respective topic.

\subsection*{Acknowledgments}
We would like to thank Michel Brion, Lucas Moulin, and Pierre-Alexandre Gillard for their comments on a former version of this article.
We are also grateful to the referee for his/her careful reading of the paper and for his/her comments and suggestions which helped us to improve the manuscript.

\section{Preliminaries} \label{sec: preliminaries}

\subsection{Notation}
Let $k$ be a perfect field. We denote by $\k$ a fixed algebraic closure of $k$, and by $\Gamma=\Gal(\k/k)$ the absolute Galois group of $k$. Since $k$ is perfect, the field extension $k \hookrightarrow \k$ is Galois and $\Gamma$ is a profinite group endowed with the \emph{Krull topology}; see  \cite[\S~I]{Ber10} for a detailed account on Galois extensions and Galois groups. An abstract (abelian) group, endowed with the discrete topology, on which $\Gamma$ acts continuously is called an (\emph{abelian}) \emph{$\Gamma$-group}.

In this article an \emph{algebraic variety} (over $k$) is a separated scheme of finite type (over $k$) which is geometrically integral and \textbf{normal}.
An \emph{algebraic group} (over $k$) is a finite type group scheme (over $k$) which is smooth. 
By an \emph{algebraic subgroup}, we always mean a closed and reduced subgroup scheme. In particular, we only consider homogeneous spaces with reduced stabilizers. 

A reductive algebraic group $F$ is always assumed to be connected and of \emph{simply-connected type}, i.e.~$\Fk=F \times_{\Spec(k)} \Spec(\k)$ is isomorphic to a product of a torus and a simply-connected semisimple group. We will always denote by $F$ an algebraic group over $k$ and by $G$ an algebraic group over $\k$. Also, we will denote by $Z(G)$ the center of $G$ and, when $H \subseteq G$ is an algebraic subgroup, by $N_G(H)$ the normalizer of $H$ in $G$.

When we write "the homogeneous space $G/H$" we always implicitly refer to a pair $(X_0,x_0)$, where $X_0$ is a homogeneous $G$-variety over $\k$ and $x_0 \in X_0(\k)$ satisfies $\Stab_G(x_0)=H$ as a  subgroup scheme of $G$. An \emph{equivariant embedding} of $G/H$ is a pair $(X,x)$ formed by a $G$-variety $X$ over $\k$ and $x \in X(\k)$ such that $\Stab_x(G)=H$, as a subgroup scheme of $G$, and the $G$-orbit of $x$ is a dense open subset of $X$. 
Two equivariant embeddings $(X_1,x_1)$ and $(X_2,x_2)$ of $G/H$ are said to be \emph{isomorphic} if there exists a $G$-equivariant isomorphism $\psi \colon X_1 \to X_2$ such that $\psi(x_1)=x_2$. 
To simplify the notation, we will only write $X$ instead of $(X,x)$ to denote an equivariant embedding of $G/H$, except when the $\k$-point $x$ plays an important role.  

We refer to \cite{Mil17} for the background concerning algebraic groups and algebraic varieties endowed with algebraic group actions.

\subsection{Forms, descent data, and inner twists for algebraic groups} \label{sec:Forms of algebraic groups and inner twists}

Let $k$ be a perfect field, let $G$ be an arbitrary algebraic group over $\k$, and let $H \subseteq G$ be an algebraic subgroup.  
In \S\S~\ref{sec:Forms of algebraic groups and inner twists}-\ref{sec:descent data for algebraic varieties with group action} we recall the basic notions of forms and descent data for algebraic groups and algebraic varieties with algebraic group action. 

\begin{definition}\label{def:k-forms of algebraic groups}\item
\begin{itemize}[leftmargin=6.5mm]
\item A \emph{$k$-form} of the algebraic group $G$ is an algebraic group $F$ over $k$ together with an isomorphism $F_{\k}=F \times_{\Spec(k)} \Spec(\k) \simeq G$ of algebraic groups (over $\k$).
\item A \emph{descent datum} on $G$ is an algebraic semilinear action $\rho\colon \Gamma \to \Aut(G)$ preserving the algebraic group structure of $G$, which means that 
\begin{enumerate}[(i)]
\item there exists a finite Galois extension $k'/k$ in $\k$ and a $k'$-form $F'$ of the $\k$-variety $G$ such that the restriction of $\rho$ to $\Gal(\k/k')$ coincides with the natural $\Gal(\k/k')$-action on $G \simeq F'\times_{\Spec(k')} \Spec(\k)$; 
\item 
for each $\gamma \in \Gamma$, we have a commutative diagram 
\[
\xymatrix@R=4mm@C=2cm{
    G \ar[rr]^{\rho_\gamma} \ar[d]  && G \ar[d] \\
    \Spec(\k) \ar[rr]^{(\gamma^*)^{-1}} && \Spec(\k)
  }
  \]
where $\rho_\gamma$ is a scheme automorphism over $\Spec(k)$; and 
\item for each $\gamma \in \Gamma$, $\rho_\gamma$ is compatible with the algebraic group structure of $G$, i.e. 
\[\iota_G \circ \rho_\gamma=\rho_\gamma \circ \iota_G \ \ \text{ and }\ \ m_G \circ (\rho_\gamma \times \rho_\gamma)=\rho_\gamma \circ m_G,\] where $\iota_G\colon G \to G$ is the inverse morphism and $m_G\colon G \times G \to G$ is the multiplication morphism.
\end{enumerate}
\item Two descent data $\rho_1$ and $\rho_2$ on $G$ are \emph{equivalent} if there exists $\psi \in \Aut_{\k}(G)$ such that
\[\forall \gamma \in \Gamma,\ \rho_{2,\gamma}=\psi \circ \rho_{1,\gamma} \circ \psi^{-1}.\]
They are \emph{strongly equivalent} if we can take $\psi \in \Inn_{\k}(G)$ (=the subgroup of inner  automorphisms). If $h\in G(\k)$, we denote by $\inn_h$ the inner automorphism of $G$ defined by $\inn_h\colon G \to G,\ g\mapsto hgh^{-1}$. 
\end{itemize}
\end{definition}

\begin{remark}
The reason why the finite Galois extension $k'/k$ appears in Definition~\ref{def:k-forms of algebraic groups} is to ensure the existence of the categorical quotient $G/\Gamma$. 
Indeed, $\Gamma=\Gal(\overline{k}/k)$ is a (possibly infinite) discrete group while $\Gal(k'/k)$ is always finite, and so the categorical quotient $G/\Gamma \simeq F'/\Gal(k'/k)$ is well-defined. (The same remark holds also for Definition~\ref{def:k-forms of G-varieties} below.)
\end{remark}

If $F$ is a $k$-form of $G$, then the homomorphism $\Gamma \to \Aut(F_{\k}),\ \gamma \to Id \times (\gamma^*)^{-1}$ gives a descent datum on $G \simeq F_{\k}$.
Conversely, if $\Gamma \to \Aut(G)$ is a descent datum on $G$, then the quotient space $F:=G/\Gamma \simeq F'/\Gal(k'/k)$ is a $k$-form of $G$; an isomorphism $G \xrightarrow{\sim} \Fk$  is given by $(q,f)$, where $q\colon G\to F$ is the quotient morphism and $f\colon G \to \Spec(\k)$ is the structure morphism.

This induces a correspondence between the $k$-forms of $G$, up to isomorphism in the category of algebraic groups over $k$, and equivalence classes of descent data on $G$. (See also \cite[\S~1.4]{FSS98} or \cite[\S~2.5]{BG21}.)

\begin{definition}\label{def:inner twist}
Let $\rho\colon \Gamma \to \Aut(G)$ be a descent datum, and let $c\colon\Gamma \to G(\k)$ be a locally constant map. If the map 
\[ \rho_c \colon \Gamma \to \Aut(G),\ \gamma \mapsto (g \mapsto \inn_{c_\gamma} \circ \rho_\gamma(g)=c_\gamma \rho_\gamma(g)c_{\gamma}^{-1})\]
is a descent datum on $G$, then it is called an \emph{inner twist} of $\rho$.
\end{definition}

We fix a descent datum $\rho$ on $G$; it induces a descent datum on $Z(G)$ and $G/Z(G)$. We denote by $Z^1(\Gamma,(G/Z(G))(\k))$ the set of $1$-cocycles with values in the $\Gamma$-group $(G/Z(G))(\k)$ and by $\HH^1(\Gamma,(G/Z(G))(\k))$ the corresponding first cohomology set; see \cite[\S~I\!I.3]{Ber10} for the definitions of these two pointed sets.

\begin{lemma}\label{lem:Z1 parametrizes inner twists}
Let $\rho\colon \Gamma \to \Aut(G)$ be a descent datum, and let $c\colon\Gamma \to G(\k)$ be a locally constant map. Then the following hold:
\begin{enumerate}
\item\label{item:Z1 i} The map $\rho_c$ is a descent datum on $G$ if and only if the locally constant map $\overline{c}\colon \Gamma \to (G/Z(G))(\k),\ \gamma \mapsto \overline{c_\gamma}$ is an element of $Z^1(\Gamma,(G/Z(G))(\k))$.
\item\label{item:Z1 ii} The descent data $\rho_{c_1}$ and $\rho_{c_2}$ are equal if and only if $\overline{c_1}=\overline{c_2}$.
\item\label{item:Z1 iii} The descent data $\rho_{c_1}$ and $\rho_{c_2}$ are strongly equivalent (Definition~\ref{def:k-forms of algebraic groups}) if and only if $\overline{c_1}$ and $\overline{c_2}$ are cohomologous. 
\end{enumerate}
\end{lemma}

\begin{proof}
\ref{item:Z1 i}: Let us note that the map $\rho_c$ is a descent datum if and only if it is a homomorphism, i.e.
\[ \forall  \gamma_1,\gamma_2 \in \Gamma,\  \rho_{c,\gamma_1\gamma_2} = \rho_{c,\gamma_1} \circ  \rho_{c,\gamma_2}.\]
Developing and simplifying these equalities, we observe that it is equivalent to the cocycle condition for $\overline{c}$
\\
\ref{item:Z1 ii}: This follows from considering the equalities $\rho_{c_1,\gamma}=\rho_{c_2,\gamma}$ for all $\gamma \in \Gamma$.\\
\ref{item:Z1 iii}: The two $1$-cocycles $\overline{c_1}$ and $\overline{c_2}$ are cohomologous 
if there exists $g \in G(\k)$ such that
\[ \forall \gamma \in \Gamma, c_{2,\gamma} \rho_\gamma(g) c_{1,\gamma}^{-1} g^{-1} \in Z(G)(\k).\]
On the other hand, the descent data $\rho_{c_1}$ and $\rho_{c_2}$ are strongly equivalent is there exists $g \in G(\k)$ such that, for all $h \in G(\k)$ and all $\gamma \in \Gamma$, we have 
\begin{align*}
\rho_{c_2,\gamma} (ghg^{-1})=g \rho_{c_1,\gamma}(h) g^{-1}  &\Leftrightarrow  c_{2,\gamma} \rho_\gamma(g)  \rho_\gamma(h)   \rho_\gamma(g)^{-1} c_{2,\gamma}^{-1}=gc_{1,\gamma}  \rho_\gamma(h) c_{1,\gamma}^{-1} g^{-1}\\  
&\Leftrightarrow c_{1,\gamma}^{-1} g^{-1} c_{2,\gamma} \rho_\gamma(g) \in Z(G)(\k) \\
&\Leftrightarrow  c_{2,\gamma} \rho_\gamma(g) c_{1,\gamma}^{-1} g^{-1}\in Z(G)(\k). 
\end{align*}
This finishes the proof of the lemma.
\end{proof}

Therefore, by Lemma~\ref{lem:Z1 parametrizes inner twists}, the set $Z^1(\Gamma,(G/Z(G))(\k))$ identifies with the inner twists of $\rho$ and the set $\HH^1(\Gamma,(G/Z(G))(\k))$ identifies with the strong equivalence classes of inner twists of $\rho$. These identifications will be used to define the map $\Delta_H$ in \S~\ref{sec:cohomological criterion}.

\subsection{Forms and descent data for algebraic varieties with group action}\label{sec:descent data for algebraic varieties with group action}
We fix a descent datum $\rho$ on $G$, and we denote the corresponding $k$-form by $F=G/\Gamma$.

\begin{definition}\label{def:k-forms of G-varieties}\item
\begin{itemize}[leftmargin=6mm]
\item A \emph{$(k,F)$-form} \emph{of a $G$-variety $X$} (over $\k$) is an $F$-variety $Z$ (over $k$) together with an isomorphism $Z_{\k} \simeq X$ of $G$-varieties, where $G$ acts on $Z_{\k}$ through $\Fk \simeq G$.
\item A \emph{$(G,\rho)$-equivariant descent datum on $X$} is an algebraic semilinear action $\mu\colon \Gamma \to \Aut(X)$ compatible with $(G,\rho)$, which means that 
\begin{enumerate}[(i)]
\item there exists a finite Galois extension $k'/k$ in $\k$ and a $k'$-form $Z'$ of the $\k$-variety $X$ such that the restriction of $\mu$ to $\Gal(\k/k')$ coincides with the natural $\Gal(\k/k')$-action on $X \simeq Z'\times_{\Spec(k')} \Spec(\k)$; 
\item 
for each $\gamma \in \Gamma$, we have a commutative diagram
\[
\xymatrix@R=4mm@C=2cm{
    X \ar[rr]^{\mu_\gamma} \ar[d]  && X \ar[d] \\
    \Spec(\k) \ar[rr]^{(\gamma^*)^{-1}} && \Spec(\k)
  }
  \]
where $\mu_\gamma$ is a scheme automorphism over $\Spec(k)$; and
\item the following condition holds
\begin{equation}\label{eq:Gamma-eq}
\forall \gamma \in \Gamma, \forall g \in G(\k), \forall x \in X(\k),\ \mu_\gamma(g \cdot x)=\rho_\gamma(g) \cdot \mu_\gamma(x).
\end{equation}
\end{enumerate}
\item Two $(G,\rho)$-equivariant descent data $\mu_1$ and $\mu_2$ on $X$ are \emph{equivalent} if there exists $\varphi \in \Aut_{\k}^{G}(X)$ such that 
\[\forall \gamma \in \Gamma,\ \mu_{2,\gamma}=\varphi \circ \mu_{1,\gamma} \circ \varphi^{-1}.\]
\end{itemize}
\end{definition}

Let $X$ be a $G$-variety.
As for algebraic groups, there is a correspondence between isomorphism classes of $(k,F)$-forms of $X$ and equivalence classes of \emph{effective} $(G,\rho)$-equivariant descent data on $X$. (See \cite[\S~5]{Bor20} or \cite[\S~2]{BG21} for details.) 

Here the word \emph{effective} means that $X$ is covered by $\Gamma$-stable affine open subsets, i.e.~that the quotient $X/\Gamma \simeq Z'/\Gal(k'/k)$, which always exists as an algebraic space, is actually a variety (see \cite[Proposition~V.1.8]{SGA1}).

Let us note that if $X$ is quasiprojective or covered by $\Gamma$-stable quasiprojective open subsets, then $X$ is covered by $\Gamma$-stable affine open subsets (\cite[Lemma~2.4]{BG21}).
In particular, since homogeneous spaces under the action of a connected algebraic group are quasiprojective (this follows for instance from \cite[Theorem~1]{Bri17}), equivariant descent data on such varieties are always effective.

\smallskip

Once one knows the existence of a $(G,\rho)$-equivariant descent datum $\mu$ on $X$, one can use Galois cohomology to parametrize the equivalence classes of all of them. First, observe that $\Aut_{\k}^{G}(X)$ is endowed with a $\Gamma$-group structure as follows:
\begin{equation}\label{eq:def inn_mu}
\inn_{\mu} \colon \Gamma \times \Aut_{\k}^{G}(X) \to \Aut_{\k}^{G}(X) ,\ (\gamma,\varphi) \mapsto
  \mu_\gamma \circ \varphi \circ \mu_{\gamma}^{-1}.
\end{equation}
Then the next result is a straightforward consequence of the definition of cohomologous $1$-cocycles with values in the $\Gamma$-group $\Aut_{\k}^{G}(X)$.

\begin{proposition}\label{prop:H1 Galois k-forms}\emph{(\cite[Corollary~8.2]{Wed18})}
Let $X$ be a $G$-variety, and let $\mu$ be a $(G,\rho)$-equivariant descent datum on $X$.
There is a bijection of pointed sets
\[
\begin{array}{ccccc}
 & \HH^1(\Gamma,\Aut_{\k}^{G}(X)) & \xrightarrow{\sim} & \left\{\text{
\begin{tabular}{c} 
 equivalence classes of\\ $(G,\rho)$-equivariant descent data on $X$
\end{tabular} 
 }
 \right\} \\
 & [\Gamma \to \Aut_{\k}^{G}(X),\ \gamma \mapsto c_\gamma] & \mapsto & [\Gamma \to \Aut(X),\ \gamma \mapsto (x \mapsto c_\gamma \circ \mu_\gamma(x))]. \\
\end{array}
\]
\end{proposition}

\section{Forms of homogeneous spaces} \label{sec:forms of qh varieties2}
From now on, and until the end of the article, we will systematically assume that the algebraic group $G$ is linear and connected, and even reductive starting from \S~\ref{sec:LV over perfect fields}.
In \S~\ref{sec: finiteness number of forms}, we explain why homogeneous spaces have finitely many forms (over base fields of type $(F)$). Subsequently, in Section \ref{sec:cohomological criterion}, we recall the definition of the cohomological invariant $\Delta_H([\rho_c])$  that appears in the statement of Theorem~\ref{th:B}.
Finally, in \S~\ref{sec:Existence of forms} we prove Proposition~\ref{prop:A} and Theorem~\ref{th:B}.

\subsection{Finiteness of the number of forms for homogeneous spaces}\label{sec: finiteness number of forms}
We recall the following classical result that describes $\Aut_{\k}^{G}(X)$ in the case where $X$ is a homogeneous space. (We will use it throughout the rest of the article.)

\begin{proposition}\label{prop:G-eq automorphisms of G/H} \emph{(\cite[Proposition~1.8]{Tim11})} 
Assume that $G$ is linear and let $H \subseteq G$ be an algebraic subgroup. 
Then there is an isomorphism of groups
\[
(N_G(H)/H)(\k) \xrightarrow{\sim} \Aut_{\k}^{G}(G/H),\ nH \mapsto (\varphi_{nH}\colon gH \mapsto gn^{-1}H).
\]
\end{proposition}

It is a difficult problem in general to determine whether a given $G$-variety admits a finite number  of isomorphism classes of $(k,F)$-forms. However, in the homogeneous case we have the following well-known result.

\begin{corollary} \label{cor:finiteness form homogeneous spaces}
Assume that for every $n \in \N$, there exist only a finite number of subextensions of $\k$ which are of degree $n$ over $k$. 
If $G$ is linear and $X \simeq G/H$ is a homogeneous space, then $X$ admits a finite number of isomorphism classes of $(k,F)$-forms.
\end{corollary}

\begin{proof}
It suffices to prove that the number of  equivalence classes of $(G,\rho)$-equiva\-riant descent data on $X$ is finite. 

Assume that there exists at least one $(G,\rho)$-equivariant descent datum $\mu$ on $X$, otherwise there is nothing to prove. 
Then, by Proposition~\ref{prop:H1 Galois k-forms}, the number of equivalence classes of $(G,\rho)$-equivariant descent data on $X$ is equal to the cardinal of the set $\HH^1(\Gamma,\Aut_{\k}^{G}(X))$,  where $\Gamma$ acts on $\Aut_{\k}^{G}(X)$ by $\mu$-conjugation.

Since the algebraic group $G$ is linear, the group $\Aut_{\k}^{G}(X)\simeq (N_G(H)/H)(\k) $ is also linear. 
Then \cite[\S~6.2. Th\'eor\`eme]{BS64} yields the finiteness of $\HH^1(\Gamma,\Aut_{\k}^{G}(X))$, which finishes the proof.
\end{proof}

\subsection{A cohomological invariant}\label{sec:cohomological criterion}
In this subsection, we present a cohomological invariant introduced by Borovoi and Gagliardi in \cite{BG21}, and also used in the real setting in \cite{MJT21}.

From now on we assume that the algebraic group $G$ is linear and connected. 
As before we fix a descent datum $\rho$ on $G$ and we consider the short exact sequence of $\Gamma$-groups
\[1 \to Z(G)(\k) \to G(\k) \to (G/Z(G))(\k) \to 1.\]
There is a connecting homomorphism
\small
\[\begin{array}{cccc}
\delta \colon & \HH^1(\Gamma,(G/Z(G))(\k)) & \to & \HH^2(\Gamma,Z(G)(\k)) \\
  & [\overline{c}\colon\Gamma \to (G/Z(G))(\k),\, \gamma \mapsto \overline{c_\gamma}] & \mapsto & [\Gamma^2 \to Z(G)(\k),\, (\gamma_1,\gamma_2) \mapsto c_{\gamma_1} \rho_{\gamma_1}(c_{\gamma_2}) c_{\gamma_1 \gamma_2}^{-1})] 
\end{array}\]
\normalsize
where $c\colon\Gamma \to G(\k)$ is a locally constant lift of $\overline{c}$ satisfying $c_e=Id_G$ and $\HH^2(\Gamma,Z(G)(\k))$ is the second cohomology group (and not just a pointed set since $Z(G)(\k)$ is an abelian $\Gamma$-group); see \cite[\S~I\!I]{Ber10} for the definition of this cohomology group and of the connecting homomorphism.

\smallskip

Let $H \subseteq G$ be an algebraic subgroup, and assume that the homogeneous space $X=G/H$ admits a $(G,\rho)$-equivariant descent datum $\mu$. Then $\Aut_{\k}^{G}(X)$ is a $\Gamma$-group (see \eqref{eq:def inn_mu} for the definition of the $\Gamma$-action induced by $\mu$), and the homomorphism $\kappa$ obtained by composing the following homomorphisms
\begin{equation}\label{eq:def of kappa}
\kappa\colon Z(G)(\k) \hookrightarrow N_G(H)(\k) \twoheadrightarrow (N_G(H)/H)(\k) \xrightarrow{\sim} \Aut_{\k}^{G}(X),\ z \mapsto (x \mapsto z^{-1} \cdot x) 
\end{equation}
is $\Gamma$-equivariant. It induces a map between second pointed sets of Galois cohomology  
\[ \lambda_H\colon \HH^2(\Gamma,Z(G)(\k)) \to \HH^2(\Gamma, \Aut_{\k}^{G}(X)),\  [(\rho,\alpha)] \mapsto [(\inn_{\mu},\kappa \circ \alpha)].\] 
We refer to \cite[\S~1.5]{Bor93} for the definition of the second nonabelian Galois cohomology set. It is also defined in  \cite[\S~1.14]{Spr66}. (Note, however, that in \textit{loc.~cit.~}the convention  differs  slightly from ours. More precisely, a  $2$-cocycle $(\tau,\beta)$ in the present article and in \cite{Bor93} corresponds to a $2$-cocycle $(\tau,\beta^{-1})$ in \cite{Spr66}.)
We denote by
\begin{equation*}
\Delta_H\colon \HH^1(\Gamma,(G/Z(G))(\k))  \to \HH^2(\Gamma, \Aut_{\k}^{G}(X))
\end{equation*}
the map obtained by composing $\delta$ and $\lambda_H$. 

The case where $\Aut_{\k}^{G}(X)$ is an abelian group is of special interest (it is the case, for example, when $X$ is a spherical variety). Indeed, in this case, $\HH^2(\Gamma, \Aut_{\k}^{G}(X))$ is the classical second Galois cohomology group, which is an abelian group, and $\Delta_H$ is a group homomorphism.
The neutral element is called \emph{neutral cohomology class} in $\HH^2(\Gamma, \Aut_{\k}^{G}(X))$. If $\Aut_{\k}^{G}(X)$ is nonabelian, the definition of a neutral cohomology class is the following:

\begin{definition}(\cite[\S~1.6]{Bor93})
The class of a $2$-cocycle $[(\tau,\beta)] \in \HH^2(\Gamma, \Aut_{\k}^{G}(X))$ is \emph{neutral} if there exists a locally constant map $d\colon \Gamma \to \Aut_{\k}^{G}(X)$ such that  
\[ \forall \gamma_1,\gamma_2 \in \Gamma,\ 
d_{\gamma_1 \gamma_2} \circ \beta_{\gamma_1,\gamma_2} \circ \tau_{\gamma_1} (d_{\gamma_2})^{-1} \circ d_{\gamma_1}^{-1}=Id_{X}.\]
\end{definition}

In particular, identifying $\HH^1(\Gamma,(G/Z(G))(\k))$ with the strong equivalence classes of inner twists of $\rho$ (see \S~\ref{sec:Forms of algebraic groups and inner twists}) and using $\Aut_{\k}^{G}(X)\simeq (N_G(H)/H)(\k)$ (see Proposition~\ref{prop:G-eq automorphisms of G/H}), we check that the cohomology class $\Delta_H([\rho_c])$ is neutral if and only if there exists a locally constant map $n\colon \Gamma \to N_G(H)(\k)$ such that  
\begin{equation}\label{eq:class is neutral} 
\forall \gamma_1,\gamma_2 \in \Gamma,\ \varphi_{n_{\gamma_1 \gamma_2}H} \circ \varphi_{c_{\gamma_1} \rho_{\gamma_1}(c_{\gamma_2}) c_{\gamma_1 \gamma_2}^{-1}H} \circ 
\mu_{\gamma_1} \circ \varphi_{n_{\gamma_2}H}^{-1} \circ \mu_{\gamma_1}^{-1} \circ \varphi_{n_{\gamma_1}H}^{-1}=Id_X.
\end{equation}

The cohomology class $\Delta_H([\rho_c])$ is the \emph{cohomological invariant} to which the title of the subsection refers. Let us note that, in the nonabelian case, the subset of neutral cohomology classes in $\HH^2(\Gamma, \Aut_{\k}^{G}(X))$ may be empty or have more than one element.

\subsection{Existence of forms for homogeneous spaces}\label{sec:Existence of forms}
In this subsection we prove Proposition~\ref{prop:A} and Theorem~\ref{th:B}. 

Recall that we denote by $G$ a connected linear algebraic group over $\k$, by $H \subseteq G$ an algebraic subgroup, by $\rho$ a descent datum on $G$, and by $F=G/\Gamma$ (where $\Gamma$ acts on $G$ through $\rho$) the $k$-form of $G$ corresponding to $\rho$.

\begin{proposition} \label{prop:existence of (k,F)-form}
The homogeneous space $X=G/H$ admits a $(k,F)$-form if and only if there exists a locally constant map $t\colon \Gamma \to G(\k)$ such that
\begin{enumerate}
\item\label{item:condition i} 
$\rho_\gamma(H)=t_\gamma H t_\gamma^{-1}$ for all $\gamma\in\Gamma$; and
\item\label{item:condition ii} 
$t_{\gamma_1 \gamma_2}\in \rho_{\gamma_1}(t_{\gamma_2}) t_{\gamma_1} H$ for all $\gamma_1,\gamma_2 \in\Gamma$. 
\end{enumerate}
If \ref{item1-prop:A}-\ref{item2-prop:A} are satisfied, then a $(G,\rho)$-equivariant descent datum on $X$ is given by 
\begin{equation}\label{eq:descent datum on X}
\mu\colon \Gamma  \to \Aut(X),\ \gamma \mapsto (gH \mapsto \rho_\gamma(g)t_\gamma H).
\end{equation}
Moreover, two locally constant maps $t_1,t_2 \colon \Gamma \to G(\k)$ as above correspond to the same $(k,F)$-form if and only if  $t_{2,\gamma} \in t_{1,\gamma}H(\k)$ for all $\gamma \in \Gamma$, and they correspond to isomorphic $(k,F)$-forms if and only if there exists $n \in N_G(H)(\k)$ such that $t_{2,\gamma} \in \rho_\gamma(n) t_{1,\gamma} n^{-1}H(\k)$ for all $\gamma \in \Gamma$.
\end{proposition}

\begin{proof}
Since $X=G/H$ is a quasiprojective variety, it admits a $(k,F)$-form if and only if it admits a $(G,\rho)$-equivariant descent datum.

Assume that $X$ admits a $(G,\rho)$-equivariant descent datum $\mu$. Then $\mu$ is determined by the collection of $\mu_\gamma(eH) \in (G/H)(\k)$ for every $\gamma \in \Gamma$. Let $t\colon \Gamma \to G(\k)$ be a locally constant map such that $\mu_\gamma(eH)=t_\gamma H$ (the existence of such a map follows from the \emph{algebraicity} assumption in Definition~\ref{def:k-forms of G-varieties}). Then property \eqref{eq:Gamma-eq} in Definition~\ref{def:k-forms of G-varieties} implies condition~\ref{item:condition i} since 
\[ \forall \gamma \in \Gamma,\ t_\gamma H=\mu_\gamma(eH)=\mu_\gamma(H \cdot eH)=\rho_\gamma(H) \cdot \mu_\gamma(eH)=\rho_\gamma(H) t_\gamma H.  \]
Moreover, $\gamma \mapsto \mu_\gamma$ being a homomorphism yields condition~\ref{item:condition ii} since
\[\forall \gamma_1,\gamma_2 \in \Gamma, \ t_{\gamma_1 \gamma_2}H=\mu_{\gamma_1 \gamma_2}(eH)=\mu_{\gamma_1}(\mu_{\gamma_2}(eH))=\mu_{\gamma_1}(t_{\gamma_2}H)=\rho_{\gamma_1}(t_{\gamma_2}) t_{\gamma_1} H.\]
Then, for all $\gamma \in \Gamma$ and all $gH \in (G/H)(\k)$, we have 
\[ 
\mu_\gamma(gH)=\rho_\gamma(g) \mu_\gamma(eH)=\rho_\gamma(g) t_\gamma H.
\]

Conversely, if there exists a locally constant map $t\colon \Gamma \to G(\k)$ satisfying conditions~\ref{item:condition i} and~\ref{item:condition ii}, then the map defined by~\eqref{eq:descent datum on X} is a $(G,\rho)$-equivariant descent datum on $X$. More precisely, Condition~\ref{item:condition i} ensures that for all $g \in G(\k)$ and all $h \in H(\k)$, we have $\mu_\gamma(gH)=\mu_\gamma(ghH)$, i.e.~that the map $\mu$ is well-defined, and Condition~\ref{item:condition ii} ensures that $\gamma \mapsto \mu_\gamma$ is a homomorphism.

Finally, using~\eqref{eq:descent datum on X} and Proposition~\ref{prop:G-eq automorphisms of G/H}, the last sentence of the proposition is an easy consequence of the fact that isomorphic $(k,F)$-forms correspond to equivalent descent data. We omit here the details of this part of the proof.
\end{proof}

\begin{remark}\label{rk:choice of the base point}
The existence of a $(k,F)$-form of $X=G/H$ does not depend on the choice of a base point of $X$. 
Indeed, if the two conditions in Proposition~\ref{prop:existence of (k,F)-form} hold and $H'=sHs^{-1}$ for some $s \in G(\k)$, then  a $(G,\rho)$-equivariant descent datum on $G/H'$ is given by
$\mu'\colon \Gamma \to \Aut(G/H'),\ \gamma \mapsto (gH' \mapsto \rho_\gamma(gs)t_\gamma s^{-1}H')$.
\end{remark}

\begin{example}
Let $X=G/P$ be a generalized flag variety. Then, using Proposition~\ref{prop:existence of (k,F)-form} and the fact that $N_G(P)=P$, one can check that $X$ admits a $(k,F)$-form if and only if there exists a locally constant map $t\colon \Gamma \to G(\k)$ such that $\rho_\gamma(H)=t_\gamma H t_\gamma^{-1}$ for all $\gamma\in\Gamma$.
\end{example}

\begin{corollary}\label{cor:existence of forms for strongly eq rho}
Let $\rho_1$ and $\rho_2$ be two strongly equivalent descent data on $G$, with corresponding $k$-forms $F_1$ and $F_2$, and let $X=G/H$ be a homogeneous space.
Then there is a bijection between the (isomorphism classes of) $(k,F_1)$-forms and the (isomorphism classes of) $(k,F_2)$-forms of $X$.  
\end{corollary}

\begin{proof}
Let $g \in G(\k)$ such that $\rho_2=\inn_g \circ \rho_1 \circ \inn_{g}^{-1}$.
By Proposition~\ref{prop:existence of (k,F)-form}, a $(k,F_1)$-form on $X$ corresponds to a locally constant map $t_1\colon\Gamma \to G(\k)$ such that conditions~\ref{item:condition i} and ~\ref{item:condition ii} of  Proposition~\ref{prop:existence of (k,F)-form} are satisfied. 

We define a locally constant map $t_2\colon\Gamma \to G(\k),\ \gamma \mapsto g\rho_{1,\gamma}(g)^{-1}t_{1,\gamma}$. Then, for all $\gamma \in \Gamma$, we have 
\[
\rho_{2,\gamma}(H)=\inn_g \circ \rho_{1,\gamma} \circ \inn_{g}^{-1}(H)=g\rho_{1,\gamma}(g)^{-1}t_{1,\gamma}H t_{1,\gamma}^{-1}   \rho_{1,\gamma}(g) g^{-1}=t_{2,\gamma}Ht_{2,\gamma}^{-1}
\]
and, for all $\gamma_1,\gamma_2 \in \Gamma$, we have
\begin{align*}
t_{2,\gamma_1 \gamma_2}=g\rho_{1,\gamma_1 \gamma_2}(g)^{-1}t_{1,\gamma_1 \gamma_2} &\in g\rho_{1,\gamma_1 \gamma_2}(g)^{-1} \rho_{1,\gamma_1}(t_{1,\gamma_2}) t_{1,\gamma_1}H\\
&=g \rho_{1,\gamma_1}(\rho_{1,\gamma_2}(g)^{-1}t_{1,\gamma_2}) \rho_{1,\gamma_1}(g)g^{-1}t_{2,\gamma_1}H\\
&=\rho_{2,\gamma_1}(t_{2,\gamma_2}) t_{2,\gamma_1}H.
\end{align*} 
Hence, conditions~\ref{item:condition i} and ~\ref{item:condition ii} of  Proposition~\ref{prop:existence of (k,F)-form} hold for the locally constant map $t_2$, i.e.~it corresponds to a $(k,F_2)$-form on $X$. Then we check that the map $t_1 \mapsto t_2$ induces a bijection between the (isomorphism classes of) $(k,F_1)$-forms and the (isomorphism classes of) $(k,F_2)$-forms of $X$.
\end{proof}

\noindent \emph{Proof of Proposition~\ref{prop:A}:} 
The first part of Proposition~\ref{prop:A} is Proposition~\ref{prop:existence of (k,F)-form} and the second part is Corollary~\ref{cor:existence of forms for strongly eq rho}.
\qed

\smallskip

We now give an example of two equivalent descent data $\rho_1$ and $\rho_2$ on $G$, but not strongly equivalent, with corresponding $k$-forms $F_1$ and $F_2$, such that $X=G/H$ admits a $(k,F_1)$-form but does not admit a $(k,F_2)$-form.

\begin{example}\label{ex:counter-ex not strongly eq}
Let $k=\R$, let $G=\mathbb{G}_{m,\C}^{2}$, and let $H=\{1\} \times \mathbb{G}_{m,\C}$. Then $\Gamma= \{1,\gamma\} \simeq \Z/2\Z$. Let $\rho_{1,\gamma}\colon G \to G,\ (u,v) \mapsto (\overline{u},\overline{v}^{-1})$, let $\varphi \colon G \to G,\ (u,v) \mapsto (uv,v)$, and let $\rho_{2,\gamma}=\varphi \circ \rho_{1,\gamma} \circ \varphi^{-1}\colon G \to G,\ (u,v) \mapsto (\overline{u} \overline{v}^{-2},\overline{v}^{-1})$. Then $\rho_{1,\gamma}(H)=H$ but $\rho_{2,\gamma}(H)=\{(t^2,t)\ | \ t \in \mathbb{G}_{m,\C}\} \neq H$, and so $X=G/H$ admits a $(k,F_1)$-form but it does not admit a $(k,F_2)$-form by Proposition~\ref{prop:existence of (k,F)-form}.
\end{example}

\smallskip

The following theorem was proved by Borovoi-Gagliardi \cite[Theorem~1.6]{BG21} for arbitrary quasiprojective $G$-varieties. Since the proof in the homogeneous case is easier and relies only on Proposition~\ref{prop:existence of (k,F)-form}, we include it for the sake of completeness. 

\begin{theorem}\label{th:cohomological criterion} \emph{(\cite[Theorem~1.6]{BG21} applied to homogeneous spaces)}\\
Let $\rho_c$ be an inner twist of $\rho$ and  let $F_c$ be the corresponding $k$-form of $G$.
Let $X=G/H$ be a homogeneous space which admits a $(k,F)$-form. It admits a $(k,F_c)$-form if and only if the cohomology class $\Delta_H([\rho_c])$ is neutral, i.e.~\eqref{eq:class is neutral} holds.
\end{theorem}

\begin{proof}
The variety $X=G/H$ admits a $(k,F)$-form if and only if there exists a locally constant map $t\colon \Gamma \to G(\k)$ such that conditions \ref{item:condition i} and \ref{item:condition ii} of Proposition~\ref{prop:existence of (k,F)-form} are satisfied. When these two conditions are satisfied, a $(G,\rho)$-equivariant descent datum on $X$ is given by the map $\mu$ defined by \eqref{eq:descent datum on X}.  Since 
\[
\forall \gamma \in \Gamma,
\forall \varphi_{nH} \in \Aut_{\k}^{G}(X) \simeq (N_G(H)/H)(\k),\ 
\mu_{\gamma} \circ \varphi_{nH} \circ \mu_{\gamma}^{-1}=\varphi_{t_{\gamma}^{-1} \rho_{\gamma}(n) t_{\gamma}H},\] we can simplify  \eqref{eq:class is neutral}; this yields
\begin{align*}
 \eqref{eq:class is neutral} &\Leftrightarrow \varphi_{n_{\gamma_1 \gamma_2}H} \circ \varphi_{c_{\gamma_1} \rho_{\gamma_1}(c_{\gamma_2}) c_{\gamma_1 \gamma_2}^{-1}H}  \circ \varphi_{t_{\gamma_1}^{-1} \rho_{\gamma_1}(n_{\gamma_2})^{-1} t_{\gamma_1}H}  \circ \varphi_{n_{\gamma_1}^{-1}H}=Id_X \\
 &\Leftrightarrow n_{\gamma_1 \gamma_2} c_{\gamma_1} \rho_{\gamma_1}(c_{\gamma_2}) c_{\gamma_1 \gamma_2}^{-1} t_{\gamma_1}^{-1} \rho_{\gamma_1}(n_{\gamma_2})^{-1} t_{\gamma_1}  n_{\gamma_1}^{-1} \in H \\
  &\Leftrightarrow c_{\gamma_1 \gamma_2} \rho_{\gamma_1}(c_{\gamma_2})^{-1}c_{\gamma_1}^{-1} \in n_{\gamma_1 \gamma_2} t_{\gamma_1}^{-1} \rho_{\gamma_1}(n_{\gamma_2})^{-1} t_{\gamma_1}  n_{\gamma_1}^{-1}H.
\end{align*}

Therefore, the cohomology class $\Delta_H([\rho_c])$ is neutral if and only if there exists a locally constant map $n\colon \Gamma \to N_G(H)(\k)$ such that  the following holds:
\begin{equation}\label{eq:new condition neutral class}
\forall \gamma_1,\gamma_2 \in \Gamma,\ c_{\gamma_1 \gamma_2} \rho_{\gamma_1}(c_{\gamma_2})^{-1}c_{\gamma_1}^{-1} \in n_{\gamma_1 \gamma_2} t_{\gamma_1}^{-1} \rho_{\gamma_1}(n_{\gamma_2})^{-1} t_{\gamma_1}  n_{\gamma_1}^{-1}H.
\end{equation}
 
\smallskip

($\Rightarrow$): We suppose that $X$ admits a $(k,F_c)$-form. Hence there exists a locally constant map $s\colon \Gamma \to G(\k)$ such that conditions \ref{item:condition i} and \ref{item:condition ii} of Proposition~\ref{prop:existence of (k,F)-form} are satisfied.  This implies
\[ \forall \gamma \in \Gamma, s_\gamma H s_{\gamma}^{-1}=c_{\gamma}t_{\gamma}H t_{\gamma}^{-1} c_{\gamma}^{-1},\]
and so we can define a locally constant map 
\[m\colon \Gamma \to N_G(H)(\k),\ \gamma \mapsto s_{\gamma}^{-1} c_\gamma t_\gamma. \]
Then
\begin{align*}
 c_{\gamma_1 \gamma_2} \rho_{\gamma_1}(c_{\gamma_2})^{-1} c_{\gamma_1}^{-1} &= s_{\gamma_1 \gamma_2} m_{\gamma_1 \gamma_2} t_{\gamma_1 \gamma_2}^{-1} \rho_{\gamma_1}(s_{\gamma_2} m_{\gamma_2} t_{\gamma_2}^{-1})^{-1}  c_{\gamma_1}^{-1}\\
 & \hspace{-15mm}\in ( \rho_{c,\gamma_1}(s_{\gamma_2}) s_{\gamma_1} H) m_{\gamma_1 \gamma_2} (H t_{\gamma_1}^{-1} \rho_{\gamma_1}(t_{\gamma_2})^{-1}) \rho_{\gamma_1}(t_{\gamma_2}) \rho_{\gamma_1}(m_{\gamma_2})^{-1} \rho_{\gamma_1}(s_{\gamma_2}^{-1})  c_{\gamma_1}^{-1}\\
 &\hspace{-15mm}= c_{\gamma_1}  \rho_{\gamma_1}(s_{\gamma_2}) c_{\gamma_1}^{-1} s_{\gamma_1}  m_{\gamma_1 \gamma_2} H t_{\gamma_1}^{-1} \rho_{\gamma_1}(m_{\gamma_2})^{-1} \rho_{\gamma_1}(s_{\gamma_2})^{-1}  c_{\gamma_1}^{-1}.\\
\end{align*}
But $c_{\gamma_1 \gamma_2} \rho_{\gamma_1}(c_{\gamma_2})^{-1} c_{\gamma_1}^{-1} \in Z(G)(\k)$, so conjugating on both side by $\rho_{\gamma_1}(s_{\gamma_2})^{-1}c_{\gamma_1}^{-1}$ yields
\begin{align*}
c_{\gamma_1 \gamma_2} \rho_{\gamma_1}(c_{\gamma_2})^{-1} c_{\gamma_1}^{-1} &\in c_{\gamma_1}^{-1} s_{\gamma_1}  m_{\gamma_1 \gamma_2} H t_{\gamma_1}^{-1} \rho_{\gamma_1}(m_{\gamma_2})^{-1}\\
&=t_{\gamma_1} m_{\gamma_1}^{-1} m_{\gamma_1 \gamma_2} H t_{\gamma_1}^{-1} \rho_{\gamma_1}(m_{\gamma_2})^{-1}.
\end{align*}
Conjugating again on both sides by $m_{\gamma_1}t_{\gamma_1}^{-1}$ yields
\[
c_{\gamma_1 \gamma_2} \rho_{\gamma_1}(c_{\gamma_2})^{-1} c_{\gamma_1}^{-1} \in
m_{\gamma_1 \gamma_2} H t_{\gamma_1}^{-1} \rho_{\gamma_1}(m_{\gamma_2})^{-1} t_{\gamma_1} m_{\gamma_1}^{-1}.
\]
Also, we check that $t_{\gamma_1}^{-1} \rho_{\gamma_1}(m_{\gamma_2})^{-1} t_{\gamma_1} \in N_G(H)(\k)$, which yields
\[
c_{\gamma_1 \gamma_2} \rho_{\gamma_1}(c_{\gamma_2})^{-1} c_{\gamma_1}^{-1} \in
m_{\gamma_1 \gamma_2}  t_{\gamma_1}^{-1} \rho_{\gamma_1}(m_{\gamma_2})^{-1} t_{\gamma_1} m_{\gamma_1}^{-1} H.
\]
Hence, \eqref{eq:new condition neutral class} holds, i.e.~the cohomology class $\Delta_H([\rho_c])$ is neutral.

\smallskip

($\Leftarrow$): We assume that the cohomology class $\Delta_H([\rho_c])$ is neutral, i.e.~that there exists a locally constant map $n\colon \Gamma \to N_G(H)(\k)$ such that \eqref{eq:new condition neutral class} holds.

We define a locally constant map $s\colon\Gamma \to G(\k),\ \gamma \mapsto s_\gamma:=c_\gamma t_\gamma n_{\gamma}^{-1}$. Then 
\[
\forall \gamma \in \Gamma,\ 
\rho_{c,\gamma}(H)=c_\gamma \rho(H) c_{\gamma}^{-1}=s_\gamma H s_{\gamma}^{-1}.\] Moreover, unraveling the previous equalities yields that \eqref{eq:new condition neutral class} is equivalent to
\[
\forall \gamma_1,\gamma_2 \in \Gamma,\ c_{\gamma_1 \gamma_2} \rho_{\gamma_1}(c_{\gamma_2})^{-1}c_{\gamma_1}^{-1} \in \rho_{c,\gamma_1}(s_{\gamma_2}) s_{\gamma_1} H n_{\gamma_1 \gamma_2} t_{\gamma_1 \gamma_2}^{-1} \rho_{\gamma}(c_{\gamma_2})^{-1} c_{\gamma_1}^{-1}
\]
and so, simplifying on both sides by $\rho_{\gamma_1}(c_{\gamma_2})^{-1}c_{\gamma_1}^{-1}$ yields
\[
\forall \gamma_1,\gamma_2 \in \Gamma,\ c_{\gamma_1 \gamma_2} \in \rho_{c,\gamma_1}(s_{\gamma_2}) s_{\gamma_1} H n_{\gamma_1 \gamma_2} t_{\gamma_1 \gamma_2}^{-1}
\] 
which is equivalent to 
\[
\forall \gamma_1,\gamma_2 \in \Gamma,\ s_{\gamma_1 \gamma_2} \in \rho_{c,\gamma_1}(s_{\gamma_2}) s_{\gamma_1} H.
\] 
Hence the locally constant map $s\colon\Gamma \to G(\k)$ satisfies the two conditions of Proposition~\ref{prop:existence of (k,F)-form}, i.e.~$X$ admits a $(k,F_c)$-form. 
\end{proof}

\begin{corollary}\label{cor:easy cases} 
We keep the assumptions of Theorem~\ref{th:cohomological criterion}.
If $Z(G) \subseteq H$ (e.g.~if  $H=N_G(H)$), then $X=G/H$ admits a $(k,F_c)$-form.
\end{corollary}

\begin{proof}
The condition $Z(G) \subseteq H$ implies that the map $\kappa$, defined by \eqref{eq:def of kappa}, is trivial. Therefore the cohomology class $\Delta_H([\rho_c])$ is neutral, and so
$X$ admits a $(k,F_c)$-form by Theorem~\ref{th:cohomological criterion}.
\end{proof}

\begin{proposition}\label{prop:number of forms}
We keep the assumptions of Theorem~\ref{th:cohomological criterion} and we assume that $X=G/H$ admits a $(k,F_c)$-form.
If $\Aut_{\k}^{G}(X)$ is abelian or $Z(G) \subseteq H$, then there is a bijection between the (isomorphism classes of) $(k,F)$-forms and the (isomorphism classes of) $(k,F_c)$-forms of $X$.
\end{proposition}

\begin{proof} 
By Proposition~\ref{prop:existence of (k,F)-form}, there exist locally constant maps $t\colon \Gamma \to G(\k)$ and $s\colon \Gamma \to G(\k)$ such that 
\begin{equation*}
\mu\colon \Gamma \to \Aut(X),\ \gamma \mapsto (gH \mapsto \rho_{\gamma}(g) t_\gamma H).
\end{equation*}
is a  $(G,\rho)$-equivariant descent datum on $X$ and
\begin{equation*}
\mu'\colon \Gamma \to \Aut(X),\ \gamma \mapsto (gH \mapsto \rho_{c,\gamma}(g) s_\gamma H).
\end{equation*}
is a  $(G,\rho_c)$-equivariant descent datum on $X$.
Also note that the equalities $\rho_\gamma(H)=t_\gamma H t_{\gamma}^{-1}$ and $\rho_{c,\gamma}(H)=s_\gamma H s_{\gamma}^{-1}$ imply that $s_{\gamma}^{-1} c_\gamma t_\gamma \in N_G(H)(\k)$. 
This yields a locally constant map 
\[
n\colon \Gamma \to N_G(H)(\k), \ \gamma \mapsto n_\gamma:=s_{\gamma}^{-1} c_\gamma t_\gamma.
\]

By Proposition~\ref{prop:H1 Galois k-forms}, to prove the statement it suffices to construct a bijection between the sets $\HH^1((\Gamma,\mu),\Aut_{\k}^{G}(X))$ and $\HH^1((\Gamma,\mu'),\Aut_{\k}^{G}(X))$, where the notation ($\Gamma,\mu$) and ($\Gamma,\mu'$) means that $\Gamma$ acts on $\Aut_{\k}^{G}(X)$ through $\inn_\mu$ and $\inn_{\mu'}$ respectively.

Recall that $(N_G(H)/H)(\k) \xrightarrow{\sim} \Aut_{\k}^{G}(X),\ mH \mapsto \varphi_{mH}$ (Proposition~\ref{prop:G-eq automorphisms of G/H}). Then, for all $\gamma \in \Gamma$ and all $mH \in (N_G(H)/H)(\k)$, we have
\begin{align*}
\inn_\mu(\gamma,\varphi_{mH}):=\mu_\gamma \circ \varphi_{mH} \circ \mu_{\gamma}^{-1}&=\varphi_{t_{\gamma}^{-1}\rho_\gamma(m)t_\gamma H} \text{; and}\\
\inn_{\mu'}(\gamma,\varphi_{mH}):={\mu'}_\gamma \circ \varphi_{mH} \circ {\mu'}_{\gamma}^{-1}&=\varphi_{s_{\gamma}^{-1}\rho_{c,\gamma}(m)s_\gamma H}\\
&=\varphi_{(c_\gamma t_\gamma n_{\gamma}^{-1})^{-1}\rho_{c,\gamma}(m)c_\gamma t_\gamma n_{\gamma}^{-1} H}\\
&=\varphi_{n_\gamma t_{\gamma}^{-1}\rho_{\gamma}(m)t_\gamma n_{\gamma}^{-1} H}.
\end{align*}
Moreover $\rho_\gamma(H)=t_\gamma H t_{\gamma}^{-1}$ implies that $t_{\gamma}^{-1}\rho_{\gamma}(m)t_\gamma \in N_G(H)(\k)$.
Hence, if the group $N_G(H)/H$ is abelian, we have $n_\gamma t_{\gamma}^{-1}\rho_{\gamma}(m)t_\gamma n_{\gamma}^{-1} H=t_{\gamma}^{-1}\rho_{\gamma}(m)t_\gamma H$. 
On the other hand, if $Z(G) \subseteq H$, we can take $s_\gamma=c_\gamma t_\gamma$ (i.e.~$n_\gamma=1$)  and check that the two conditions of Proposition~\ref{prop:existence of (k,F)-form} holds using the fact that $c_{\gamma_1} \rho_{\gamma_1}(c_{\gamma_2})c_{\gamma_1 \gamma_2}^{-1} \in Z(G)(\k)$.

Either way, we see that $\inn_\mu(\gamma,\varphi_{nH})=\inn_{\mu'}(\gamma,\varphi_{nH})$, for all $\gamma \in \Gamma$ and all $nH \in (N_G(H)/H)(\k)$. Therefore, the identity map is an isomorphism between the $\Gamma$-groups ($\Aut_{\k}^{G}(X),\inn_\mu$) and ($ \Aut_{\k}^{G}(X),\inn_{\mu'}$), and so the corresponding pointed sets $\HH^1((\Gamma,\mu), \Aut_{\k}^{G}(X))$ and $\HH^1((\Gamma,\mu'), \Aut_{\k}^{G}(X))$ are in bijection. This finishes the proof.
\end{proof}

\noindent \emph{Proof of Theorem~\ref{th:B}:} 
Part~\ref{item part 1 th B} is Theorem~\ref{th:cohomological criterion} together with Corollary~\ref{cor:easy cases}, and part~\ref{item part 2 th B} is Proposition~\ref{prop:number of forms}.
\qed

\medskip

Let us mention that forms of spherical homogeneous spaces (see Definition~\ref{def:complexity}) over an arbitrary base field of characteristic zero were studied by Borovoi and Gagliardi in \cite{Bor20,BG21}. The reader is referred to \cite[\S~3]{MJT21} and \cite[\S~11]{BG21} for examples of spherical homogeneous spaces  for which versions of Proposition~\ref{prop:A} and Theorem~\ref{th:B} are applied to determine their $(k,F)$-forms. Other results, based in part on a weaker version of Proposition~\ref{prop:A}, concerning the real forms of complex symmetric spaces can be found in \cite{MJT19}.

\section{Luna-Vust theory over perfect fields}\label{sec:LV over perfect fields}
In the founding paper \cite{LV83} Luna and Vust studied equivariant embeddings of homogeneous spaces under the action of a reductive algebraic group and gave a combinatorial classification of these embeddings when the base field is algebraically closed of characteristic zero. Later in \cite{Tim97,Tim11} Timashev extended the results of Luna-Vust to classify the (not necessarily almost homogeneous) $G$-varieties in a given $G$-birational class when the base field is algebraically closed of arbitrary characteristic. 

\smallskip

Our principal objective in \S~\ref{sec:LV over perfect fields} is twofold: First we recall how this combinatorial classification works when the base field is algebraically closed (in \S~\ref{sec:LV over alg closed fields}), then we explain how this classification extends when the base field is perfect (in \S\S~\ref{sec:Gamma action on colored data}-\ref{sec:k-forms for models}). 
Moreover, in \S~\ref{sec:caseof complexity one} we briefly review how the Luna-Vust theory specializes for certain families of varieties of complexity $\leq 1$, and finally in \S~\ref{sec:strategy} we explain our strategy to determine the forms of arbitrary almost homogeneous varieties using the results obtained in \S\S~\ref{sec:forms of qh varieties2}-\ref{sec:LV over perfect fields}. This strategy will be then applied in \S~\ref{sec:real forms of SL2-threefolds} to determine the real forms of complex almost homogeneous $\SL_2(\C)$-threefolds.

\subsection{Recollections on Luna-Vust theory over algebraically closed fields}\label{sec:LV over alg closed fields}
Let $k$ be an algebraically closed field, let $G$ be a reductive algebraic group over $k$, let $B$ be a Borel subgroup of $G$, and let $X_0$ be a $G$-variety over $k$. 

In this subsection we give a brief overview of the Luna-Vust theory, which is detailed in \cite[Chapter~3]{Tim11}, concentrating only on the essential information   necessary to understand the content of Theorem~\ref{th:C}, which is the main result of \S~\ref{sec:LV over perfect fields}. The goal of the Luna-Vust theory is to classify the $G$-varieties in the $G$-birational class of $X_0$ (i.e.~$G$-equivariantly birational to $X_0$) in terms of certain combinatorial objects depending on $X_0$. In \S~\ref{sec:k-forms for models} we will see how to adapt this combinatorics to classify $G$-varieties over perfect fields.

\begin{definition}\label{def:colored equipment} (Colored equipment of the $G$-variety $X_0$.)
\begin{itemize}[leftmargin=*]
\item Let $K=k(X_0)$ be the function field of $X_0$. The group $G$ acts naturally on $K$.
\item By a \emph{valuation} $\nu$  of $K$ we mean a surjective homomorphism $\nu\colon (K^*,\times) \to (\Z,+)$ satisfying $\nu(a+b) \geq \min(\nu(a),\nu(b))$ when $a+b \neq 0$ and whose kernel contains $k^*$. A valuation $\nu$ is \emph{geometric} if there exists a variety $X$ in the birational class of $X_0$ such that $\nu=\nu_D$ with $D$ a prime divisor of $X$ (here $\nu_D(f)$ denotes the order of vanishing of $f \in K^*$ along $D$). 
Moreover, a valuation $\nu$ of $K$ is \emph{$G$-invariant} if $\nu(g \cdot f)=\nu(f)$ for all $g \in G$ and all $f \in K$.  We denote 
\[\V^G=\V^G(X_0)=\{\text{$G$-invariant geometric valuations of } K\}.\]
\item The set of \emph{colors} of $X_0$ is 
\[\DD^B=\DD^B(X_0)=\{ \text{$B$-stable prime divisors of $X_0$ that are not $G$-stable} \}.\]
\end{itemize}
The pair $(\V^G,\DD^B)$ is called \emph{colored equipment} of $X_0$.
\end{definition}

\begin{remark} \label{rk:geometric valuations and stable divisors}
By \cite[Proposition~19.8]{Tim11} any valuation $\nu \in \V^G$ is equal to $\nu_D$, where $D$ is a $G$-stable prime divisor of a $G$-variety in the $G$-birational class of $X_0$. Similarly, there is a one-to-one correspondence between the $B$-stable prime divisors of $X_0$ and the $B$-invariant geometric valuations of $K$.
\end{remark}

Let $\psi\colon X_0 \dashrightarrow X_1$ be a $G$-equivariant birational map. It induces a field isomorphism $\psi^*\colon k(X_1) \simeq k(X_0)$ and bijections $\V^G(X_0) \simeq \V^G(X_1)$ and $\DD^B(X_0) \simeq \DD^B(X_1)$. This is clear for $\V^G$ through the identification $k(X_1) \simeq k(X_0)$. For $\DD^B$ this follows from the fact that the complements of the $G$-stable dense open subsets of $X_0$ and $X_1$ over which $\psi$ is an isomorphism are a union of $G$-stable closed subvarieties, and therefore they contain no colors.

\begin{definition} \label{def:colored data of G-orbits}
Let $X$ be a $G$-variety in the $G$-birational class of $X_0$, and let $\psi\colon X_0 \dashrightarrow X$ be a $G$-equivariant birational map. 
\begin{enumerate}
\item The pair $(X,\psi)$ is called a \emph{$G$-model} of $X_0$. 
\item The \emph{colored data} of a $G$-orbit $Y\subseteq X$ (with respect to $\psi$) is the pair
\begin{itemize}[leftmargin=0mm]
\item $\V_Y^G=\{ \nu_D \in \V^G(X) \ | \ Y \subseteq D,\ \text{ where $D$ is a $G$-stable prime divisor of $X$} \}\\  \subseteq \V^G(X) \simeq \V^G$; \text{ and}
\item $\DD_Y^B=\{ D \in \DD^B(X) \ | \ Y \subseteq D \} \subseteq \DD^B(X) \simeq \DD^B$.
\end{itemize}
\end{enumerate}
\end{definition}

\begin{definition}\label{def:equivalence models}
Let $X_1$ and $X_2$ be $G$-varieties in the $G$-birational class of $X_0$. For $i=1,2$, let $\psi_i\colon X_0 \dashrightarrow X_i$ be  a $G$-equivariant birational map. We say that the $G$-models $(X_1,\psi_1)$ and $(X_2,\psi_2)$ are \emph{equivalent} if there exists a $G$-isomorphism $\varphi\colon X_1 \to X_2$ such that $\psi_2=\varphi \circ \psi_1$. 
\end{definition}

Let us note that if the $G$-models $(X_1,\psi_1)$ and $(X_2,\psi_2)$ are equivalent, then the colored data of the $G$-orbits of $X_1$ (with respect to $\psi_1$) coincides with the colored data of the $G$-orbits of $X_2$ (with respect to $\psi_2$). 

\begin{remark}
In the case where a $G$-variety $X$ has a dense open $G$-orbit $X_0=G/H$, then a canonical representative of the $G$-birational class of $X$ is given by $X_0$, and the category of $G$-models of $X_0$ identifies with the category of $G$-equivariant embeddings of $X_0$.
This is the original framework considered by Luna and Vust.
\end{remark}

\smallskip

The next statement is the central pillar of the Luna-Vust theory.

\begin{theorem}\label{th:Luna-Vust alg closed field}\emph{(see \cite[\S~14]{Tim11})}
The map 
\[(X,\psi) \mapsto \FF(X,\psi)=\{(\V_Y^G,\DD_Y^B), \text{ for every $G$-orbit $Y \subseteq X$} \}\]
induces a bijection between the equivalence classes of $G$-models of $X_0$ and the collections of pairs $(\W_i,\RR_i)_{i \in I}$, where $\W_i \subseteq \V^G$ and $\RR_i \subseteq \DD^B$, satisfying the conditions listed in \S~\ref{sec:A.2}. 
\end{theorem}

The collection of pairs $\FF(X,\psi)$ is called \emph{colored data} of the $G$-model $(X,\psi)$.

\subsection{Galois actions on the colored equipment} \label{sec:Gamma action on colored data}
Let $k$ be a perfect field, let $\k$ be an algebraic closure of $k$, and let $\Gamma=\Gal(\k/k)$ be the absolute Galois group of $k$. Let $F$ be a reductive algebraic group over $k$, and let $X_0$  be an $F$-variety over $k$. We denote $G=\Fk=F \times_{\Spec(k)} \Spec(\k)$ and $X_{0,\k}=X_0 \times_{\Spec(k)} \Spec(\k)$. Let $B$ be a Borel subgroup of $G$.

In \S~\ref{sec:LV over alg closed fields}, we introduced the \emph{colored equipment} $(\V^G,\DD^B)$ of a $G$-variety over $\k$.
In this subsection, we define (continuous) $\Gamma$-actions on this colored equipment. These actions were originally given by Hurguen in his thesis (see \cite[\S~2.2]{Hur11}).

\smallskip

First, $\Gamma$ acts on $G=F_{\k}$ and on $X_{0,\k}$ through its natural action on $\k$. We denote by $\rho\colon \Gamma \to \Aut(G)$ the corresponding descent datum on $G$ and by $\mu\colon \Gamma \to \Aut(X_{0,\k})$ the corresponding $(G,\rho)$-equivariant descent datum on $X_{0,\k}$ (see Definitions~\ref{def:k-forms of algebraic groups} and~\ref{def:k-forms of G-varieties}). 
The $\Gamma$-action on $X_{0,\k}$ induces a $\Gamma$-action on $K=\k(X_{0,\k})$ defined by
\begin{equation}\label{eq:Gamma linearity of the action}
\forall \gamma \in \Gamma,\ \forall f \in K, \ \forall x \in \Def(f \circ \mu_{\gamma^{-1}}),\ (\gamma \cdot f)(x):=\gamma \left( f(\mu_{\gamma^{-1}}(x)) \right),
\end{equation}
where $\Def(f) \subseteq X_{0,\k}$ is the maximal dense open subset over which $f$ is defined.

\smallskip

Second, there is also a $\Gamma$-action on 
\[\V^B=\{\text{$B$-invariant geometric valuations of $K$}\}
\] 
as we now explain.
If $\gamma \in \Gamma$, then $\rho_\gamma(B)$ is a Borel subgroup of $G$, therefore there exists $e_\gamma \in G$ such that $\rho_\gamma(B)=e_\gamma B e_\gamma^{-1}$. Moreover, $e_\gamma$ is unique up to right multiplication by an element of $B$. The group $\Gamma$ acts on  $\V^B$ as follows
\begin{equation*}
 \forall \gamma \in \Gamma, \ \forall \nu \in \V^B, \ \forall f \in K, \ (\gamma \cdot \nu)(f):=\nu(\gamma^{-1} \cdot (e_\gamma \cdot f)).
\end{equation*}
By \cite[Proposition~2.15]{Hur11}, this $\Gamma$-action is well-defined and does not depend on the particular choice of the $(e_\gamma)_{\gamma \in \Gamma}$. 

Given $\gamma \in \Gamma$ and $D \in \DD^B$, let $\gamma \cdot D$ be the unique $B$-stable prime divisor of $X_{0,\k}$ such that $\nu_{\gamma \cdot D}=\gamma \cdot \nu_D$; this defines a $\Gamma$-action on $\DD^B$ such that $D \in \DD^B \mapsto \nu_D \in \V^B$ is $\Gamma$-equivariant.  
Moreover, $\V^G \subseteq \V^B$ is $\Gamma$-stable and, using \eqref{eq:Gamma-eq}, the restriction of the $\Gamma$-action to $\V^G$ can be rewritten as follows
\begin{equation*}
 \forall \gamma \in \Gamma, \ \forall \nu \in \V^G, \ \forall f \in K, \ (\gamma \cdot \nu)(f):=\nu(\gamma^{-1} \cdot f).
\end{equation*}

\subsection{Forms of models over perfect fields}\label{sec:k-forms for models}
We keep the same notation as in \S~\ref{sec:Gamma action on colored data}. 
Our main goal in this subsection is to prove Theorem~\ref{th:C} which is a generalization of Theorem~\ref{th:Luna-Vust alg closed field} over perfect fields.

\smallskip

To extend the Luna-Vust theory over non algebraically closed fields we first need to introduce the notion of forms of $G$-models of $X_{0,\k}$ (which is distinct from the notion of forms for $G$-varieties introduced in \S~\ref{sec:descent data for algebraic varieties with group action}).

\begin{definition}\label{def:k-form of a model}\item
\begin{itemize}[leftmargin=6.5mm]
\item An \emph{$F$-model} of an $F$-variety $X_0$ over $k$ is a pair $(X,\delta)$, where $X$ is an $F$-variety over $k$ and $\delta\colon X_0 \dashrightarrow X$ is an $F$-equivariant birational map. 
\item Two models $(X_1,\delta_1)$ and $(X_2,\delta_2)$ are said to be \emph{equivalent} if there exists an $F$-isomorphism $\varphi\colon X_1 \to X_2$ such that $\delta_2=\varphi \circ \delta_1$.
\item A \emph{$(k,F,X_0)$-form of a $G$-model $(Z,\psi)$ of $X_{0,\k}$} is an $F$-model $(X,\delta)$ of $X_0$ such that $(X_{\k},\delta_{\k})$ is equivalent to $(Z,\psi)$, where $\delta_{\k}\colon X_{0,\k} \dashrightarrow X_{\k}$ is the $G$-equivariant birational map obtained from $\delta$ by extending the scalars.
\end{itemize}
\end{definition}

Let us note that if two $F$-models $(X_1,\delta_1)$ and $(X_2,\delta_2)$ of $X_0$ are equivalent, then the two $G$-models $(X_{1,\k},\delta_{1,\k})$ and $(X_{2,\k},\delta_{2,\k})$ of $X_{0,\k}$ are also equivalent. We will see in Lemma~\ref{lem:uniqueness of the k-form} that the converse holds as well.

\smallskip

Recall that we denote by $\mu\colon \Gamma \to \Aut(X_{0,\k})$ the descent datum on $X_{0,\k}$ defined by $\mu_\gamma=Id \times (\gamma^*)^{-1}$ for each $\gamma \in \Gamma$.
Then $\Gamma$ acts $k$-birationally on any $G$-model $(Z,\psi)$ of $X_{0,\k}$ through the homomorphism 
\[\Theta_{(Z,\psi)}\colon \Gamma \to \Bir(Z),\ \gamma \mapsto \psi \circ \mu_\gamma \circ  \psi^{-1},\] 
where we denote by $\Bir(Z)$ the group of birational transformations of the $\k$-variety $Z$.

\begin{lemma}\label{lem:uniqueness of the k-form}
Let $(Z,\psi)$ be a $G$-model of $X_{0,\k}$. If $(Z,\psi)$ admits a $(k,F,X_0)$-form, then it is unique up to equivalence.
\end{lemma}

\begin{proof}
Let $(X_1,\delta_1)$ and $(X_2,\delta_2)$ be two $(k,F,X_0)$-forms of $(Z,\psi)$. There exists a unique $G$-isomorphism $\varphi\colon X_{1,\k} \to X_{2,\k}$ such that $\delta_{2,\k}=\varphi \circ (\delta_{1,\k})$. Indeed, if such a $\varphi$ exists, then it must coincide with $\delta_{2,\k} \circ \delta_{1,\k}^{-1}$ over a dense open subset of $X_{1,\k}$, and so it is unique.) Moreover,
\[\forall \gamma \in \Gamma,\  \gamma \cdot \varphi=(\delta_{2,\k} \circ  \mu_\gamma  \circ \delta_{2,\k}^{-1}) \circ \varphi \circ (\delta_{1,\k} \circ  \mu_{\gamma^{-1}}  \circ \delta_{1,\k}^{-1})=\varphi. \]
In other words, $\varphi$ is defined over $k$, and so by \cite[Proposition~8.1]{Wed18} it is induced by a $G$-isomorphism $\phi\colon X_1 \to X_2$ satisfying $\delta_2=\phi \circ \delta_1$, which proves the lemma.
\end{proof}

\begin{proposition} \label{prop:two conditions existence of k forms for models}
A $G$-model $(Z,\psi)$ of $X_{0,\k}$ admits a $(k,F,X_0)$-form if and only if it satisfies the two conditions~\ref{label:existence descent datum} and~\ref{label:effectiveness descent datum}, or equivalently~\ref{label:existence descent datum} and~\ref{label:effectiveness descent datum2}, below.
\begin{enumerate}
\item \label{label:existence descent datum} The image of $\Theta_{(Z,\psi)}$ is contained in $\Aut(Z)$.
\item \label{label:effectiveness descent datum} The variety $Z$ is covered by $\Gamma$-stable affine open subsets.
\item \label{label:effectiveness descent datum2} The variety $Z$ is covered by $\Gamma$-stable quasiprojective open $G$-subvarieties.
\end{enumerate}
\end{proposition}

\begin{proof}
The first condition means that the map $\Theta_{(Z,\psi)}\colon \Gamma \to \Aut(Z)$
is an equivariant descent datum on $Z$ (see Definition~\ref{def:k-forms of G-varieties}), and the second condition that this descent datum is effective, i.e.~that the quotient $Z/\Gamma$ is a variety (and not just and algebraic space).

Therefore, if conditions \ref{label:existence descent datum} and \ref{label:effectiveness descent datum} are satisfied, then the quotient space $X:=[Z/\Gamma]$ is a $G$-variety such that $X_{\k} \simeq Z$ as $G$-varieties. Fix such a $G$-isomorphism $\varphi\colon X_{\k} \xrightarrow{\sim} Z$. Then the $G$-equivariant birational map $\psi'=\varphi^{-1} \circ \psi\colon X_{0,\k} \dashrightarrow X_{\k}$ is defined over $k$ (i.e.~$\gamma \cdot \psi'=\psi'$ for all $\gamma \in \Gamma$), and so by \cite[Proposition~8.1]{Wed18} there exists a $G$-equivariant birational map $\delta\colon X_0 \dashrightarrow X$ such that $\varphi \circ \delta_{\k}=\varphi \circ \psi'= \psi$. Hence, $(X,\delta)$ is a $k$-form of $(Z,\psi)$.

Conversely, assume that $(Z,\psi)$ admits a $k$-form, say $(X,\delta)$. Then there exists a $G$-isomorphism $\varphi\colon X_{\k} \ \xrightarrow{\sim} Z$ such that $\psi=\varphi \circ \delta_{\k}$. The group $\Gamma$ acts on $X_{\k}=X \times_{\Spec(k)} \Spec(\k)$ through its action on $\k$, therefore it corresponds to a homomorphism $\Theta_{(X_{\k},\delta_{\k})}\colon \Gamma \to \Aut(X_{\k})$, and 
\[\Theta_{(Z,\psi)}(\Gamma)=\varphi \circ \Theta_{(X_{\k},\delta_{\k})}(\Gamma) \circ \varphi^{-1} \subseteq \varphi \circ \Aut(X_{\k}) \circ \varphi^{-1}=\Aut(Z),\]
which is condition~\ref{label:existence descent datum}.
Moreover, $[Z/\Gamma] \simeq [X_{\k}/\Gamma] \simeq X$, which implies condition~\ref{label:effectiveness descent datum} by \cite[Proposition~V.1.8]{SGA1}. 

It remains to check that the conditions~\ref{label:effectiveness descent datum} and~\ref{label:effectiveness descent datum2} are equivalent.
The fact that~\ref{label:effectiveness descent datum2} implies~\ref{label:effectiveness descent datum} was already mentioned in \S~\ref{sec:descent data for algebraic varieties with group action}:
This follows from the fact that any $\Gamma$-orbit of $X(\k)$ is finite, hence contained in some affine open subset $U(\k)$ of $X(\k)$, and then the affine open subset $\bigcap_{\gamma \in \Gamma} \gamma \cdot U(\k)$ is non-empty and $\Gamma$-stable in $X(\k)$.
Conversely, assume that \ref{label:effectiveness descent datum} holds, and let $V \subset X$ be a $\Gamma$-stable affine dense open subset. Then $V':=G \cdot V \subset X$ is a $G$-stable quasi-projective dense open subset by \cite[Corollary~3]{Ben13}, and the fact that $V'$ is $\Gamma$-stable follows from the fact that the $\Gamma$-action on $X$ corresponds to an equivariant descent datum on $X$ (more precisely, it follows from Condition~\eqref{eq:Gamma-eq} in Definition~\ref{def:k-forms of G-varieties}).
\end{proof}

Let $(Z,\psi)$ be a $G$-model of $X_{0,\k}$. A \emph{$B$-chart} of $Z$ is a $B$-stable affine dense open subset of $Z$. It follows from a local structure theorem (see \cite[Theorem~4.7]{Tim11}) that the variety $Z$ is covered by $G$-translates of finitely many $B$-charts (but this cover is not unique in general). Let us mention that this key-fact, which generalizes Sumihiro's theorem for torus actions, was the starting point of the Luna-Vust theory.   

We denote 
\[\V_{Z}^{G}=\bigcup_{\text{$G$-orbits $Y \subseteq Z$}}  \V_{Y}^{G} \ \text{ and }\  \DD_{Z}^{B}=\bigcup_{\text{$G$-orbits $Y \subseteq Z$}}  \DD_{Y}^{B},\]
where $(\V_{Y}^{G},\DD_{Y}^{B})$ is the colored data of the $G$-orbit $Y \subseteq Z$.
Also, for a set $S$, we denote by $\mathfrak{P}(S)$ the powerset of $S$. 

The next result gives a combinatorial description of the $B$-charts of $Z$.

\begin{proposition}\label{prop:colored data of B-charts} \emph{(\cite[Corollary~13.9]{Tim11})}
The map 
\[ \{\text{$B$-charts of }Z\} \to \mathfrak{P}(\V_{Z}^{G}) \times \mathfrak{P}(\DD_{Z}^{B}),\ \mathring{Z} \mapsto  
 \bigcup_{\substack{\text{$G$-orbits $Y \subseteq Z$},\\ Y \cap \mathring{Z} \neq \varnothing}} \left(  \V_{Y}^{G}, \DD_{Y}^{B} \right) \]
 is a bijection between the $B$-charts of $Z$ and the pairs $(\W,\RR) \in   \mathfrak{P}(\V_{Z}^{G}) \times \mathfrak{P}(\DD_{Z}^{B})$ satisfying conditions $(C)$, $(F)$, and $(W)$ given in \S~\ref{sec:A.1}. 
 
 The converse map sends the pair $(\W,\RR)$ to $\Spec(A[\W,\RR])$, where
\begin{equation}\label{eq:B-chart as intersection of valuation rings}
 A[\W,\RR]= \left(\bigcap_{w \in \W} \O_w\right) \cap  \left(\bigcap_{D \in \RR} \O_{\nu_D}\right) \cap  \left(\bigcap_{D \in \DD \setminus \DD^B} \O_{\nu_D}\right) \subseteq  \k(Z)
 \end{equation} 
and $\DD$ is the set of prime divisors on $Z$ that are not $G$-stable. 
\end{proposition}

\begin{proposition}\label{prop:Gamma stability}
Let $(Z,\psi)$ be a $G$-model of $X_{0,\k}$.
The condition~\ref{label:existence descent datum} in Proposition~\ref{prop:two conditions existence of k forms for models} holds if and only if the colored data $\FF(Z,\psi)$ are preserved by the $\Gamma$-action on $\V^{G}$ and $\DD^B$ introduced in \S~\ref{sec:Gamma action on colored data}, i.e.~
\[\forall \gamma \in \Gamma,\ \forall (\V_{Y}^{G},\DD_Y^B) \in \FF(Z,\psi),\ \gamma \cdot (\V_{Y}^{G},\DD_Y^B) \in \FF(Z,\psi).\] 
\end{proposition}

\begin{proof}
Let us assume that the image of $\Theta_{(Z,\psi)}$ is contained in $\Aut(Z)$. 
Relation~\eqref{eq:Gamma-eq} implies that if $\gamma \in \Gamma$ and $Y$ is a $G$-orbit of $Z$, then $Y':=\gamma \cdot Y$ is also a $G$-orbit. Moreover, the $\Gamma$-actions defined in \S~\ref{sec:Gamma action on colored data} satisfy
$\gamma \cdot (\V_{Y}^{G},\DD_Y^B)=(\V_{Y'}^{G}, \DD_{Y'}^{B})$, and so   
the colored data $\FF(Z,\psi)$ are preserved by the $\Gamma$-action on the colored equipment $(\V^{G},\DD^B)$ of $X_{0,\k}$.

Conversely, we now assume that the colored data $\FF(Z,\psi)$ are globally preserved by the $\Gamma$-action on $(\V^{G},\DD^B)$. 
Fix $\gamma \in \Gamma$. Let $Y$ be a $G$-orbit of $Z$, and let $\mathring{Z}$ be a $B$-chart of $Z$ such that $\mathring{Z} \cap Y \neq \varnothing$. 
According to Proposition~\ref{prop:colored data of B-charts} there exists a unique pair $(\W,\RR)$ such that $\Spec(A[\W,\RR])=\mathring{Z}$.
Then the pair $(\W',\RR'):=\gamma^{-1} \cdot (\W,\RR)$ also satisfies the conditions $(C)$, $(F)$, and $(W)$ (recalled in \S~\ref{sec:A.1}). Moreover, the $\Gamma$-stability of $\FF(Z,\psi)$ means that $\mathring{Z}'=\Spec(A[\W',\RR'])$ is a $B$-chart of $Z$ (and not just a $B$-chart of some $G$-variety in the $G$-birational class of $Z$). 

The element $\gamma \in \Gamma$ acts on $\k(Z)$ and, using \eqref{eq:B-chart as intersection of valuation rings}, we see that it maps $A[\W',\RR']$ to $A[\W,\RR]$. Hence, it corresponds to a regular morphism $\mathring{Z} \to \mathring{Z}'$. In other words, the open locus $U_\gamma \subseteq Z$ over which $\overline{\gamma}:=\Theta_{(Z,\psi)}(\gamma)=\psi \circ  \mu_\gamma \circ \psi^{-1}$ is defined contains $\mathring{Z}$. 
It follows from the relation~\eqref{eq:Gamma linearity of the action} that $U_\gamma$ is $G$-stable, so it also contains $G \cdot \mathring{Z}$. In particular, $U_\gamma$ contains the $G$-orbit $Y$. But $Y$ is an arbitrary $G$-orbit of $Z$, which means that $U_\gamma=Z$, i.e.~$\overline{\gamma}$ is an element of $\Aut(Z)$. Hence condition~\ref{label:existence descent datum} in Proposition~\ref{prop:two conditions existence of k forms for models} holds.
 \end{proof}

We can now state and prove Theorem~\ref{th:C}, which generalizes Theorem~\ref{th:Luna-Vust alg closed field} to the case where the base field is perfect but not necessarily algebraically closed.

\begin{theorem}\label{th:C}
Let $F$ be a reductive algebraic group over $k$, let $G=F_{\k}$, and let $X_0$ be an $F$-variety over $k$. The map 
\[(X,\delta) \mapsto \FF(X_{\k},\delta_{\k})=\left\{(\V_{Y}^{G},\DD_{Y}^{B}), \text{ for every $G$-orbit $Y \subseteq X_{\k}$} \right\},\]
where $(\V_{Y}^{G},\DD_{Y}^{B})$ is the colored data of the $G$-orbit $Y \subseteq X_{\k}$ (see Definition~\ref{def:colored data of G-orbits}), is a bijection between the equivalence classes of $F$-models of $X_0$ and the collections of pairs $(\W_i,\RR_i)_{i \in I}$, where $\W_i \subseteq \V^{G}$ and $\RR_i \subseteq \DD^B$ are subsets such that 
\begin{enumerate}[leftmargin=10mm]
\item \label{item:condition 1 LV perfect field} the conditions listed in \S~\ref{sec:A.2} are satisfied, i.e.~the collection $(\W_i,\RR_i)_{i \in I}$ corresponds to an equivalence class of a  $G$-model $(Z,\psi)$ of $X_{0,\k}$;
\item \label{item:condition 2 LV perfect field} the collection $(\W_i,\RR_i)_{i \in I}$ is globally preserved by the $\Gamma$-actions on $\V^{G}$ and $\DD^B$ introduced in \S~\ref{sec:Gamma action on colored data}; and
\item\label{item:condition 3 LV perfect field} the variety $Z$ is covered by $\Gamma$-stable affine open subsets.
\end{enumerate}
\end{theorem}

\begin{proof}
By Theorem~\ref{th:Luna-Vust alg closed field}, the map $(Z,\psi)\mapsto \F(Z,\psi)$ induces a bijection between equivalence classes of $G$-models $(Z,\psi)$ of $X_{0,\k}$ and the collections of pairs $(\W_i,\RR_i)_{i \in I}$, where $\W_i \subseteq \V^{G}$ and $\RR_i \subseteq \DD^B$ are subsets satisfying the condition~\ref{item:condition 1 LV perfect field}.

By Lemma~\ref{lem:uniqueness of the k-form}, if $(Z,\psi)$ admits a $(k,G,X_0)$-form then it is unique up to equivalence. Moreover, by Proposition~\ref{prop:two conditions existence of k forms for models}, it admits a $(k,G,X_0)$-form if and only if 
\begin{enumerate}[(a), leftmargin=*]
\item the image of the map $\Theta_{(Z,\psi)}$ is contained in $\Aut(Z)$, which is equivalent to the condition~\ref{item:condition 2 LV perfect field} by Proposition~\ref{prop:Gamma stability}; and 
\item the variety $Z$ is covered by $\Gamma$-stable affine open subsets, which is the condition~\ref{item:condition 3 LV perfect field}.
\end{enumerate}

Therefore, the equivalence classes of $F$-models of $X_0$ are in one-to-one correspondence with the collections of pairs $(\W_i,\RR_i)_{i \in I}$ satisfying the three conditions \ref{item:condition 1 LV perfect field}-\ref{item:condition 2 LV perfect field}-\ref{item:condition 3 LV perfect field} of the statement, and this correspondence is indeed given by the map $(X,\delta) \mapsto \F(X_{\k},\delta_{\k})$. This concludes the proof.
\end{proof}

\begin{remark}
In fact, condition~\ref{item:condition 3 LV perfect field} of Theorem~\ref{th:C} is equivalent to the fact that the $G$-variety $Z$ is covered by $\Gamma$-stable quasiprojective open $G$-subvarieties (see Proposition~\ref{prop:two conditions existence of k forms for models}), and the condition for a $G$-variety to be quasiprojective can be expressed in terms of the colored data.
However, if one prefers to work in the category of algebraic spaces instead of schemes (as Wedhorn in \cite{Wed18}), then one can eliminate condition~\ref{item:condition 3 LV perfect field}, which simplifies the statement.   
\end{remark}

\begin{remark}
If $k=\k$ and $X_0$ is a homogeneous space under some reductive algebraic group, then we recover the classical result of Luna-Vust \cite{LV83} which actually gives a combinatorial classification of the equivariant embeddings of $X_0$.
Other versions of Theorem~\ref{th:C} for varieties of complexity $\leq 1$ can be found in the literature; see \S~\ref{sec:caseof complexity one} for a detailed account.
\end{remark}

\subsection{The case of complexity \texorpdfstring{$\leq 1$}{<=1}}\label{sec:caseof complexity one}
We keep the same notation as in \S~\ref{sec:Gamma action on colored data}.

\begin{definition}\label{def:complexity}
The \emph{complexity} of an $F$-variety $X$ is the codimension of a general $B$-orbit of $X_{\k}$. A complexity-zero variety (i.e.~an $F$-variety $X$ such that $B$ acts on $X_{\k}$ with a dense open orbit) is usually called a \emph{spherical} variety.
\end{definition}

Let  $X_0$ be an $F$-variety of complexity $\leq 1$. Then the description of $F$-models of $X_0$ via the Luna-Vust theory is effective in the sense that the colored equipment of $X_0$ (see Definition~\ref{def:colored equipment}) can be expressed in terms of implementable combinatorial structures that allow us to study these varieties similarly as is done for toric varieties.

We now briefly recall what this combinatorial description looks like when the base field is algebraically closed and how Theorem~\ref{th:C} specializes in this situation. 

\smallskip

\begin{itemize}[leftmargin=*]

\item \textbf{Torus embeddings:} Let $T$ be a torus over $\k$, and let $Z_0 \simeq T$. 
Consider the $\Q$-vector space $N_{\Q}=N \otimes_\Z \Q$, with $N=\Hom_{gr}(\Gm,T)$.  It is a classical  result (see e.g. \cite{Ful93}) that equivalence classes of $T$-models of $Z_0$, also known as \emph{torus embeddings} or \emph{toric varieties}, are in bijection with combinatorial objects, known as  \emph{fans}. They are certain finite collections of convex cones  in the space  $N_{\Q}$.
In this situation $\V^T\simeq N$, $\DD^B=\varnothing$, and the cones of a fan of a  torus embedding $(Z,\iota)$ of $Z_0$ correspond to the colored data of $Z$.

Forms of torus embeddings (in the sense of Definition~\ref{def:k-forms of G-varieties}) were studied by Huruguen in \cite{Hur11} for arbitrary torus embeddings, by Elizondo--Lima-Filho--Sottile--Teitler in \cite{ELST14} for projective spaces and toric surfaces, and by Duncan in \cite{Dun16} who considered other notions of forms of torus embeddings.

In this setting, Theorem~\ref{th:C} corresponds to \cite[Theorem~1.22]{Hur11}.

\smallskip
\item \textbf{Complexity-one $T$-varieties:} Let again $T$ be a torus over $\k$, let $C$ be a curve over $\k$, and let $Z_0=T \times C$ on which $T$ acts by left-multiplication on $T$ and trivially on $C$. 
It follows from the work of Altmann-Hausen \cite{AH06} (affine case) and Altmann-Hausen-S\"u\ss \, \cite{AHS08} (general case) that equivalence classes of $T$-models of $Z_0$ are in bijection with the \emph{divisorial fans} on ($C,N_\Q$) with $N_\Q=\Hom_{gr}(\Gm,T) \otimes_\Z \Q$; roughly speaking these are $1$-dimensional families of pseudo-fans in $N_\Q$ parametrized by $C$ and constant over a dense open subset (see \cite[\S~5]{AHS08} for a precise definition).

Forms of affine complexity-one $T$-varieties were studied by Langlois in \cite{Lan15} (and by Gillard in \cite{Gil22,Gil} in arbitrary complexity). 
In this setting, Theorem~\ref{th:C} corresponds to \cite[Theorem~5.10]{Lan15}.

\smallskip
\item \textbf{Spherical embeddings:} Let $G$ be a reductive algebraic group over $\k$, let $B \subseteq G$ be a Borel subgroup, and let $Z_0 \simeq G/H$ be a $G$-homogeneous space with a dense open $B$-orbit. The notion of fan defined for torus embeddings generalizes to the notion of \emph{colored fan}, which is a finite collection $\mathcal{E}=\{(\mathcal{C}_i,\mathcal{F}_i), \ i \in I\}$ of \emph{colored cones} satisfying some properties (see \cite[\S~3]{Kno91} for details).

It follows from the work of Luna-Vust \cite{LV83} and Knop \cite{Kno91} that equivalence classes of $G$-models of $Z_0$, also known as \emph{spherical embeddings} or \emph{spherical varieties}, are in bijection with the \emph{colored fans} in some $\Q$-vector space $N_{\Q}$ (depending on $Z_0$). 
In this situation the colored cones of a colored fan of a spherical embedding $(Z,\iota)$ of $Z_0$ correspond to the colored data of $Z$. 

Forms of spherical embeddings were studied by Huruguen in \cite{Hur11} over perfect fields (see also \cite[\S~7]{BG21}), and by Wedhorn  (who works in the category of algebraic spaces instead of schemes) in \cite{Wed18} over arbitrary fields.

In this setting, Theorem~\ref{th:C} corresponds to \cite[Theorem~2.26]{Hur11}.

\smallskip
\item \textbf{Complexity-one $G$-varieties:} Let again $G$ be a reductive algebraic group over $\k$, and let $Z_0$ be a complexity-one $G$-variety. 
Timashev introduced in \cite{Tim97} (see also \cite[\S~16]{Tim11}) the notions of \emph{colored hypercones} and \emph{colored hyperfans} to represent colored data of models of complexity-one varieties in a way similar to the spherical case. He proves that there is a bijection between equivalence classes of $G$-models of $Z_0$ and colored hyperfans in some colored hyperspace (depending on $Z_0$). See also \cite{LT16} for an alternative combinatorial description of the $G$-models of $Z_0$ in terms of \emph{colored divisorial fans} in the case where $Z_0 \simeq G/H \times C$ with $C$ a curve and $G/H$ a horospherical homogeneous space (which means that $H$ contains a maximal unipotent subgroup of $G$).  

As far as we know, forms of complexity-one $G$-varieties have never been systematically studied, except for torus actions (i.e.~for complexity-one $T$-varieties).

\end{itemize}

\subsection{Strategy to determine the forms of almost homogeneous varieties}\label{sec:strategy}
We keep the same notation as in \S~\ref{sec:Gamma action on colored data}. 
Let $Z$ be a given almost homogeneous $G$-variety with dense open orbit $Z_0$. We are interested in classifying the $(k,F)$-forms of $Z$. Recall that these correspond to effective $(G,\rho)$-equivariant descent data on $Z$ (see \S~\ref{sec:descent data for algebraic varieties with group action}).
Let us note that if $\mu\colon \Gamma \to \Aut(Z)$ is an equivariant descent datum on $Z$, then it induces an equivariant descent datum on $Z_0$ by restriction. This suggests to start by studying the equivariant descent data on $Z_0$ (which correspond to $(k,F)$-forms of $Z_0$), and then 
to determine which ones extend to $Z$. 

(Also note that, in what follows, we will implicitly fix a $\k$-point in the dense open orbit $Z_0 \simeq G/H$ of $Z$, but this choice plays a role only in the calculations, not in the classification result of the isomorphism classes of $(k,F)$-forms of $Z$.)

\smallskip

The strategy to determine the isomorphism classes of $(k,F)$-forms of $Z$ is therefore the following:
\begin{enumerate}[leftmargin=7mm]
\item Determine whether the dense open orbit $Z_0 \simeq G/H$ admits a $(k,F)$-form with Proposition~\ref{prop:A} and Theorem~\ref{th:B}. 

\item If $Z_0$ has a $(k,F)$-form, then use Proposition~\ref{prop:H1 Galois k-forms} to parametrize the isomorphism classes of $(k,F)$-forms of $Z_0$.

\item Pick a $(k,F)$-form $X_0$ of $Z_0$ in each  isomorphism class of $(k,F)$-forms of $Z_0$ and apply Luna-Vust theory (Theorem~\ref{th:C}) to determine if the corresponding equivariant descent datum extends to an effective equivariant descent datum on $Z$, i.e.~if the $G$-model $(Z,\psi)$ of $X_{0,\k}$ admits a $(k,F,X_0)$-form $(X,\delta)$, see Definition~\ref{def:k-form of a model}, where $\psi\colon X_{0,\k} \dashrightarrow Z$ is the $G$-equivariant birational map obtained by composing the $G$-isomorphism $X_{0,\k} \xrightarrow{\sim} Z_0$ with the inclusion $Z_0 \hookrightarrow Z$.

\item Determine whether the natural homomorphism $\Aut_{\k}^{G}(Z) \hookrightarrow \Aut_{\k}^{G}(Z_0)$ is an isomorphism, in which case our strategy provides one representative for each isomorphism class of $(k,F)$-forms of $Z$. Otherwise one is left to determine which $(k,F)$-forms in each isomorphism class of $(k,F)$-forms of $Z_0$ correspond to equivariant descent data that 
extend to effective equivariant descent data on $Z$, and which ones are equivalent after extension.
\end{enumerate} 

\smallskip

This strategy, which was applied in \cite[\S~3.6]{MJT21} to determine the real forms of the complex horospherical varieties with Picard rank $1$, will be applied in the next section to determine the real forms of complex almost homogeneous $\SL_2(\C)$-threefolds.

\section{Real forms of complex almost homogeneous \texorpdfstring{$\SL_2$}{SL2}-threefolds}\label{sec:real forms of SL2-threefolds}
In this section, we study the real forms of the complex $\SL_2(\C)$-threefolds with a dense open orbit. 
Since the Galois group of $\C/\R$ is a group of order two, the notation introduced in \S\S~\ref{sec:Forms of algebraic groups and inner twists}-\ref{sec:descent data for algebraic varieties with group action} regarding the descent data for algebraic groups and algebraic varieties endowed with algebraic group action can be simplified; this is explained in \S~\ref{sec:real forms SL2}.
In \S~\ref{sec:SL2 homogeneous case}, we classify all the equivariant real structures on homogeneous spaces $\SL_2(\C)/H$ with $H$ a finite subgroup. Then, in \S~\ref{sec:almost homogeneous SL2 threefolds}, we explain how Theorem~\ref{th:C} specializes in this setting and illustrate it with some examples. Finally, in \S~\ref{sec:minimal models}, we apply the strategy described in \S~\ref{sec:strategy} to classify the equivariant real structures on minimal smooth completions of $\SL_2(\C)/H$ when $H$ is non-cyclic.

\subsection{Real forms, real structures, and real loci}\label{sec:real forms SL2}
In the case of  $k=\R$ and $\overline{k}=\C$, the Galois group is $\Gamma=\{1,\gamma\} \simeq \Z/2\Z$. Given a complex algebraic group $G$, a descent datum $\rho\colon \Gamma \to \Aut(G)$ satisfies : $\rho_1=Id_G$ and $\rho_\gamma=:\sigma$  is a scheme involution on $G$ such that the diagram 
\[
\xymatrix@R=4mm@C=2cm{
    G \ar[rr]^{\sigma} \ar[d]  && G \ar[d] \\
    \Spec(\C)  \ar[rr]^{\gamma^*\ =\ \Spec(z \mapsto \overline{z})} && \Spec(\C)  
  }
  \]
commutes, and the induced morphism $\gamma^*G \to G$ is an isomorphism of complex algebraic groups, where $\gamma^*G \to \Spec(\C)$ is the base change of $G \to \Spec(\C)$ along the morphism at the bottom of the square. 
We therefore call $\sigma$ a \emph{real group structure} on $G$. Two real group structures $\sigma_1$ and $\sigma_2$ on $G$ are called \emph{equivalent} if there exists a group automorphism $\psi \in \Aut_\C(G)$ such that $\sigma_2=\psi\circ\sigma_1\circ\psi^{-1}$.

Note that the map $\sigma \mapsto G/\left\langle \sigma \right\rangle$ gives a correspondence between equivalence classes of real group structures on $G$ and isomorphism classes of real forms of $G$, as a real algebraic group.

The \emph{real locus}  of a real group structure $\sigma$ on $G$ is the set of fixed points $G(\C)^\sigma$; it identifies with the group of $\R$-points of the corresponding real form $G/\left\langle \sigma \right \rangle$.

\smallskip

We now introduce the notion of \emph{equivariant real structure} for a complex $G$-variety. Given a complex algebraic group $G$ endowed with a real group structure $\sigma$, and a complex $G$-variety $X$, we say that $\mu$ is a \emph{$(G,\sigma)$-equivariant real structure} on $X$ if $\mu$ is a scheme involution on $X$  such that the diagram 
\[
\xymatrix@R=4mm@C=2cm{
    X \ar[rr]^{\mu} \ar[d]  && X \ar[d] \\
    \Spec(\C)  \ar[rr]^{\gamma^*\ =\ \Spec(z \mapsto \overline{z})} && \Spec(\C)  
  }
  \]
commutes, and such that
\[ \forall g \in G(\C), \forall x \in X(\C),\ \mu(g\cdot x)=\sigma(g)\cdot \mu(x).
\]  
As for real group structures on complex algebraic groups, an equivariant real structure $\mu$ on $X$ corresponds to an equivariant descent datum $\tilde{\mu}$ on $X$ given by $\tilde{\mu}_1=Id$ and $\tilde{\mu}_\gamma(x)=\mu(x)$, for all $x \in X$. 
Two $(G,\sigma)$-equivariant real structures $\mu_1$ and $\mu_2$ on $X$ are called \emph{equivalent} if there exists a $G$-equivariant automorphism $\varphi\in \Aut_{\C}^{G}(X)$ such that $\mu_2=\varphi\circ \mu_1\circ \varphi^{-1}$.

Given a $(G,\sigma)$-equivariant real structure $\mu$ on a $G$-variety $X$ such that $X$ is covered by $\left\langle \mu \right \rangle$-stable affine open subsets, the quotient $X/\left\langle \mu \right \rangle$ defines a real algebraic variety endowed with an algebraic action of the real algebraic group $F=G/\left\langle \sigma \right \rangle$; it is an $(\R,F)$-form of $X$ determined by $\mu$. 
Any real form of the $G$-variety $X$ is obtained from a unique equivariant real structure on $X$. But for a given equivariant real structure $\mu$ on $X$, the quotient $X/\left\langle \mu \right \rangle$ may not be an algebraic variety (only an algebraic space) in which case there is no real form associated with $\mu$. Also, two real forms $X/\left\langle \mu_1 \right \rangle$ and $X/\left\langle \mu_2 \right \rangle$ are isomorphic if and only if $\mu_1$ and $\mu_2$ are equivalent.   

The \emph{real locus} of a $(G,\sigma)$-equivariant real structure $\mu$ on $X$ is the (possibly empty) set  of fixed points $X(\C)^\mu$; it identifies with the set of $\R$-points of the corresponding real form $X/\left \langle \mu \right \rangle$. Moreover, it is endowed with an action of $G(\C)^\sigma$.

\smallskip

For the rest of this article we consider only the case $G=\SL_2(\C)$ and, to simplify the notation, we use the terminology of real structures instead of descent data.

Let $\sigma_s$ be the real group structure defined by 
$\sigma_s(g)=\text{\emph{the complex conjugate of $g$}}$. Its real locus is $\SL_2(\R)$. This structure is called \emph{split}, in the sense that the corresponding real form $G/\left\langle \sigma_s \right \rangle$ is split.  Let $\sigma_c$ be the real group structure defined by $\sigma_c(g)={}^{t}\sigma_s(g)^{-1}$. It is called \emph{compact} because its real locus is the compact real Lie group $\SU_2(\C)$. 
Any real group structure on $G$ is strongly equivalent to either $\sigma_s$ or $\sigma_c$ (see e.g. \cite[\S~V\!I.10]{Kna02}). 
Also, note that $\sigma_c= \inn_{e} \circ \sigma _s$ (i.e.~$\sigma_c(g)=e \sigma_s(g) e^{-1}$ for all $g \in G$), where $e=
\begin{bmatrix}
0 & 1 \\-1 & 0
\end{bmatrix}.$ In particular, $\sigma_c$ is an inner twist of $\sigma_s$ (see Definition~\ref{def:inner twist}). 

\smallskip

In the next subsections, we will study the equivariant real structures on almost homogeneous $G$-threefolds. For each such variety, we will often consider separately the $(G,\sigma_s)$-equivariant real structures and the $(G,\sigma_c)$-equivariant real structures on it.

\subsection{Real forms in the homogeneous case}\label{sec:SL2 homogeneous case} 
We start with the homogeneous case. The main goal of this subsection is to prove Theorem~\ref{th:D}.

\smallskip

A homogeneous $G$-threefold is $G$-isomorphic to $G/H$, where $H$ is a finite subgroup of $G=\SL_2(\C)$. Thus the first step to study the equivariant real structures on homogeneous $G$-threefolds  is to recall the set of conjugacy classes of all finite subgroups of $G$.

\begin{remark} \label{prop:finite subgroups of SL2}\emph{(Finite subgroups of $\SL_2(\C)$; see \cite{Klein93})}.
If $H$ is a finite subgroup of $\SL_2(\C)$, then it is conjugate to one of the following subgroups:
\begin{itemize}[leftmargin=4mm]
\item $(A_n$, $n \geq 1)$ The cyclic group, of cardinal $n$, generated by $\w_n=\begin{bmatrix}
\zeta_{n} & 0 \\ 0 & \zeta_{n}^{-1}
\end{bmatrix}$, where $\zeta_n$ is a primitive $n$-th root of unity.
\item $(D_n$, $n \geq 4)$ The binary dihedral group, of cardinal $4n-8$, generated by $A_{2n-4}$ and $\begin{bmatrix}
0 & i \\ i & 0
\end{bmatrix}$.
\item $(E_6)$ The binary tetrahedral group, of cardinal $24$, generated by $D_4$ and $\alpha=\frac{1}{2} \begin{bmatrix}
1-i & 1-i \\ -1-i & 1+i
\end{bmatrix}$.
\item $(E_7)$ The binary octahedral group, of cardinal $48$, generated by $E_6$ and $A_8$.
\item $(E_8)$ The binary icosahedral group, of cardinal $120$, generated by $A_{10}$, $\begin{bmatrix}
0 & 1 \\ -1 & 0
\end{bmatrix}$, and $\beta=\frac{1}{\zeta_5^2-\zeta_5^3} \begin{bmatrix}
\zeta_5+\zeta_5^{-1} & 1 \\ 1 & -\zeta_5-\zeta_5^{-1}
\end{bmatrix}$.
\end{itemize}
\end{remark}

We now prove the existence of an equivariant real structure on each $G/H$. Until the end of \S~\ref{sec:SL2 homogeneous case}, for a fixed subgroup $H$ of $G$, the notation $\modH{g}\in G/H$ will indicate the class of $g$ modulo $H$ (not to be confused with the complex conjugation).

\begin{lemma}\label{lem:existence of an equivariant structure}
Let $H$ be a finite subgroup of $G=\SL_2(\C)$, and let $\sigma$ be a real group structure on $G$. 
Then $G/H$ admits a $(G,\sigma)$-equivariant real structure. Moreover, for $\sigma \in \{ \sigma_s,\sigma_c\}$ and $H$ one of the subgroup of $G$ listed in Remark \ref{prop:finite subgroups of SL2}, a $(G,\sigma)$-equivariant real structure on $G/H$ is given by 
\[
G/H \to G/H,\ \modH{g} \mapsto \modH{\sigma(g)}.
\]
\end{lemma}

\begin{proof}
If $H'$ is conjugate to $H$, then $G/H'$ admits a $(G,\sigma)$-equivariant real structure if and only if $G/H$ admits a $(G,\sigma)$-equivariant real structure (Remark~\ref{rk:choice of the base point}). Therefore we can assume that $H$ is one of the group listed in Remark~\ref{prop:finite subgroups of SL2}.
Also, since $\sigma$ is strongly equivalent (i.e.~conjugate by an inner automorphism) to $\sigma_s$ or $\sigma_c$, we can assume that $\sigma \in \{\sigma_s,\sigma_c\}$ by the second part of Proposition~\ref{prop:A}. 

We now check in each case that $\sigma(H)=H$, which will imply that the map $\modH{g} \mapsto \modH{\sigma(g)}$ is a $(G,\sigma)$-equivariant real structure on $G/H$ (this can be deduced from Proposition~\ref{prop:A} or even directly from the definition of an equivariant real structure).
 
The cases of $A_n$ and $D_n$ are straightforward. For $E_6$ we check that $\sigma_s(\alpha)=\begin{bmatrix}
i & 0 \\ 0 & -i
\end{bmatrix} \alpha$ and that $\sigma_c(\alpha)=\alpha$. This also implies the result for $E_7$. Finally for $E_8$ we check that $\sigma_s(\beta)=-\beta$ and that $\sigma_c(\beta)=-\beta^{-1}$, which concludes the proof.
\end{proof}

Now that we know the existence of a $(G,\sigma)$-equivariant real structure on each $G/H$, the next step is to use Propositions~\ref{prop:H1 Galois k-forms} and~\ref{prop:G-eq automorphisms of G/H} to determine the equivalence classes of all the $(G,\sigma)$-equivariant real structures on $G/H$. Since the set of such equivalence classes is in bijection with the first pointed set of Galois cohomology $\HH^1(\Gamma, N_G(H)/H)$, we need first to determine the group $N_G(H)/H$ in each case.

\begin{lemma}\label{lem:N(H)/H}
Let $G=\SL_2(\C)$, and let $H \subseteq G$ be one of the finite subgroups listed in Remark~\ref{prop:finite subgroups of SL2}. Then $N_G(H)$ and $N_G(H)/H$ are described below.  
\begin{itemize}[leftmargin=6mm]
\item If $n \in \{1,2\}$, then $N_G(A_n)=G$.  Also, $G/A_1 = G$ and $G/A_2=\PGL_2(\C)$.
\item If $n \geq 3$, then $N_G(A_n)= \Mon:= \left\{\begin{bmatrix}
* & 0 \\ 0 & *
\end{bmatrix}\right\} \cup \left\{ \begin{bmatrix}
0 & * \\ * & 0
\end{bmatrix}\right\}$ and $\Mon/A_n \simeq \Mon$.\\
$($Here $\Mon$ stands for the subgroup of monomial matrices in $G$; it is isomorphic to $\Gm\rtimes\Z/2\Z. )$
\item $N_G(D_4)=E_7$ and $E_7/D_4 \simeq \mathfrak{S}_3=\left\langle  \modH{\frac{1}{2}\begin{bmatrix}
1-i & 1-i \\ -1-i & 1+i
\end{bmatrix}},  \modH{\begin{bmatrix}
\zeta_8 & 0 \\ 0 & \zeta_8^{-1}
\end{bmatrix}}\right\rangle$.
\item If $n\geq 5$, then  $N_G(D_n)=D_{2n-2}$ and $D_{2n-2}/D_n \simeq \Z/2\Z=\left\langle \modH{\begin{bmatrix}
\zeta_{4n-8} & 0 \\ 0 & \zeta_{4n-8}^{-1}
\end{bmatrix}}\right\rangle$.
\item $N_G(E_6)=E_7$ and $E_7/E_6 \simeq \Z/2\Z=\left\langle \modH{\begin{bmatrix}
\zeta_{8} & 0 \\ 0 & \zeta_{8}^{-1}
\end{bmatrix}}\right\rangle$.
\item $N_G(E_7)=E_7$. 
\item $N_G(E_8)=E_8$. 
\end{itemize}
\end{lemma}

\begin{proof}
Remark~\ref{prop:finite subgroups of SL2} gives a list of generators for each finite subgroup $H \in\{A_n,D_n,E_n\}$ of $G$. 
Writing $g=\begin{bmatrix}
a & b \\ c & d
\end{bmatrix} \in G$, the condition $ghg^{-1} \in H$, for each generator $h$ of $H$, imposes conditions on the coefficients of $g$. In the case where $H \in \{A_n,D_n\}$, this is sufficient to determine $N_G(H)$ with very few calculations due to the simple form of the generators. (We omit details here.)

Let us now consider the case where $H=E_n$ with $n \in \{7,8,9\}$. Assume that $N_G(H)$ is infinite, then $N_G(H)$ is conjugate to a subgroup of $\Mon$ or of a Borel subgroup $B$ of $G$ (see e.g. \cite[Theorem~4.29]{vdPS03}). 
\begin{itemize}[leftmargin=*]
\item If $N_G(H)$ is conjugate to a subgroup of $\Mon$, then either $H$ is conjugate to a subgroup of $\Mon^0 \simeq \Gm$, which cannot be as $H$ is not abelian, or there is a surjective homomorphism $H \to \Z/2\Z$, induced by the exact sequence 
\[0 \to \Gm \to \Mon \to \Z/2\Z \to 0,\]
and so $H$ should contain an abelian normal subgroup of index $2$, which is not the case. 
\item If $N_G(H)$  is conjugate to a subgroup of $B$, then the generators of $H$ should have a common eigenvector which is not the case. 
\end{itemize}
Therefore $N_G(H)$ is not conjugate to a subgroup of $\Mon$ or $B$. Hence, $N_G(H)$ is a finite subgroup of $G$. But $E_6$ is a normal subgroup of $E_7$, and it is not a normal subgroup of any other finite subgroup of $G$, thus $N_G(E_6)=E_7$. Also, $E_7$ and $E_8$ are not contained in any larger finite subgroup of $G$, and so they must be self-normalized. 
\end{proof}

For each finite subgroup $H \subseteq G$, we want to determine the first Galois cohomology pointed set $\HH^1(\Gamma,N_G(H)/H)$. 
Any $1$-cocycle $c\colon \Gamma \to N_G(H)/H$ is determined by $c_\gamma \in N_G(H)/H$ since $c_1=\modH{1}$. We will denote by $[c_\gamma] \in \HH^1(\Gamma,N_G(H)/H)$ the cohomology class of the $1$-cocycle determined by $c_\gamma$. Actually, since $\Gamma \simeq \Z/2\Z $, the definition of this set simplifies a bit. 
Indeed, if $\Gamma \simeq \Z/2\Z $ and $A$ is a $\Gamma$-group, then
\[ 
\HH^1(\Gamma,A)=Z^1(\Gamma,A)/\sim, \text{\ where \ } Z^1(\Gamma,A)=\{ a \in A \ | \   a^{-1}= \gamma \cdot a \}
\] 
and $a_1$, $a _2 \in Z^1(\Gamma,A)$ satisfy $a_1 \sim a_2$ if $a_2=b^{-1} a_1 (\gamma \cdot  b)$ for some $b \in A$.

\smallskip

Before stating the next theorem, which is the main result of this subsection, we need some notation.  Recall that $e=\begin{bmatrix}
 0 & 1 \\-1 & 0 \end{bmatrix}$, let $f=\begin{bmatrix} 0 & i \\i & 0 \end{bmatrix}$, let $d=\frac{1}{\sqrt{2}}\begin{bmatrix} 1 & i \\i & 1 \end{bmatrix}$, and for $n\in\N$,  let $\w_n=\begin{bmatrix} \zeta_{n} & 0 \\0 & \zeta_{n}^{-1}\end{bmatrix}$ which are all elements of $G$. Note that, $d^{-1}=\sigma_s(d)$, $d^2=f$, and $\w_4e=-e\w_4=f$.

\begin{proposition} \label{prop:description of H1 for SL2/H}
Let $H$ be one of the finite subgroups of $G=\SL_2(\C)$ listed in Remark~\ref{prop:finite subgroups of SL2}, let $\sigma \in \{\sigma_s,\sigma_c\}$, and let 
$\Gamma=\{1,\gamma\} \simeq \Z/2\Z $ acting on $N_G(H)/H$ as follows:
\[
\forall \modH{g}=gH \in N_G(H)/H,\ \  \gamma \cdot \modH{g}:=\modH{\sigma(g)}. 
\]
Then an exhaustive list of class representatives of the finite set $\HH^1(\Gamma,N_G(H)/H)$ is given in the table below.
\end{proposition}
\setlength\extrarowheight{2pt}
\[
\begin{array}{|c|l|l|}
\hline 
\text{Subgroup } H \subseteq G & \HH^1(\Gamma,N_G(H)/H)& \HH^1(\Gamma,N_G(H)/H)\\
&\text{for }\sigma=\sigma_s&\text{for }\sigma=\sigma_c\\
\hline 
\hline 
  A_1    &  [I_2]&  [I_2],\ \  [-I_2]  \\
  \hline
  A_2    & [\modH{I_2}],\ \ [\modH{e}] & [\modH{I_2}],\ \ [\modH{e}]   \\
      \hline
A_n, n\ge 3, n \text{ odd} &[\modH{I_2}],\ \ [\modH{f}] &[\modH{I_2}],\ \ [\modH{-I_2}]\\
  \hline
 A_n, n\ge 4, n \text{ even} &[\modH{I_2}],\ \ [\modH{e}],\ \  [\modH{e\w_{2n}}]&[\modH{I_2}],\ \ [\modH{e}],\ \  [\modH{\w_{2n}}]\\
   \hline
           D_{n}, n\ge 4    &  [\modH{I_2}], \ \ [\modH{\w_{4n-8}}] &  [\modH{I_2}], \  \  [\modH{\w_{4n-8}}]  \\
       \hline
        E_6    &   [\modH{I_2}], \ \ [\modH{\w_8}]  &  [\modH{I_2}], \ \ [\modH{\w_8}]  \\
          \hline
              E_7   &  [\modH{I_2}]  &  [\modH{I_2}]  \\
                \hline
                    E_8   &  [\modH{I_2}]  &  [\modH{I_2}]  \\
                    \hline
\end{array} 
\]

\smallskip

\begin{proof}
As a preliminary remark, let us note that the finiteness of $\HH^1(\Gamma, N_G(H)/H)$ is given by Corollary~\ref{cor:finiteness form homogeneous spaces}.

First of all,  when $N_G(H)/H$ is trivial, then $\HH^1(\Gamma, N_G(H)/H))$ is reduced to one element, the class of the identity matrix, for either of the choices of $\sigma$. This finishes the cases of $E_7$ and $E_8$.

Secondly, in the cases where $Z(G)=\{ \pm I_2\} \subseteq H$ (i.e.~the cases $A_n$, with n even, $D_n$ and $E_6$), the first cohomology set for $\sigma_s$ and for $\sigma_c$ have the same number of elements (Theorem~\ref{th:B}~\ref{item part 2 th B}). 
In fact, the  map sending $[\modH{g}]$ to $[\modH{eg}]$ induces a bijection between the two cohomology sets. This is easily checked by direct computation. (One uses the facts that $e\in N_G(H)$ for each $H$,  $e^2=-I_2$, and that $\sigma(e)=e$.) 
In particular, when $e\in H$ (i.e.~in the cases $D_n$, with $n$ even, and $E_6$), the two cohomology sets coincide. Let us note that there is one case where $e\not\in H$, but $-I_2$ is in $H$, and the two sets are different (it is the case $A_n, n \geq 4,n$ even).

Thus we are left to study all the finite subgroups for $\sigma_s$, and also the subgroups $A_n$ with $n$ odd  for $\sigma_c$.

We start by treating the cases where $N_G(H)/H$ is finite. 
\begin{itemize}[leftmargin=*]
\item \textbf{Case $H=D_4$.} Let $\alpha=\frac{1}{2}\begin{bmatrix}
1-i & 1-i \\ -1-i & 1+i
\end{bmatrix}$ and $\w_8= \begin{bmatrix}
\zeta_8 & 0 \\ 0 & \zeta_8^{-1}
\end{bmatrix}$. Then $\modH{\alpha^3}=\modH{\w_8^2}=\modH{I_2}$ and $N_G(D_4)/D_4 \simeq \mathfrak{S}_3=\{\modH{I_2}, \modH{\alpha}, \modH{\alpha^2}, \modH{\w_8},\modH{\w_8 \alpha},\modH{\w_8 \alpha^2}\}$.
We check that $\sigma(\modH{\alpha})=\modH{\alpha}$ and $\sigma(\modH{\w_8})=\modH{\w_8}$, and so $\Gamma$ acts trivially on $\mathfrak{S}_3$.
Thus $Z^1(\Gamma,\mathfrak{S}_3) =\{\modH{I_2}, \modH{\w_8}, \modH{\w_8 \alpha}, \modH{\w_8 \alpha^2}\}$, and two elements of $Z^1(\Gamma,\mathfrak{S}_3)$ are equivalent if and only if they are conjugate in $\mathfrak{S}_3$. This means that $\HH^1(\Gamma,\mathfrak{S}_3) $ has two elements which are $[\modH{I_2}]$ and $[\modH{\w_8}]$. 
\smallskip

\item \textbf{Case $H=E_6$ or $H=D_n$ with $n \geq 5$.}  In both cases, the quotient group $N_G(H)/H$ is isomorphic to $\Z/2\Z =\{\modH{I_2},\modH{\delta}\}$ with $\delta=\begin{bmatrix}
\zeta & 0 \\ 0 & \zeta^{-1}
\end{bmatrix}$ for some non-trivial primitive root of unity $\zeta$. 
We check that 
$\sigma(\modH{\delta})=\modH{\delta}$, and so $\Gamma$ acts trivially on $\Z/2\Z $.
Thus, in both cases for $\sigma$, we have 
\[
\HH^1(\Gamma,N_G(H)/H) \simeq Z^1(\Gamma,N_G(H)/H) \simeq N_G(H)/H=\{ \modH{I_2}, \modH{\delta}\}.
\]
This yields 
\[
\HH^1(\Gamma,N_G(H)/H)=\left\{
    \begin{array}{ll}
        \{[\modH{I_2}], [\modH{\w_8}]\} &\text{ when } H=E_6; \text{ and} \\
        \{[\modH{I_2}], [\modH{\w_{4n-8}}]\} &\text{ when } H=D_n \text{ with } n\geq 5.
    \end{array}
\right.
\]
\end{itemize}
\smallskip

It remains to treat the cases where $N_G(H)/H$ is infinite, i.e.~\textbf{the case $H=A_n$.}

We check that for $\sigma=\sigma_s$ and $n \geq 3$, the set of $1$-cocycles $Z^1(\Gamma,N_G(A_n)/A_n)$ is 
\[
\left\{\modH{\begin{bmatrix}
t & 0 \\ 0 & t^{-1}
\end{bmatrix}}; \ |t|=1 \right\} 
\sqcup \left\{ \modH{\begin{bmatrix}
0 & is \\ is^{-1} & 0
\end{bmatrix}}; \  s\in \R^* \right\}
\sqcup \left\{ \modH{\begin{bmatrix}
0 & is\zeta_{2n} \\ is^{-1}\zeta_{2n}^{-1} & 0
\end{bmatrix}}; \ s \in \R^* \right\}.
\]
Note that, when $n$ is odd, then the two last subsets are equal.  This is not the case for $n$ even. 
 To determine the cohomology set, we note that any element of $\left\{\modH{\begin{bmatrix}
t & 0 \\ 0 & t^{-1}
\end{bmatrix}};
 \ |t|=1 \right\}$ is equivalent  to $\modH{I_2}$, and any element of $\left\{ \modH{\begin{bmatrix}
0 & is \\ is^{-1} & 0
\end{bmatrix} }; \ s \in \R^* \right\}$  is equivalent  to $\modH{f}$. 
One checks also that these two  elements are  inequivalent. 
This means in particular that when $n$ is odd,  $\HH^1(\Gamma,N_G(A_n)/A_n)$ has  exactly two  elements when $\sigma=\sigma_s$ and $n \geq 3$. 
 When $n$ is even, there is a third equivalence class, coming the cycles of the form $\left\{ \modH{\begin{bmatrix}
0 & is\zeta_{2n} \\ is^{-1}\zeta_{2n}^{-1} & 0
\end{bmatrix}}; \ s \in \R^* \right\}$.
 
 Note that, when $n$ is even, the table indicates other representatives of the cohomology set.  Here we explain the choice in the table, from which the case of $\sigma=\sigma_c$ can be easily deduced.

\begin{itemize}[leftmargin=6mm]
\item \textbf{Case $H=A_n$ with $n\equiv 0 \,[4]$.} 
In this case, $\modH{e}=\modH{f}$, since $\w_4\in H$. Thus, for $\sigma=\sigma_s$, we have $\HH^1(\Gamma,N_G(A_n)/A_n)=\{[\modH{I_2}],[\modH{e}], [\modH{ew_{2n}}]\}$. 
Then the set $\HH^1(\Gamma,N_G(A_n)/A_n)$  for $\sigma=\sigma_c$ is obtained by multiplication by $\modH{e}$, and we get the result for $\sigma=\sigma_c$.
\smallskip
\item \textbf{Case $H=A_n$ with $n\equiv 2 \, [4] $ and $n\ge 6$.} 
In this case $\modH{e}=\modH{f\w_{2n}}$ and $\modH{e\w_{2n}}=\modH{f}$, and so the result is the same as in the case $n\equiv 0 \,[4]$. 
\smallskip
\item \textbf{Case $H=A_n$, with $n$ is odd and $n\ge 3$.}
We have already treated the case where $\sigma=\sigma_s$. Now we consider the case where $\sigma=\sigma_c$. Since $-I_2\not\in H$, we cannot deduce the result directly from that of $\sigma_s$. Here we find
\[
Z^1(\Gamma,N_G(A_n)/A_n)=
\left\{\modH{\begin{bmatrix}
t & 0 \\ 0 & t^{-1}
\end{bmatrix};} \ t\in\R^* \right\} .
\]
If $t> 0$, then $\modH{\begin{bmatrix}
t & 0 \\ 0 & t^{-1}
\end{bmatrix}}$ is equivalent to $\modH{I_2}$, and if $t<0$, then it is equivalent to $\modH{-I_2}$.  Moreover, $\modH{I_2}$ and $\modH{-I_2}$ are inequivalent. Thus $\HH^1(\Gamma,N_G(A_n)/A_n)=\{[\modH{I_2}],[\modH{-I_2}]\}$.
\end{itemize}

\smallskip

We are left with the final two cases, where $H=A_1$ or $A_2$. Direct calculations in these cases are more difficult, since $N_G(H)=G$. We will therefore use instead the well-known result on the number of distinct real group structures on $G$. 

\begin{itemize}[leftmargin=6mm]
\item \textbf{Case $H=A_2$.} In this case, for $\sigma \in \{\sigma_s,\sigma_c\}$, we have 
$\HH^1(\Gamma,N_G(A_2)/A_2)=\HH^1(\Gamma,G/Z(G))$  
which parametrizes the strong equivalence classes of inner twists of $\sigma$ by Lemma~\ref{lem:Z1 parametrizes inner twists}. But we already know that there are two such classes, namely those of $\sigma_s$ and $\sigma_c$, corresponding to $[\modH{I_2}]$ and $[\modH{e}]$ since $\sigma_c=\inn_e \circ \sigma_s$. 
\smallskip
\item \textbf{Case $H=A_1$.} We will use the previous result. Suppose the $g\in Z^1(\Gamma, G) \subseteq G$. Then the class modulo $A_2=Z(G)$ will be a $1$-cocycle in $G/Z(G)$. 
Thus, $[\modH{g}]=[\modH{I_2}]$ or $[\modH{g}]=[\modH{e}]$ in $\HH^1(\Gamma,G/Z(G))$ by the previous case.
But $e$ and $-e$ are not contained in $Z^1(\Gamma, G)$, hence $g$ must be equivalent to $I_2$ or $-I_2$ in $\HH^1(\Gamma,G)$.

If $\sigma=\sigma_s$, then $I_2$ and $-I_2$ are equivalent. Indeed, if $ b=\begin{bmatrix} i&0\\ 0 &-i\end{bmatrix}$, then $b^{-1}\sigma_s(b)=-I_2$. Hence $\HH^1(\Gamma,G)=\{[I_2]\}$.

If  $\sigma=\sigma_c$, then  $I_2$ and $-I_2$ are inequivalent. This can be seen for instance by checking the fixed locus of the corresponding involutions. We find that the involution $G \to G,\ g\mapsto-\sigma_c(g)$, corresponding to the class $[-I_2]$, has no fixed points. Hence $\HH^1(\Gamma,G)=\{[I_2],[-I_2]\}$.
\end{itemize}
\end{proof}

\begin{remark}\label{rk:rk sec 3.2}
The case $H=A_1$ provides an example where a homogeneous space $G/H$ admits a $(k,F)$-form and a $(k,F_c)$-form (with $F=G/\left\langle \sigma_s \right \rangle$ and $F_c=G/\left\langle \sigma_c \right \rangle$) but for which the number of isomorphism classes of $(k,F)$-forms is not equal to the number of isomorphism classes of $(k,F_c)$-forms, even though $\sigma_c$ is an inner twist of $\sigma_s$.  
\end{remark}

\smallskip

\noindent \emph{Proof of Theorem~\ref{th:D}:} 
Since $\sigma$ is strongly equivalent to either $\sigma_s$ or $\sigma_c$, we can assume that $\sigma \in \{\sigma_s,\sigma_c\}$ by the second part of Proposition~\ref{prop:A}.
Moreover, by Remark~\ref{prop:finite subgroups of SL2}, it suffices to consider the cases where $H \in \{A_n,D_n,E_n\}$.
The list of $(G,\sigma)$-equivariant real structures on $X=G/H$, which can be found in Appendix~\ref{ap:table2}, follows then from Proposition~\ref{prop:description of H1 for SL2/H} together with Proposition~\ref{prop:H1 Galois k-forms}. (We recall that for homogeneous spaces any descent datum is effective.)
It remains to determine the corresponding real locus in each case.

\smallskip

$\bullet$ Let us start with the case $\sigma=\sigma_c$.
First of all, notice that in all cases $H\subseteq \SU_2(\C)$, and the elements $-I_2$, $e$ and $\w_m$, for all $m \geq 1$, are also in $\SU_2(\C)$.  Let $\mu$ be one of the $(G,\sigma_c)$-equivariant real structures on $X$ listed in the table of Appendix~\ref{ap:table2}; it is of the form $\mu(\modH{g})=\modH{\sigma_c(g)t}$, where $t \in \{\pm I_2, e, \w_m\}$. Then $\modH{g}$ belongs to the real locus of $\mu$ if and only if $g^{-1} \sigma_c(g) \in H t^{-1}$. In particular, we must have $g^{-1} \sigma_c(g)  \in \SU_2(\C)$.

But for any $g\in G$, the element  $g^{-1}\sigma_c(g)$ is a matrix of the form 
$\begin{bmatrix}
r_1 & z \\ \overline{z} & r_2
\end{bmatrix}$, where $r_1, r_2\in\R_+$, and $\overline{z}$ denotes the conjugate of the complex number $z$ (not to be confused with the notation $\modH{g}=gH$ for an element $g \in G$). 
But the only such element which is in   $\SU_2(\C)$  is the identity matrix. 
Thus $g^{-1}\sigma_c(g)\in SU_2(\C)$ if and only if $g^{-1}\sigma_c(g)=I_2$.
This implies that the real locus of $\mu$ is 
\[
\left\{
    \begin{array}{cl}
        \SU_2(\C)/H & \text{ if } t \in H; \text{ and} \\
        \varnothing & \text{ if } t \not\in H.
    \end{array}
\right.
\]
We can therefore complete the column of the table in Appendix~\ref{ap:table2} corresponding to $\sigma=\sigma_c$.

\smallskip

$\bullet$ Let us now consider the case $\sigma=\sigma_s$.
The cases $H=A_1$ and $H=A_2$ are quite easy to deal with. Let us give details when $H=A_n$ with $n \geq 3$ and $n$ odd.
Write $g=\begin{bmatrix}
a & b \\ c & d
\end{bmatrix}$. We have $\modH{g}=\modH{\sigma_s(g)}$ if and only if, for some $n$-th root of unity $\xi=e^{i2\pi\frac{k}{n}}$ ($0 \leq k \leq n-1$), we have
\[
\begin{bmatrix}
a & b \\ c & d
\end{bmatrix}=
\begin{bmatrix}
\overline{a} & \overline{b} \\ \overline{c} & \overline{d}
\end{bmatrix} \begin{bmatrix}
\xi & 0 \\ 0 & \xi^{-1}
\end{bmatrix}
=
\begin{bmatrix}
\xi \overline{a} & \xi^{-1}\overline{b} \\ \xi\overline{c} & \xi^{-1}\overline{d}
\end{bmatrix} \Leftrightarrow 
\begin{bmatrix}
a & b \\ c & d
\end{bmatrix} \in \SL_2(\R) \begin{bmatrix}
\xi_0 & 0 \\ 0 & \xi_{0}^{-1}
\end{bmatrix},
\] 
where $\xi_0=e^{i\pi\frac{k}{n}}$ is a $2n$-th root of unity. 
Therefore $\modH{g}=\modH{\sigma_s(g)}$ if and only if 
$\modH{g} \in \SL_2(\R)\modH{I_2} \sqcup  \SL_2(\R)\modH{\w_{2n}}$.

We now consider the equality $\modH{g}=\modH{-\sigma_s(g)f}$ i.e.~$-g^{-1}\sigma_s(g)f \in A_n$. A direct computation yields that $-g^{-1}\sigma_s(g)f \in A_n$ if and only if $g^{-1}\sigma_s(g)=f$.
Denote $g'=gd$, where $d$ is the matrix defined before Proposition~\ref{prop:description of H1 for SL2/H}. Then
\small
\[
gf=\sigma_s(g) \Leftrightarrow g'd^{-1}f=\sigma_s(g')\sigma_s(d^{-1})  
\Leftrightarrow g'd=\sigma_s(g')d \Leftrightarrow g'=\sigma_s(g') \Leftrightarrow g' \in \SL_2(\R).
\]
\normalsize
Therefore $g \in \SL_2(\R) d^{-1} $, and so the real locus in $G/H$ is $\SL_2(\R) \modH{d^{-1}}$.

It remains to consider the equality $\modH{g}=\modH{-\sigma_s(g)f\w_{2n}}$. 
But a d direct computation yields that the equation $-g^{-1}\sigma_s(g)f\w_{2n} \in A_n$ 
if and only if $g^{-1}\sigma_s(g)=-f$. 
Then
\[
gf=-\sigma_s(g) 
\Leftrightarrow gd=-\sigma_s(gd)
\Leftrightarrow gd \in \SL_2(\R) \w_4 \Leftrightarrow g \in \SL_2(\R) \w_4 d^{-1}.
\]
But $\w_4 d^{-2}=-\w_4 f=e \in \SL_2(\R)$, hence $\SL_2(\R) \w_4 d^{-1}=\SL_2(\R)d$, and so the real locus in $G/H$ is $\SL_2(\R) \modH{d}$.

\smallskip

$\bullet$   To determine the real locus in the other cases when $H \in \{A_n,D_n\}$, the approach is similar using the special form of the elements of $H$ (namely diagonal and antidiagonal elements).

\smallskip

$\bullet$  In the case $H=E_6$ and $\mu(\modH{g})=\modH{\sigma_s(g)}$, we check that $g^{-1} \sigma_s(g) \in E_6$ if and only if $g^{-1} \sigma_s(g) \in D_4$. Hence, the corresponding real locus is the image of the real locus obtained for $H=D_4$ through the quotient map $G/D_4 \to G/E_6$. But the three components of the real locus of $\mu$ when $H=D_4$ are permuted by the action of $E_6/D_4 \simeq \Z/3\Z$, which gives the single component $\PSL_2(\R)\modH{I_2}$ for the real locus of $\mu$ when $H=E_6$.

For the last three cases with $\sigma=\sigma_s$ and $H \in \{E_6,E_7,E_8\}$, the real locus is determined by computer calculations using the software system \emph{SageMath} \cite{sage}. We omit here the details of these calculations.
\qed

\subsection{Real forms in the almost homogeneous case}\label{sec:almost homogeneous SL2 threefolds}
Fix a finite subgroup $H \subseteq G=\SL_2(\C)$, and let $B=\begin{bmatrix}
* & 0 \\ * &* 
\end{bmatrix}$ which is a Borel subgroup of $G$. The description of the combinatorial data to classify the $G$-equivariant embeddings of $G/H$ is detailed in Appendix~\ref{sec:appendix B}. 
In this subsection we make explicit the $\Gamma$-actions introduced in \S~\ref{sec:Gamma action on colored data}, on the colored equipment $(\V^G(G/H),\DD^B(G/H))$ of $G/H$,  and then specialize the Luna-Vust theory in this setting (Theorem~\ref{th:E}).

\smallskip

Let $\sigma \in \{\sigma_s,\sigma_c\}$, and let $\mu$ be a $(G,\sigma)$-equivariant real structure on $G/H$. 
To avoid a possible confusion with the complex conjugation, we will use the notation $\mu(gH)=\sigma(g)tH$, where  $t$ is the appropriate element of $G$, as indicated in the table of Appendix~\ref{ap:table2}. 
Then the Galois group $\Gamma=\{1,\gamma\}$ acts on $\C(G/H)$ as follows (here $\overline{\alpha}$ indicates the complex conjugate of $\alpha$):
\[
\forall f \in \C(G/H),\ \forall gH \in \mathrm{Def}(f),\ (\gamma \cdot f)(gH)=\overline{f (\sigma(g)tH)}.
\]

Recall that the set of colors $\DD^B(G/H)$ is identified with $\P^1/H$ (see \S~\ref{sec:B1}). The $\Gamma$-action on $\DD^B(G/H) \simeq \P^1/H$ is given by  
\[
 \left\{
    \begin{array}{ll}
\,        [\alpha:\beta]\cdot H \mapsto [\overline\alpha:\overline\beta]\cdot tH & \text{ if } \sigma=\sigma_s; \text{ and } \\
\,         [\alpha:\beta]\cdot H \mapsto [-\overline\beta:\overline\alpha]  \cdot tH &  \text{ if } \sigma=\sigma_c.
    \end{array}
\right.
\]
Finally, the $\Gamma$-action on $\V^G(G/H)$ is given by
\[
\forall j \in \DD^B(G/H) \simeq \P^1/H,\ \forall \nu(j,r)\in \V^G(G/H),\ \gamma \cdot \nu(j,r)=\nu(\gamma \cdot j,r).
\]
In particular, it should be noted that the $\Gamma$-action preserves the type of the $G$-orbits.

The $\Gamma$-action can be visualized on the "skeleton diagrams" introduced in Appendix~\ref{sec:appendix B} (see Figure~\ref{Vg diagram}). Indeed, for any $j\in\DD^B(G/H)$, consider the set of valuations of $\Vg$ of the form $\nu(j,r)$. Call this set the \emph{spoke} of $j$ in the diagram. Then the $\Gamma$-action  corresponds to a permutation of the spokes according to the $\Gamma$-action on $\DD^B(G/H)$. 
In particular, two spokes exchanged by $\Gamma$ must be of the same length.

\begin{remark}\label{rk:spherical behavior}
Let us note that, contrary to the case of spherical homogeneous spaces (see \cite[\S~2]{Hur11}), the $\Gamma$-action on the colored equipment of $X_0=G/H$ depends not only on $\sigma$, but also on $\mu$. In fact it is even possible, for a given $\sigma$, to have two equivalent $(G,\sigma)$-equivariant real structures on $X_0$ such that only one of them extends to a given $G$-equivariant embedding of $X_0$.
\end{remark}

The next result is mostly a consequence of Theorem~\ref{th:C} specialized in the case of almost homogeneous $\SL_2$-threefolds.

\begin{theorem}\label{th:E}
Let $H$ be a finite subgroup of $G=\SL_2(\C)$, let $\sigma$ be a real group structure on $G$ corresponding to the real form $F$ of $G$, and
let $\mu$ be a $(G,\sigma)$-equivariant real structure on $X_0=G/H$.
Let $X_0 \hookrightarrow X$ be a $G$-equivariant embedding of $X_0$. 
Then $\mu$ extends to an effective $(G,\sigma)$-equivariant real structure $\tilde{\mu}$ on $X$, which corresponds to an $(\R,F)$-form of $X$ through the map $X \mapsto X/\left\langle \tilde{\mu} \right \rangle$, if and only if
\begin{enumerate}[leftmargin=10mm]
\item\label{lab:item i for SL2 th} the  $\Gamma$-actions on $\V^G(G/H)$ and $\DD^B(G/H)$ induced by $\mu$ preserves the collection of colored data of the $G$-orbits of $X$; and
\item\label{lab:item ii for SL2 th} every $G$-orbit of $X$ of type $\BB_0$ or $\BB_-$ is stabilized by the $\Gamma$-action.
\end{enumerate}
\end{theorem}

\begin{proof}
According to Theorem~\ref{th:C}, the $(G,\sigma)$-equivariant real structure $\mu$ on $X_0$ extends to an effective $(G,\sigma)$-equivariant real structure $\tilde{\mu}$ on $X$ if and only if the collection of colored data of the $G$-equivariant embedding $X_0 \hookrightarrow X$ is $\Gamma$-stable, which is the condition~\ref{lab:item i for SL2 th} of the theorem, and $X$ is covered by $\Gamma$-stable affine open subsets.

Hence, it suffices to show that $X$ is covered by $\Gamma$-stable affine open subsets if and only if every $G$-orbit of $X$ of type $\BB_0$ or $\BB_-$ is stabilized by the $\Gamma$-action. 
For this we will use the following result: an equivariant embedding of $X_0$ is quasiprojective if and only if it has at most one $G$-orbit of type $\BB_0$ or $\BB_-$; see \cite[\S~6.7.1, Proposition~13]{Bou00}.

Suppose that $X$ satisfies the condition~\ref{lab:item ii for SL2 th}. 
Then, according to \S~\ref{sec:B4},  the variety $X$ is covered by $\Gamma$-stable quasiprojective open $G$-subvarieties, and thus by $\Gamma$-stable affine open subsets.

We now suppose that $X$ has a $G$-orbit $Y$ of type $\BB_0$ or $\BB_-$ which is not $\Gamma$-stable.
Let $\nu \in \V^G(G/H)$ be the valuation such that $\V_Y^G=\{ \nu \}$. It follows from \S~\ref{sec:B4} that $X$ contains a $G$-stable dense open subset $X'$ formed by the union of five $G$-orbits: the dense open orbit $X_0$, two orbits of type $\CC$ with colored data $(\{\nu\},\varnothing)$ and $(\{\gamma \cdot \nu\}, \varnothing)$, and the two orbits $Y$ and $\gamma \cdot Y$. The open subset $X'$ is not quasiprojective as it contains two orbits of type $\BB_0$ or $\BB_-$. 
But any $\Gamma$-stable  open $G$-subvariety of $X$ which contains $Y$ must contain $X'$. Thus the condition~\ref{label:effectiveness descent datum2} of Proposition~\ref{prop:two conditions existence of k forms for models} does not hold, and therefore $X$ is not covered by $\Gamma$-stable affine open subsets.
\end{proof}

\begin{example} \label{ex:baby example}
Let $X=\P^2\times\P^1=\P(R_0\oplus R_1)\times\P(R_1)$, where $R_i$  is the irreducible representation of dimension $i+1$ for $G=SL_2(\C)$. We consider the $G$-equivariant embedding of $G$ given by
\[
G \to \P^2 \times \P^1,\ 
\begin{bmatrix}
a & b \\ c &d
\end{bmatrix} \mapsto ([1:a:c],[b:d]).
\]
The orbit decomposition of $X$ is $\ell_1 \sqcup \ell_2 \sqcup S_1 \sqcup S_2 \sqcup X_0$, where the five orbits are the following (see \S~\ref{sec:B3}  for the types of the orbits and see Figure~\ref{Diagram for first example} for the diagram associated with this $G$-equivariant embedding):
\begin{itemize}[leftmargin=7mm]
\item $\ell_1=\{([0:u_1:u_2],[u_1:u_2])\}\simeq \P^1$ of type $\AA_2([0:1],[1:0],1,0)$.
\item $\ell_2=\{([1:0:0],[w_1:w_2])\}\simeq \P^1$ of type $\BB_+([1:0],0)$.
\item $S_1=\{([0:u_1:u_2],[w_1:w_2])\}$, with $[u_1:u_2]\not=[w_1:w_2]$, of type $\CC([0:1],1)$. This orbit is isomorphic to $\P^1\times\P^1\setminus\Delta$, where $\Delta$ is the diagonal. 
\item $S_2=\{([1:\alpha:\beta],[u_1:u_2])\}$ with $[\alpha:\beta]=[w_1:w_2]$, of type $\CC([1:0],0)$. This orbit is isomorphic to $\A^2\setminus\{0\}$. 
\item The dense open orbit $X_0 \simeq G$, which is the complement of all the other orbits. 
\end{itemize}

The $\Gamma$-action induced by $\pm\sigma_c$ on the colored equipment of $G$ does not preserve the collection of colored data of $X$, and thus the $(G,\sigma_c)$-equivariant real structures $\pm \sigma_c$ on $G$  do not extend to $X$. (Actually, one can check that $X$ admits no $(G,\sigma_c)$-equivariant real structure.)

On the other hand, the $\Gamma$-action induced by $\sigma_s$ fixes the colored data of each orbit of $X$, and so the $(G,\sigma_s)$-equivariant real structure $\sigma_s$ on $G$ extends to $X$. Moreover, since there is no orbit of type $\BB_0$ or $\BB_-$, this real structure on $X$ is effective and corresponds therefore to a real form of $X$. 
In fact, using \cite[Proposition~1.18]{MJT21} together with the fact that
\small 
\[ 
\mathbb{G}_{m,\C} \xrightarrow{\sim}  \Aut_{\C}^{G}(X) \subseteq \Aut_\C(X) \simeq \PGL_3(\C) \times \PGL_2(\C),\ \alpha \mapsto 
\left\{ \left( \begin{bmatrix}
\alpha^{-2} & 0 & 0 \\ 0 &\alpha & 0 \\ 0 & 0 & \alpha
\end{bmatrix},I_2 \right) \right \},  
\]
\normalsize
we check that the extension of $\sigma_s$ to $X$ is the only $(G,\sigma_s)$-equivariant real structure on $X$ up to equivalence. Its real locus is $\P_\R^2 \times \P_\R^1$.
\end{example}

\begin{figure}[ht]
\begin{tikzpicture}[scale=1.2]
\path (0,0) coordinate (origin);
\draw (0,0) circle (2pt) ;
\path (90:1.5cm) coordinate (P0);
\path (135:1.5cm) coordinate (P1);
\path (200:1.5cm) coordinate (P2);
\path (300:1.5cm) coordinate (P3);
\path (30:1.5cm) coordinate (P4);
\path (90:.8cm) coordinate (Q0);
\draw[line width=0.3mm] (Q0) -- +(180:2pt)  (Q0)-- +(0:2pt); ;
\draw[line width=0.5mm] (origin) -- (P0) (origin) -- (P1) (origin) -- (P2)(origin) -- (P3) (origin) -- (P4);
\fill[white] (0,0) circle (1.5pt) ;
\node at (-1,.4){$\vdots$};
\node at (-.3,.8){$0$};
\node at (-.2,1.3){$+$};
\draw[line width=0.3mm] (P4) -- +(120:3pt)  (P4)-- +(330:3pt); 
\node at (1.5,1){$1$};

\end{tikzpicture}
\caption{Diagram for Example~\ref{ex:baby example}.}
\label{Diagram for first example}
\end{figure}
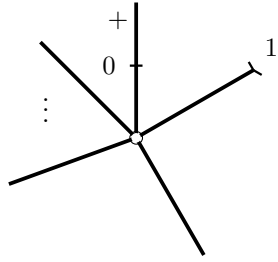

\begin{example} \label{ex:structure extends but the quotient is not a scheme}
Let $X$ be the $G$-equivariant embedding of $G$ with the following five $G$-orbits (see Figure~\ref{Diagram for second example} for the corresponding diagram): 
\begin{itemize}[leftmargin=7mm]
\item the dense open orbit $X_0 \simeq G$;
\item two orbits of type $\CC$, one associated with the valuation $\nu_1=\nu([1:0],\frac{1}{2})$ and the other with the valuation $\nu_2=\nu([0:1],\frac{1}{2})$; and
\item two orbits of type $B_-$, one with $\V^G_Y=\{\nu_1\}$ and the other with $\V^G_Y=\{\nu_2\}$. 
\end{itemize}
This variety is smooth (as follows from \cite[Theorem~3.1]{MJ90}) but not quasiprojective as it contains two orbits of type $B_-$.

The collection of colored data of $X$ is stable for the $\Gamma$-action induced by $\sigma_s$ or $\pm\sigma_c$. For $\mu=\sigma_s$, each orbit of type $B_-$ is fixed, and so the $(G,\sigma_s)$-equivariant real structure $\mu$ is effective and the corresponding quotient $X/\Gamma$
is a real variety. On the other hand, for $\mu=\pm\sigma_c$, the two orbits of type $B_-$ are permuted. Thus the corresponding quotient $X/\Gamma$, which exists as a real algebraic space, is not a real variety. Therefore $X$ admits a real form if and only if $\mu=\sigma_s$. 
\end{example}

\begin{figure}[ht]
\begin{tikzpicture}[scale=1.2]
\path (0,0) coordinate (origin);
\draw (0,0) circle (2pt) ;
\path (90:1.5cm) coordinate (P0);
\path (135:1.5cm) coordinate (P1);
\path (200:1.5cm) coordinate (P2);
\path (300:1.5cm) coordinate (P3);
\path (30:1.5cm) coordinate (P4);
\path (30:1.1cm) coordinate (P5);
\path (30:.8cm) coordinate (P6);

\path (90:.8cm) coordinate (Q0);
\path (90: 1.1cm) coordinate (Q1);
\draw[line width=0.3mm] (Q0) -- +(180:2pt)  (Q0)-- +(0:2pt); 
\draw[line width=0.3mm] (Q1) -- +(180:2pt)  (Q1)-- +(0:2pt); 
\draw(origin) -- (P0) (origin) -- (P1) (origin) -- (P2)(origin) -- (P3) (origin) -- (P4);
\draw[line width=0.5mm] (Q1) -- (P0)  (P5) -- (P4);
\fill[white] (0,0) circle (1.5pt) ;
\node at (-1,.4){$\vdots$};
\node at (-.3,.8){\tiny{$0$}};
\node at (-.2,1.3){\tiny{$-$}};
\node at (.3,1.1){\tiny{$\frac{1}{2}$}};
\draw[line width=0.3mm] (P6) -- +(120:3pt)  (P6)-- +(330:3pt); 
\draw[line width=0.3mm] (P5) -- +(120:3pt)  (P5)-- +(330:3pt); 
\node at (.6,.7){\tiny{$0$}};
\node at (1,.9){\tiny{$-$}};
\node at (1.1,.3){\tiny{$\frac{1}{2}$}};


\end{tikzpicture}
\caption{Diagram for Example~\ref{ex:structure extends but the quotient is not a scheme}.}
\label{Diagram for second example}
\end{figure}
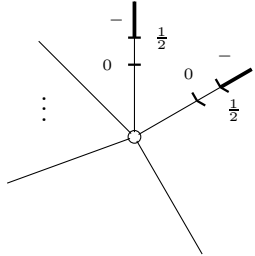

\begin{remark}
In the same vein as Example~\ref{ex:structure extends but the quotient is not a scheme}, we refer to \cite{LMV89} for the construction of examples of three-dimensional Moishezon manifolds which are not schemes and on which $G$ acts with a dense open orbit.
\end{remark}

\begin{example}\label{ex:first bis example}
Let $X=\P^1\times\P^1 \times \P^1=\P(R_1)\times\P(R_1) \times \P(R_1)$. Then the stabilizer of the point $x=([1:1],[1:0],[0:1])$ is $H=\{ \pm I_2 \}=A_2$, and so $(X,x)$ is a $G$-equivariant embedding of $G/H=\PGL_2(\C)$. 
The orbit decomposition of $X$ is $\ell \sqcup S_1 \sqcup S_2 \sqcup S_3 \sqcup X_0$, where $X_0 \simeq \PGL_2(\C)$ is the dense open orbit, the $S_i \simeq \P^1 \times \P^1 \setminus \Delta$ are of type $\CC([1:1],1),\CC([1:0],1), \CC([0:1],1)$ respectively, and $\ell \simeq \P^1$ is of type $\AA_3([1:1],[1:0],[0:1],1,1,1)$; see Figure~\ref{Diagram for first-bis example} for the corresponding diagram.

Let us note that
\[
\mathfrak{S}_3 \simeq \Aut_\C^G(X) \hookrightarrow \Aut_\C^G(X_0) \simeq \PGL_2(\C), \ (12) \mapsto \begin{bmatrix}
 i & 0 \\i & -i
\end{bmatrix} \ \text{and}  \ (23) \mapsto \begin{bmatrix}
0 & i \\ i & 0
\end{bmatrix},
\]
where the symmetric group $\mathfrak{S}_3$ acts on $X$ by permuting the three factors. Then, using Proposition~\ref{prop:H1 Galois k-forms}, a direct computation of $\HH^1(\Gamma,\Aut_\C^G(X))$ yields that $X$ admits exactly two equivalence classes of $(G,\sigma)$-equivariant real structures for each $\sigma \in \{\sigma_s,\sigma_c\}$. 

Let us first consider the case $\sigma=\sigma_s$. The $\Gamma$-action induced by the equivariant real structure $\mu_1\colon gH \mapsto \sigma_s(g)eH$ does not preserve the collection of colored data of $X$, and so $\mu_1$ does not extend to $X$. On the other hand, the $\Gamma$-action induced by the equivariant real structure $\mu_2\colon gH \mapsto \sigma_s(g)H$ stabilizes the colored data of each orbit of $X$, and so it extends to an equivariant real structure $\widetilde{\mu_2}$ on $X$. Moreover, the equivariant real structure $\mu_3\colon gH \mapsto \sigma_s(g)fH$, which is equivalent to $\mu_2$, also extends to an equivariant real structure $\widetilde{\mu_3}$ on $X$, but $\widetilde{\mu_2}$ and $\widetilde{\mu_3}$ are inequivalent as real structures on the  $G$-variety $X$. Hence  $\widetilde{\mu_2}$ and $\widetilde{\mu_3}$ are the two inequivalent equivariant real structures on $X$. Also, the real locus of $\widetilde{\mu_2}$ is $\P_\R^1 \times \P_\R^1 \times \P_\R^1$ and the real locus of $\widetilde{\mu_3}$ is $\P_\R^1 \times \mathbb{S} \simeq \P_\R^1 \times \P_\C^1$, where $\mathbb{S}=\{ ([u_0:u_1],[\overline{u_0}:\overline{u_1
}]) \} \subseteq \P_\C^1 \times \P_\C^1$.

The situation is very similar when $\sigma=\sigma_c$. Indeed, we check that among the two inequivalent equivariant real structures $\mu_4\colon gH \mapsto \sigma_c(g)H$ and $\mu_5\colon gH \mapsto \sigma_c(g)eH$ on $G/H$, only $\mu_5$ extends to an equivariant real structure $\widetilde{\mu_5}$ on $X$, that $\mu_6\colon gH \mapsto \sigma_c(g)efH$, which  is equivalent to $\mu_5$, also extends to an equivariant real structure $\widetilde{\mu_6}$ on $X$, but that $\widetilde{\mu_5}$ and $\widetilde{\mu_6}$ are inequivalent. 
Hence  $\widetilde{\mu_5}$ and $\widetilde{\mu_6}$ are the two inequivalent equivariant real structures on $X$. Also, the corresponding real loci are both empty. Indeed, the real locus of a complex smooth variety is either empty or Zariski dense (see e.g.~\cite[Corollary~2.2.10]{Man20}).
\end{example}

\begin{figure}[ht]
\begin{tikzpicture}[scale=1.2]
\path (0,0) coordinate (origin);
\draw (0,0) circle (2pt) ;
\path (90:1.5cm) coordinate (P0);
\path (135:1.5cm) coordinate (P1);
\path (200:1.5cm) coordinate (P2);
\path (300:1.5cm) coordinate (P3);
\path (30:1.5cm) coordinate (P4);
\path (90:.8cm) coordinate (Q0);
\draw [line width=0.5mm] (origin) -- (P0) (origin) -- (P1) (origin) -- (P2)(origin) -- (P3) (origin) -- (P4);
\fill[white] (0,0) circle (1.5pt) ;
\node at (-1,.4){$\vdots$};
\draw[line width=0.3mm] (P4) -- +(120:3pt)  (P4)-- +(330:3pt);
\draw[line width=0.3mm] (P0) -- +(0:3pt)  (P0)-- +(180:3pt);
\draw[line width=0.3mm] (P3) -- +(30:3pt)  (P3)-- +(210:3pt);
\end{tikzpicture}
\caption{Diagram for Example~\ref{ex:first bis example}.}
\label{Diagram for first-bis example}
\end{figure}
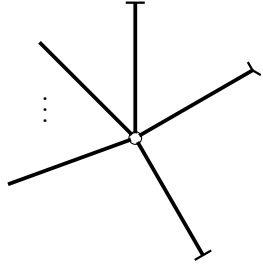

\begin{remark}\label{rk:Gamma map}
Let $X$ be a $G$-equivariant embedding of some $X_0=G/H$. 
Then any $(G,\sigma)$-equivariant real structure on $X$ restricts to a $(G,\sigma)$-equivariant real structure on the dense open orbit $X_0$. This restriction map induces a map
\begin{align*}
\Upsilon \colon &\{\text{equivalence classes of $(G,\sigma)$-equivariant real structure on $X$}\} \\
&\to  \{\text{equivalence classes of $(G,\sigma)$-equivariant real structure on $X_0$}\}.
\end{align*}
The previous examples show that the map $\Upsilon$ is generally neither injective nor surjective.
However, it is injective for instance when $\Aut_{\C}^{G}(X) \simeq \Aut_{\C}^{G}(X_0)$.
\end{remark}

\subsection{Real forms of minimal smooth completions of \texorpdfstring{$G/H$}{G/H}} \label{sec:minimal models}
In this subsection we follow the strategy given in \S~\ref{sec:strategy} to determine the real forms of the minimal smooth completions of $X_0=G/H$ when $H\subseteq G=\SL_2(\C)$ is non-cyclic, i.e.~conjugate to $D_n$ or $E_n$ (Corollary~\ref{cor:F}).

\smallskip

Here we call \emph{minimal smooth completion} of $X_0$ a $G$-equivariant embedding $X_0 \hookrightarrow X$ such that $X$ is a smooth complete variety and any $G$-equivariant birational morphism $X \to X'$, with $X'$ smooth, is an isomorphism. 
(Also, to simplify the situation, we identify the minimal smooth completions $X_0 \hookrightarrow X$ of $X_0$ that coincide up to an element of $\Aut_\C^G(X_0)$.
In other words, we consider two embeddings of $X_0$ equivalent if the underlying varieties are isomorphic as $G$-varieties.)

The minimal smooth completions of $X_0$ when $H$ is a finite subgroup of $G$ are well-known.
They were classified and studied by Mukai-Umemura in \cite{MU83} (for $H$ conjugate to $E_7$ or $E_8$), by Umemura in \cite[\S~4]{Ume88} (for $H$ conjugate to $D_n$), and by Nakano in \cite{Nak89} (for all $H$, but under projectivity assumption) as an application of Mori theory. Finally, the full classification (with the non-projective cases for $H$ conjugate to $A_n$) and the description of the corresponding colored data were obtained by Moser-Jauslin and Bousquet in \cite{MJ90,Bou00} as an application of the Luna-Vust theory. 

(We refer to \cite{IP99} for a general survey on Fano varieties.)

\begin{proposition}\label{prop:embeddings in type E}
Let $H \subseteq \SL_2(\C)$  be a finite subgroup, and let $X_0=G/H$.
\begin{itemize}[leftmargin=7mm]
\item If $H$ is conjugate to $E_6$, then the smooth quadric $Q_3 \subseteq \P^4=\P(R_4)$ is the unique minimal smooth completion of $X_0$, and $\Aut_{\C}^{G}(Q_3)=\{Id\}$.
\item If $H$ is conjugate to $E_7$, then the Fano threefold $V_5$ is the unique minimal smooth completion of $X_0$, and $\Aut_{\C}^{G}(V_5)=\{Id\}$.
\item If $H$ is conjugate to $E_8$, then the Fano threefold $V_{22}^{MU}$ is the unique minimal smooth completion of $X_0$, and $\Aut_{\C}^{G}(V_{22}^{MU})=\{Id\}$.
\end{itemize}
Moreover, the diagrams associated with these three minimal smooth completions of $X_0$ are pictured in Figure~\ref{fig:third example}.
\end{proposition}

\begin{figure}[ht]
\begin{tikzpicture}[scale=1.3]
\path (0,0) coordinate (origin);
\draw (0,0) circle (2pt) ;
\path (0:1.7cm) coordinate (P0);
\path (0:1.1cm) coordinate (P0););
\path (75:1.3cm) coordinate (P1););
\path (19:.75cm) coordinate (P2);
\path (310:1.3cm) coordinate (P3);
\path (55:.75cm) coordinate (P4);
\path (-30:.75cm) coordinate (P5);
\path (90:.8cm) coordinate (Q0);
\draw[line width=0.3mm] (origin) -- (P0) (origin) -- (P3) (origin) -- (P4) (origin) -- (P5) (origin) -- (P1) (origin) -- (P2) ;
\fill[white] (0,0) circle (1.5pt) ;
\node at (.5,.5){$\vdots$};
\node at (.6,.8){\tiny{$-\frac{5}{6}$}};
\node at (0,1.3){\tiny{$-\frac{1}{2}$}};
\node at (.4,-1.1){\tiny{$-\frac{1}{2}$}};
\node at (1.5,0){\tiny{$-\frac{2}{3}$}};
\node at (-.2,-.2){\tiny{$-1$}};

\node at (.3,-2){\tiny{$H=E_6$}};
\draw[line width=0.3mm] (P3) -- +(40:3pt)  (P3)-- +(220:3pt);

\end{tikzpicture}
\quad
\begin{tikzpicture}[scale=1.3]
\path (0,0) coordinate (origin);
\draw (0,0) circle (2pt) ;
\path (0:1.7cm) coordinate (P0);
\path (0:1.1cm) coordinate (P0););
\path (75:1.1cm) coordinate (P1););
\path (19:.75cm) coordinate (P2);
\path (310:1.5cm) coordinate (P3);
\path (55:.75cm) coordinate (P4);
\path (-30:.75cm) coordinate (P5);
\path (90:.8cm) coordinate (Q0);
\draw[line width=0.3mm] (origin) -- (P0) (origin) -- (P3) (origin) -- (P4) (origin) -- (P5) (origin) -- (P1) (origin) -- (P2) ;
\fill[white] (0,0) circle (1.5pt) ;
\node at (.5,.5){$\vdots$};
\node at (.6,.8){\tiny{$-\frac{11}{12}$}};
\node at (0,1.3){\tiny{$-\frac{5}{6}$}};
\node at (.4,-1.1){\tiny{$-\frac{2}{3}$}};
\node at (1.5,0){\tiny{$-\frac{3}{4}$}};
\node at (-.2,-.2){\tiny{$-1$}};

\node at (.3,-2){\tiny{$H=E_7$}};
\draw[line width=0.3mm] (P3) -- +(40:3pt)  (P3)-- +(220:3pt); 

\end{tikzpicture}
\quad
\begin{tikzpicture}[scale=1.3]
\path (0,0) coordinate (origin);
\draw (0,0) circle (2pt) ;
\path (0:1.7cm) coordinate (P0);
\path (0:1.1cm) coordinate (P0););
\path (75:1.1cm) coordinate (P1););
\path (19:.75cm) coordinate (P2);
\path (310:1.5cm) coordinate (P3);
\path (55:.75cm) coordinate (P4);
\path (-30:.75cm) coordinate (P5);
\path (90:.8cm) coordinate (Q0);
\draw[line width=0.3mm] (origin) -- (P0) (origin) -- (P3) (origin) -- (P4) (origin) -- (P5) (origin) -- (P1) (origin) -- (P2) ;
\fill[white] (0,0) circle (1.5pt) ;
\node at (.5,.5){$\vdots$};
\node at (.6,.8){\tiny{$-\frac{29}{30}$}};
\node at (0,1.3){\tiny{$-\frac{14}{15}$}};
\node at (.4,-1.1){\tiny{$-\frac{5}{6}$}};
\node at (1.5,0){\tiny{$-\frac{9}{10}$}};
\node at (-.2,-.2){\tiny{$-1$}};

\node at (.3,-2){\tiny{$H=E_8$}};
\draw[line width=0.3mm] (P3) -- +(40:3pt)  (P3)-- +(220:3pt);

\end{tikzpicture}
\caption{Diagrams for the minimal smooth completions\newline  of $G/H$ when $H=E_6$, $E_7$, and $E_8$.}
\label{fig:third example}
\end{figure}
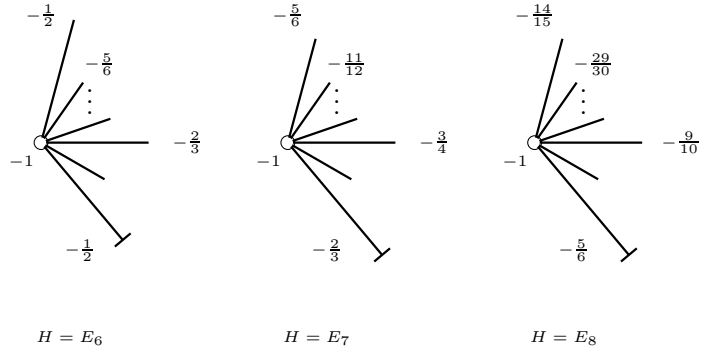

\begin{proof}
Assume first that $H$ is conjugate to $E_7$ or $E_8$. The fact that $X_0=G/H$ admits a unique minimal smooth completion is \cite[Theorem~2.1]{MU83} (see also \cite[\S~6.2.2, Proposition~7]{Bou00}), and the underlying $G$-variety can then be identified from \cite[Lemma~3.3]{MU83}. The fact that the groups $\Aut_{\C}^{G}(V_5)$ and $\Aut_{\C}^{G}(V_{22}^{MU})$ are trivial follows from Lemma~\ref{lem:N(H)/H} (together with Proposition~\ref{prop:G-eq automorphisms of G/H}).

Assume now that $H$ is conjugate to $E_6$.
The fact that $X_0$ admits a unique minimal smooth completion is \cite[Proposition~2.5]{Nak89} (see also \cite[\S~6.2.1, Proposition~7]{Bou00}), and the underlying $G$-variety is identified in \cite[Proposition~2.4]{Nak89}. 
The fact that the group $\Aut_{\C}^{G}(Q_3)$ is trivial follows from Lemma~\ref{lem:N(H)/H}, combined with an explicit calculation to check that the unique non-trivial $G$-equivariant automorphism of the open orbit $X_0$ does not extend to the whole variety $Q_3$.

The diagrams associated with these three minimal smooth completions can be found in \cite[\S\S~6.2.1-6.2.2]{Bou00}.
\end{proof}

\begin{corollary}\label{cor:structures on minimal models E cases}
Let $\sigma \in \{\sigma_s, \sigma_c\}$. Then, for $H \subseteq G=\SL_2(\C)$ conjugate to $E_6$, $E_7$ or $E_8$, the unique minimal smooth completion of $X_0=G/H$ admits exactly one $(G,\sigma)$-equivariant real structure. 
Hence, given a real form $F$ of $G$, there is exactly one $(\R,F)$-form for each of the three $G$-varieties $Q_3$, $V_5$ and $V_{22}^{MU}$.
\end{corollary}

\begin{proof}
Let $X$ be one of the  three $G$-varieties $Q_3$, $V_5$, or $V_{22}^{MU}$.
First note that, if a $(G,\sigma)$-equivariant real structure on $X$ exists, then it is unique. Indeed, otherwise $\mu_1 \circ \mu_2$ would give a non-trivial element of $\Aut_{\C}^{G}(X)=\{Id\}$. 
Also, we can assume without loss of generality that $H \in \{E_6,E_7,E_8\}$.

The $\Gamma$-actions on the collection of colored data of a minimal smooth completion of $X_0$ permute spokes of the diagram (see Figure~\ref{fig:third example}) that have the same length. This implies that, when $H=E_7$ or $H=E_8$, each $(G,\sigma)$-equivariant real structure on $X_0$ extends uniquely to the minimal smooth completion. This leaves the case of $H=E_6$, where the $\Gamma$-action can either fix the two spokes of the same length or exchange them. We will check for each $\Gamma$-action, associated with one of the $(G,\sigma)$-equivariant real structure on $X_0$ listed in Appendix~\ref{ap:table2}, which situation holds. 

We follow the notation of Appendix~\ref{sec:appendix B} and fix $H=E_6$. 
We now determine $j \in \DD^B(G/H) \simeq \P^1/H$ such that $b(j)=-1/2$. To do this, we find $\widehat{j}=[\alpha:\beta] \in\P^1$ such that the $H$-orbit of $\widehat{j}$ is of order $4$. 
Let $f_1=\alpha w+\beta x$, $f_2$, $f_3$, and $f_4$ be the four corresponding elements in $\C[G]$ (each of them defined up to a nonzero scalar).
Let $\nu=\nu([\alpha:\beta],1)\in\Vgt$;  its restriction to $\C[G]^H$, after normalization, satisfies the condition that $\nu(F^6)=\frac{6}{24}(1-1-1-1)=-1/2$, where $F=f_1 f_2 f_3 f_4 \in \C[G]^{(H)}$ is an $H$-semi-invariant such that $F^6 \in \C[G]^H$. Thus the element $j\in \DD^B(G/H)$ has a spoke with $b(j)=-1/2$.  There are, up to a constant multiple, exactly two such homogeneous polynomials that we denote by $F_1$ and $F_2$. 
Indeed, by using the generators of $H$, we determine 
\[
\begin{array}{lcl}
F_1=(w+\beta_1 x)(w-\beta_1 x)(w+\beta_{1}^{-1}x)(w-\beta_{1}^{-1}x)=w^4+x^4-2\sqrt{3}iw^2x^2, \ \text{and}\\
F_2=(w+\beta_2 x)(w-\beta_2 x)(w+\beta_{2}^{-1}x)(w-\beta_{2}^{-1}x)=w^4+x^4+2\sqrt{3}iw^2x^2,
\end{array}
\]
where $\beta_1=(1+i)\frac{1+\sqrt{3}}{2}$ and $\beta_2=(1+i)\frac{1-\sqrt{3}}{2}$.

Let $j_1,j_2\in\DD^B(G/H)$ be defined by $F_1$ and $F_2$ respectively. 
For each $(G,\sigma)$-equivariant real structure $\mu$ on $X_0$, either $\gamma \cdot j_1=j_1$ (and $\gamma \cdot j_2=j_2$), in which case $\mu$ extends to $Q_3$, or $\gamma \cdot j_1=j_2$, in which case $\mu$ does not extend to $Q_3$.
A straightforward computation (see \S~\ref{sec:almost homogeneous SL2 threefolds}) yields that
\[
 \left\{
    \begin{array}{lclcl}
      \gamma \cdot j_1=j_1 & \text{ for } &\mu\colon \modH{g} \to \modH{\sigma_s(g)\w_8} &\text{ and } &\mu\colon \modH{g} \to \modH{\sigma_c(g)\w_8}; \\
         \gamma \cdot j_1=j_2 & \text{ for } &\mu\colon \modH{g} \to \modH{\sigma_s(g)} &\text{ and } &\mu\colon \modH{g} \to \modH{\sigma_c(g)}.
    \end{array}
\right.
\]
This means that, for $\mu$ of the form $\modH{g} \to \modH{\sigma_s(g)\w_8}$ or $\modH{g} \to \modH{\sigma_c(g)\w_8}$, the equivariant real structure extends to $Q_3$ while for the other two choices of $\mu$, the equivariant real structure does not extend to $Q_3$.

The last sentence of the statement follows from the fact that the three $G$-varieties $Q_3$, $V_5$ and $V_{22}^{MU}$ are projective, and so any equivariant descent datum is effective. 
\end{proof}

To describe the minimal smooth completions of $X_0=G/H$, when $H$ is conjugate to $D_n$, we  introduce the following $\P^1$-bundle over $\P^2$.

\begin{definition}
Let $m\ge 0$ and let $\kappa\colon~\p^1\times\p^1 \to \p^2$ be the $(2:1)$-cover defined by
\[\begin{array}{rccc}
\kappa\colon~& \P^1\times\P^1 & \to &\P^2\\
& ([y_0:y_1],[z_0:z_1]) &\mapsto &[y_0 z_0:y_0 z_1+y_1 z_0:y_1z_1],\end{array}\]
whose branch locus is the diagonal $\Delta\subseteq\P^1\times\P^1$ and whose ramification locus is the smooth conic $\mathscr{C}=\{ [x:y:z] \mid y^2=4xz\}\subseteq\P^2.$
 The \emph{$m$-th Schwarzenberger $\P^1$-bundle} $\SS_m\to \P^2$ is the $\P^1$-bundle defined by \[\SS_m=\P(\kappa_* \O_{\P^1\times\P^1}(-m,0))\to \P^2.\] 
\end{definition}

Note that $\SS_m$ is the projectivization of the classical Schwarzenberger vector bundle $\kappa_* \O_{\P^1\times\P^1}(-m,0)$ introduced in \cite{Sch61}. It was first studied by Umemura in \cite[\S~4]{Ume88} and then by Blanc-Fanelli-Terpereau in \cite[\S~4.2]{BFT23} in order to determine the maximal connected algebraic subgroups of the Cremona group $\Bir(\P^3)$.

\begin{proposition}\label{prop:embeddings in type D}
Let $H \subseteq \SL_2(\C)$ be a finite subgroup, and let $X_0=G/H$.
\begin{itemize}[leftmargin=6mm]
\item If $H$ is conjugate to $D_4$, then $\SS_{2} \simeq \P(T_{\P^2})$ is the unique minimal smooth completion of $X_0$, and $\Aut_{\C}^{G}(\SS_2) \simeq \Z/2\Z$.
\item If $H$ is conjugate to $D_5$, then $\P^3=\P(R_3)$ is the unique minimal smooth completion of $X_0$, and $\Aut_{\C}^{G}(\P^3)=\{Id\}$.
\item If $H$ is conjugate to $D_n$ with $n \geq 6$, then $\SS_{n-2}$ is the unique minimal smooth completion of $X_0$, and $\Aut_{\C}^{G}(\SS_{n-2})=\{Id\}$.
\end{itemize}
Moreover, the diagrams associated with these $G$-equivariant embeddings of $X_0$ are pictured in Figure~\ref{fig:fourth example}. 
\end{proposition}
\vspace{-2mm}
\begin{figure}[ht]
\begin{tikzpicture}[scale=1.3]
\path (0,0) coordinate (origin);
\draw (0,0) circle (2pt) ;
\path (0:1.7cm) coordinate (P0);
\path (0:1.7cm) coordinate (P0););
\path (85:1.7cm) coordinate (P1););
\path (19:.75cm) coordinate (P2);
\path (290:1.7cm) coordinate (P3);
\path (55:.75cm) coordinate (P4);
\path (-30:.75cm) coordinate (P5);
\path (90:.8cm) coordinate (Q0);
\draw[line width=0.3mm] (origin) -- (P0) (origin) -- (P3) (origin) -- (P4) (origin) -- (P1) (origin) -- (P2) ;
\fill[white] (0,0) circle (1.5pt) ;
\node at (.5,.5){$\vdots$};
\node at (.6,.8){\tiny{$-\frac{1}{2}$}};
\node at (-.1,1.7){\tiny{$0$}};
\node at (.1,-1.5){\tiny{$0$}};
\node at (1.9,0){\tiny{$0$}};
\node at (-.2,-.2){\tiny{$-1$}};
\draw[line width=0.3mm] (P0) -- +(90:3pt)  (P0)-- +(270:3pt); 
\draw[line width=0.3mm] (P1) -- +(175:3pt)  (P1)-- +(355:3pt); 

\node at (.3,2){\tiny{$H=D_4$ }};

\end{tikzpicture}
\quad\quad
\begin{tikzpicture}[scale=1.3]
\path (0,0) coordinate (origin);
\draw (0,0) circle (2pt) ;
\path (0:1.7cm) coordinate (P0);
\path (0:1.7cm) coordinate (P0););
\path (85:1.3cm) coordinate (P1););
\path (19:.75cm) coordinate (P2);
\path (290:1.3cm) coordinate (P3);
\path (55:.75cm) coordinate (P4);
\path (-30:.75cm) coordinate (P5);
\path (90:.8cm) coordinate (Q0);
\draw [line width=0.3mm](origin) -- (P0) (origin) -- (P3) (origin) -- (P4) (origin) -- (P1) (origin) -- (P2) ;
\fill[white] (0,0) circle (1.5pt) ;
\node at (.5,.5){$\vdots$};
\node at (.6,.8){\tiny{$-\frac{2}{3}$}};
\node at (-.1,1.5){\tiny{$-\frac{1}{3}$}};
\node at (.2,-1.1){\tiny{$-\frac{1}{3}$}};
\node at (1.9,0){\tiny{$0$}};
\node at (-.2,-.2){\tiny{$-1$}};
\draw[line width=0.3mm] (P1) -- +(175:3pt)  (P1)-- +(355:3pt);

\node at (.3,2){\tiny{$H=D_5$ }};

\end{tikzpicture}
\quad\quad
\begin{tikzpicture}[scale=1.3]
\path (0,0) coordinate (origin);
\draw (0,0) circle (2pt) ;
\path (0:1.7cm) coordinate (P0);
\path (0:1.7cm) coordinate (P0););
\path (85:1.3cm) coordinate (P1););
\path (19:.75cm) coordinate (P2);
\path (290:1.3cm) coordinate (P3);
\path (55:.75cm) coordinate (P4);
\path (-30:.75cm) coordinate (P5);
\path (90:.8cm) coordinate (Q0);
\draw[line width=0.3mm] (origin) -- (P0) (origin) -- (P3) (origin) -- (P4) (origin) -- (P1) (origin) -- (P2) ;
\fill[white] (0,0) circle (1.5pt) ;
\node at (.5,.5){$\vdots$};
\node at (.6,.8){\tiny{$b$}};
\node at (-.1,1.5){\tiny{$b_1$}};
\node at (.2,-1.1){\tiny{$b_1$}};
\node at (1.9,0){\tiny{$0$}};
\node at (-.2,-.2){\tiny{$-1$}};

\node at (.4,-1.5){\tiny{$b_1=\frac{4-n}{n-2}$, $b=\frac{3-n}{n-2}$}};
\node at (.3,2){\tiny{$H=D_n$ , $n\ge 6$}};
\draw[line width=0.3mm] (P0) -- +(90:3pt)  (P0)-- +(270:3pt); 
\draw[line width=0.3mm] (P1) -- +(175:3pt)  (P1)-- +(355:3pt); 

\end{tikzpicture}
\caption{Diagrams for the minimal smooth completions\newline of $G/H$ when $H=D_n$ ($n \geq 4$).}
\label{fig:fourth example}
\end{figure}
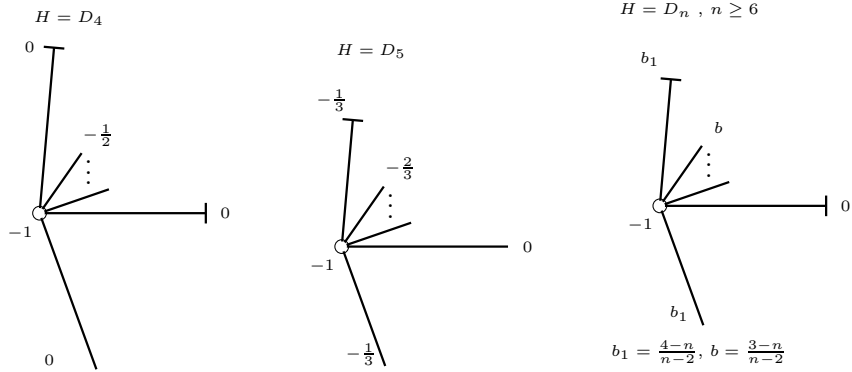

\begin{proof}
The fact that $X_0$ admits a unique minimal smooth completion is \cite[Theorem~2.8]{Nak89} (see also \cite[\S~6.2.5, Proposition~12]{Bou00}), and the underlying $G$-variety is described in \cite[Lemmas~2.6 and~2.7]{Nak89} (see also \cite[\S~4]{Ume88}). The fact that $\Aut_{\C}^{G}(\SS_{n-2})=\{Id\}$, for $n \geq 6$, follows for instance from the fact that $\Aut_{\C}(\SS_{n-2}) \simeq \PGL_2(\C)$ (see \cite[Lemma~4.2.5~(2)]{BFT23}). The fact that $\Aut_{\C}^{G}(\SS_2) \simeq \Z/2\Z$ and $\Aut_{\C}^{G}(\P^3)=\{Id\}$ comes from an explicit calculation.

The diagrams associated with these minimal smooth completions can be found in \cite[\S~6.2.5]{Bou00}.
\end{proof}


\begin{corollary} \label{cor:structures on minimal models D cases}
Let $\sigma \in \{\sigma_s, \sigma_c\}$, and let $H$ be a subgroup of $G=\SL_2(\C)$ conjugate to $D_n$ with $n \geq 4$.
\begin{enumerate}[leftmargin=7mm]
\item If $n=4$, then the unique minimal smooth completion of $G/H$ admits exactly two inequivalent $(G,\sigma)$-equivariant real structures. Hence, given a real form $F$ of $G$, there are exactly two non-isomorphic $(\R,F)$-forms for the $G$-variety $\SS_2$.
\item If $n\ge 5$, then the unique minimal smooth completion of $G/H$ admits exactly one $(G,\sigma)$-equivariant real structure. Hence, given a real form $F$ of $G$, there is exactly one $(\R,F)$-form for the $G$-varieties $\P^3$ and $\SS_{n-2}$ with $n \geq 6$.
\end{enumerate}
\end{corollary}

\begin{proof}
The proof is done similarly to that of Corollary~\ref{cor:structures on minimal models E cases}. 
We can assume without loss of generality that $H=D_n$ with $n \geq 4$.
The polynomials which correspond to the three longer spokes are 
\[
F_1=wx,\  \ F_2=w^{n-2}+x^{n-2},  \text{ and } \  F_3=w^{n-2}-x^{n-2}.
\] 
It suffices then to check, for each equivariant real structure on $G/H$ listed in the table of Appendix~\ref{ap:table2}, whether the associated $\Gamma$-action fixes each of the three longer spokes when $n \geq 5$, and only permutes the two longer spokes corresponding to orbits of type $\CC$ when $n=4$.
The technique of finding which real structures extend to the almost homogeneous varieties is similar to the case of $H=E_6$. We omit the details here.
\end{proof}

\noindent \emph{Proof of Corollary~\ref{cor:F}:} 
Since $H$ is assumed to be non-cyclic, it is conjugate either to $E_n$ with $n \in \{6,7,8\}$ or to $D_n$ with $n \geq 4$. We can assume without loss of generality that $H=E_n$ or $H=D_n$. Also, since the real group structure $\sigma$ is strongly equivalent to $\sigma_s$ or $\sigma_c$, we can assume that $\sigma \in \{\sigma_s,\sigma_c\}$ by the second part of Proposition~\ref{prop:A}.

Now if $H=E_n$ then the result follows from Corollary~\ref{cor:structures on minimal models E cases}, and if $H=D_n$ then the result follows from Corollary~\ref{cor:structures on minimal models D cases}.
\qed

\begin{remark}\label{rk:cyclic case}
In \cite{Mou23}, using similar techniques, Moulin determined the real forms of the minimal smooth completions when $H$ is conjugate to $A_n$.  In this case the group $\Aut_{\C}^{G}(G/H)$ is infinite (see Lemma~\ref{lem:N(H)/H}) and, for each $n \geq 1$, there are between seven and eleven minimal smooth completions of $G/H$ to consider (depending on $n$), up to a $G$-equivariant automorphism of the dense open orbit $G/H$. Moreover, the underlying $G$-varieties are not all projective, which makes the question of the effectiveness of the Galois descent non-trivial in this case.  
\end{remark}

More precise information about which equivariant real structures on $X_0=G/H$ extend to the minimal smooth completion $X_0 \hookrightarrow X$, possibly after conjugating by an element of $\Aut_{\C}^{G}(X_0)$ in the case $H=D_4$, is given in the table below.
\smallskip
\small
\[
\renewcommand{\arraystretch}{1.1}
    \begin{array}{|c|c|c|}
\hline 
\text{Subgroup } H \subseteq \SL_2(\C) &  \mu \text{ on } \SL_2(\C)/H& \text{Does $\mu$ extends to $X$ ?}\\
\hline 
\hline 
             D_{n},\ n\ge 4    & \modH{g} \mapsto \modH{\sigma_s(g)}&  {YES}\\
          & \modH{g} \mapsto \modH{\sigma_s(g)\w_{4n-8}}&   {YES\ if\ n=4,\, NO\ if\ not}\\
          & \modH{g} \mapsto \modH{\sigma_c(g)}&    {YES\ if\ n\ is\ even,\, NO\ if\ not}\\
                & \modH{g} \mapsto \modH{\sigma_c(g)\w_{4n-8}}&   {YES\ if\ n\ is\ odd\ or\ n=4,\, NO\  if\ not}\\
       \hline
        E_6    &  \modH{g} \mapsto \modH{\sigma_s(g)}&     {NO}\\
                &  \modH{g} \mapsto \modH{\sigma_s(g)\w_8}&    {YES}\\
                 &  \modH{g} \mapsto \modH{\sigma_c(g)}&     {NO}\\
                &  \modH{g} \mapsto \modH{\sigma_c(g)\w_8}&    {YES}\\
          \hline
              E_7   &  \modH{g} \mapsto \modH{\sigma_s(g)} &    {YES}\\   
                        &  \modH{g} \mapsto \modH{\sigma_c(g)} &   {YES}\\
                    \hline
                    E_8   &  \modH{g} \mapsto \modH{\sigma_s(g)} &   {YES}\\
                       &  \modH{g} \mapsto \modH{\sigma_c(g)} &    {YES}\\
                    \hline
\end{array} 
\renewcommand{\arraystretch}{1}
\]
\normalsize

\bigskip

\appendix
\addappheadtotoc 

\addtocontents{toc}{\protect\setcounter{tocdepth}{0}}
\section{Some technical conditions}\label{sec:appendix A}

\subsection{Conditions (C), (F), and (W)}\label{sec:A.1}
In this subsection we explain what are the conditions (C), (F), and (W)  mentioned in the statement of Proposition~\ref{prop:colored data of B-charts} and in the proof of Proposition~\ref{prop:Gamma stability}. 

We follow the notation of Proposition~\ref{prop:colored data of B-charts}, and we denote $\mathbb{K}=\overline{k}(Z)$ and 
\[
\mathbb{K}^{(B)}=\{ f \in \mathbb{K}^*\ |\  \exists \chi \in \mathbb{X}(B),  \forall z \in Z, \forall b \in B,\ (b \cdot f)(z)=\chi(b)f(z) \}.
\]
We refer to \cite[\S~13]{Tim11} for the proofs of the following statements.
\begin{itemize}
\item \textbf{Condition (C)} is
\[
\forall \V_0 \subseteq \W \sqcup \RR,\ \V_0 \text{ finite},\ \exists f \in \mathbb{K}^{(B)} \text{ such that }  \left\{
    \begin{array}{ll}
       \nu(f) \geq 0, & \forall \nu \in \W \sqcup \RR; \text{ and } \\
         \nu(f)>0, & \forall \nu \in \V_0.
    \end{array}
\right.
\]
This condition is equivalent to the equality $\Quot(A[\W,\RR])=\mathbb{K}$.
\bigskip

\item \textbf{Condition (F)} is
\begin{align*}
&A[\W,\RR]^U=\overline{k}[f \in \mathbb{K}^{(B)}\ |\ \forall \nu \in \W \sqcup \RR, \nu(f) \geq 0] \text{ is finitely generated}\\ 
&\text{(where $U$ is the unipotent radical of $B$).}
\end{align*}
This condition is equivalent to the fact that $A[\W,\RR]$ is finitely generated.
\bigskip

\item \textbf{Condition (W)} is
\[
\forall \omega \in \W, \ \exists f \in \mathbb{K}^{(B)} \text{ such that }  \left\{
    \begin{array}{ll}
       \nu(f) \geq 0, & \forall \nu \in \W \sqcup \RR \setminus \{ \omega \}; \text{ and } \\
         \omega(f)<0. & 
    \end{array}
\right.
\]
This condition is equivalent to the fact that all valuations of $\W$ are essential for $A[\W,\RR]$.
\smallskip
Here the word \emph{essential} refers to ring theory: A valuation $\nu\in \W \sqcup \RR \sqcup (\DD \setminus \DD^B)$ is essential for $A[\W,\RR]$ if  it cannot be removed from the right-hand side of the equality  \eqref{eq:B-chart as intersection of valuation rings}. Let us note that by \cite[Proposition~13.7~(1)]{Tim11} all valuations from $\RR$ and $\DD \setminus \DD^B$ are essential for $A[\W,\RR]$.
\end{itemize}

\subsection{Conditions to meet to have a model over an algebraically closed field}\label{sec:A.2}
We now assume that $k=\k$ and we use the notation of \S~\ref{sec:LV over alg closed fields}.
Let $(\W_i,\RR_i)_{i \in I}$ be a collection of pairs with $\W_i \subseteq \V^G$ and $\RR_i \subseteq  \DD^B$. 

In this subsection we give the list of the conditions that the collection $(\W_i,\RR_i)_{i \in I}$ must meet to define an equivalence class of $G$-models of $X_0$. (These conditions are mentioned in the statements of Theorem~\ref{th:Luna-Vust alg closed field} and Theorem~\ref{th:C}.)

\begin{definition*}
For a given pair ($\W_i,\RR_i$) as above, we define its \emph{support} $\mathscr{S}_i$ as the subset of all elements $\nu_0 \in \V^G$ satisfying
\[
\begin{array}{l}
\forall f \in \mathbb{K}^{(B)}, (\nu(f) \geq 0,\ \forall \nu \in \W_i \sqcup \RR_i) \Rightarrow \nu_0(f) \geq 0,\\
\text{and if $>$ occurs in the left-hand side, then }\nu_0(f)>0. 
\end{array}
\]
\end{definition*} 

By \cite[\S~14.2]{Tim11}, the collection $(\W_i,\RR_i)_{i \in I}$ defines an equivalence class of $G$-models of $X_0$ if and only if 
\begin{itemize}[leftmargin=*]
\item the supports $\mathscr{S}_i$, $i \in I$, are non-empty and pairwise disjoints in $\V^G$; and
\item there exist $r \geq 1$ and $I_1,\ldots,I_r$ subsets of $I$ such that
\smallskip

\begin{enumerate}[leftmargin=*]
\item $\bigcup_{j=1}^{r} I_j=I$;
\item for each $j \in \{1,\ldots,r\}$, the pair $\bigcup_{i \in I_j} \left( \W_i, \RR_i\right)$ satisfies the conditions (C), (F), and (W) listed in \S~\ref{sec:A.1} (with $Z=X_0$); and
\item for each $i_0 \in I$, for each $j \in \{1,\ldots,r\}$ such that $i_0 \in I_j$, the element
$\omega \in \bigcup_{i \in I_j}  (\W_i \sqcup \RR_i)$ belongs to $\W_{i_0} \sqcup \RR_{i_0}$ if and only if  we have  $\omega(f)=0$ for all $f \in \mathbb{K}^{(B)}$ such that
\[
\nu_0(f)=0 \text{ and } \nu(f) \geq 0,\ \forall \nu \in \bigcup_{i \in I_j} (\W_i \sqcup \RR_i),
\]
where $\nu_0$ is some fixed element of $\mathscr{S}_{i_0}$.
\end{enumerate}
\end{itemize}

\begin{remark*}
The number $r$ and the subsets $I_1,\ldots,I_r$ of $I$ are not unique in general (see \cite[Remark~14.3]{Tim11}).
\end{remark*}

\section{Luna-Vust theory for equivariant\\ embeddings of \texorpdfstring{$\SL_2(\C)/H$}{SL2(C)/H}}\label{sec:appendix B}
In this appendix we recall how the Luna-Vust theory specializes to classify the equivariant embeddings of homogeneous $\SL_2(\C)$-threefolds; this is used in \S\S\ref{sec:almost homogeneous SL2 threefolds}-\ref{sec:minimal models}.

This description, in terms of "skeleton diagrams" (see Figure~\ref{Vg diagram} below), was first given by Luna-Vust \cite[\S~9]{LV83} for equivariant embeddings of $\SL_2(\C)$, and then later generalized to equivariant embeddings of homogeneous $\SL_2(\C)$-threefolds by Moser-Jauslin \cite{MJ87,MJ90} and Bousquet \cite{Bou00}.

\begin{remark*}
There is another description for equivariant embeddings of $\SL_2(\C)/H$, in terms of colored hypercones, due to Timashev \cite[\S~5]{Tim97} (see also \cite{BMJ04} for a smoothness criterion in terms of these data), but in this article we chose to use the former description.
\end{remark*}

\subsection{The colors}\label{sec:B1}
Let $H \subseteq G=\SL_2(\C)$ be a finite subgroup, and let $B=\begin{bmatrix}
* & 0 \\ * & *
\end{bmatrix}$ which is a Borel subgroup of $G$.
We recall that the set of colors of $G/H$ is
\[\DD^B=\DD^B(G/H)=\{ \text{$B$-stable prime divisors of $G/H$} \}.\]
This set identifies with 
\[B\backslash G/H\simeq \P^1/H,\ BgH \mapsto [1:0] \cdot gH.
\] 
Let us make the identification between $\DD^B$ and $\P^1/H$ more explicit. We denote the coordinate ring of $G$ by $\C[G]=\C[w,x,y,z]/(xy-wz-1)$ with $g=\begin{bmatrix}
 x & w \\ z &y
\end{bmatrix} \in G$. Then the colors of $G$ are parametrized as follows:
\[
\forall [\alpha:\beta] \in \P^1, \ D_{[\alpha:\beta]}=Z(\alpha w+\beta x) \subseteq G.
\]
(Here $Z(\alpha w+\beta x)$ denotes the zero locus of the regular function $\alpha w+\beta x$ in $G$.)

Now consider the natural right action of $H$  on $\P^1\simeq B\backslash G$. The colors of $G/H$ identifies with $\P^1/H$ as follows: 
\[
\forall \modH{[\alpha_0:\beta_0]} \in \P^1/H, \ D_{\modH{[\alpha_0:\beta_0]}}=\overline{Z\left( \alpha w+\beta x\right)} \subseteq G/H,
\]
where $[\alpha:\beta]\in \P^1$ is any representative of $\modH{[\alpha_0:\beta_0]}$.

\subsection{The \texorpdfstring{$G$}{G}-invariant geometric valuations}
In this subsection, for each finite subgroup $H$ of $G=SL_2(\C)$, we describe the set 
\[
\V^G=\V^G(G/H)=\{\text{$G$-invariant geometric valuations of $\C(G/H)$}\}.
\]

We first consider the case where $H$ is trivial. 
Any $B$-invariant geometric valuation on $\C(G)$ is determined by its restriction on the subset 
\[
\P^1\simeq \{\alpha w+\beta x\ |\  [\alpha:\beta]\in\P^1\} \subseteq \C[G].
\] 
Moreover, if $\nu \in \V^G(G)$, then there exist $[\alpha_0:\beta_0]\in\P^1$  and relatively prime integers $p$ and $q$, with $q\geq 1$ with $-1\leq r=\frac{p}{q}\leq 1$ such that $\nu (\alpha_0 w+\beta_0 x)=p$ and $\nu(\alpha w+\beta x)=-q$ for all  $[\alpha:\beta] \neq [\alpha_0:\beta_0]$.  Therefore
\[
\V^G(G)\simeq \{ \P^1\times ([-1,1]\cap \Q) \}/ \sim, 
\]
where $([\alpha:\beta],-1)\sim ([0:1],-1)$ for all $[\alpha:\beta]\in\P^1$.  

After an appropriate normalization, we can  choose the element $\nu$ of $\Vgt$ corresponding to $([\alpha_0:\beta_0],r)$ to satisfy $\nu(\alpha_0 w+\beta_0 x)=r$ and $\nu(\alpha w+\beta x)=-1$ for all $[\alpha:\beta]\not=[\alpha_0:\beta_0]$. 
For $j \in\P^1$ and $r\in[-1,1]\cap\Q$, we denote by $\nu(j,r)$ the $G$-invariant geometric valuation corresponding to $(j,r)$. When $r=-1$, the valuation $\nu(j,-1)$ does not depend on $j$; we denote this valuation by $\nu_0$. 

\smallskip

Now for any finite subgroup $H \subseteq G$, the set $\V^G(G/H)$ is obtained from $\V^G(G)$. 
More specifically, the restriction map $\V^G(G) \to \V^G(G/H)$ is onto, where $\C(G/H)$ is identified with a subfield of $\C(G)$ through $\C(G/H) \simeq \C(G)^H \subseteq \C(G)$. Also, any valuation of $\V^G(G/H)$ is determined by its restriction on the subset 
\[
\left\{  \left(\prod_{[\alpha,\beta] \in [\alpha_0,\beta_0] \cdot H} \alpha w +\beta x \right)^{\frac{\Card(H)}{\Card([\alpha_0,\beta_0] \cdot H)}} \ \middle| \  \modH{[\alpha_0:\beta_0]} \in \P^1/H \right\} \subseteq \C[G]^H.
\]

For each $j\in\P^1/H$, let $\widehat{j} \in \P^1$ be a representative of $j$. Let $\widehat{r}\in[-1,1]\cap\Q$.
The restriction of the valuation $\nu(\widehat{j},\widehat{r}) \in \V^G(G)$ to $\C(G/H)$ yields a valuation of $\V^G(G/H)$ such that, after a suitable normalization,  its value on any $j'\in \P^1/H \setminus \{j\}$ is $-1$ and its value on $j$ is some rational number $r \geq -1$.
We denote this valuation by   $\nu(j,r)\in \V^G(G/H)$. 

Let $b(j)\in\Q$ be the maximum rational number $r$ such that $\Vg$ has a valuation of the form $\nu(j,r)$. In particular, for $H \subseteq \{ \pm I_2 \}$, we have $b(j)=1$ for all $j$. In general, $b(j)$ is indicated in the diagrams of Figure~\ref{Vg diagram}.  (Since $\V^G(G/H)$ depends on $H$ only up to conjugacy, we can always assume that $H$ is one of the finite subgroups listed in Remark \ref{prop:finite subgroups of SL2}.)

Let us note that the set of valuations corresponding to the closure of a two-dimensional $G$-orbit in some $G$-equivariant embedding of $G/H$ is given by
\[
\Vone= \Vg\setminus\{\nu_0\}.
\]
On the other hand, the valuation $\nu_0$ corresponds to an infinite union of $G$-orbits.

\begin{figure}[ht]
\begin{tikzpicture}[scale=1.3]
\path (0,0) coordinate (origin);
\draw (0,0) circle (2pt) ;
\path (90:1.5cm) coordinate (P0);
\path (135:1.5cm) coordinate (P1);
\path (200:1.5cm) coordinate (P2);
\path (300:1.5cm) coordinate (P3);
\path (30:1.5cm) coordinate (P4);
\path (90:.8cm) coordinate (Q0);
\draw (origin) -- (P0) (origin) -- (P1) (origin) -- (P2)(origin) -- (P3) (origin) -- (P4);
\fill[white] (0,0) circle (1.5pt) ;
\node at (-1,.4){$\vdots$};
\node at (-.1,1.5){$_1$};
\node at (-.1,-.2){\tiny{$-1$}};
\node at (0,2){\tiny{$H=A_1$ or $A_2$}};

\end{tikzpicture}
\begin{tikzpicture}[scale=1.3]
\path (0,0) coordinate (origin);
\draw (0,0) circle (2pt) ;
\path (75:1.5cm) coordinate (P0);
\path (0:1.5cm) coordinate (P3);
\path (55:1cm) coordinate (P4);
\path (15:1cm) coordinate (P5);
\draw (origin) -- (P0) (origin) -- (P3) (origin) -- (P4) (origin) -- (P5) ;
\fill[white] (0,0) circle (1.5pt) ;
\node at (.6,.6){{{$\vdots$}}};
\node at (.65,1){\tiny{$b$}};
\node at (.3,1.5){\tiny{$1$}};
\node at (1.6,0){\tiny{$1$}};
\node at (-.1,-.2){\tiny{$-1$}};
\node at (.3,-1){\tiny{$b=\frac{4}{m}-1$}};
\node at (.3,2){\tiny{$H=A_n$ , $n\ge 3$}};

\end{tikzpicture}
\begin{tikzpicture}[scale=1.3]
\path (0,0) coordinate (origin);
\draw (0,0) circle (2pt) ;
\path (0:1.7cm) coordinate (P0);
\path (0:1.7cm) coordinate (P0););
\path (85:1.3cm) coordinate (P1););
\path (19:.75cm) coordinate (P2);
\path (290:1.3cm) coordinate (P3);
\path (55:.75cm) coordinate (P4);
\path (-30:.75cm) coordinate (P5);
\path (90:.8cm) coordinate (Q0);
\draw (origin) -- (P0) (origin) -- (P3) (origin) -- (P4) (origin) -- (P1) (origin) -- (P2) ;
\fill[white] (0,0) circle (1.5pt) ;
\node at (.5,.5){$\vdots$};
\node at (.6,.8){\tiny{$b$}};
\node at (-.1,1.5){\tiny{$b_1$}};
\node at (.2,-1.1){\tiny{$b_1$}};
\node at (1.9,0){\tiny{$0$}};
\node at (-.2,-.2){\tiny{$-1$}};

\node at (.4,-1.5){\tiny{$b_1=\frac{4-n}{n-2}$, $b=\frac{3-n}{n-2}$}};
\node at (.3,2){\tiny{$H=D_n$ , $n\ge 4$}};

\end{tikzpicture}

\bigskip
\bigskip

\bigskip

\begin{tikzpicture}[scale=1.3]
\path (0,0) coordinate (origin);
\draw (0,0) circle (2pt) ;
\path (0:1.7cm) coordinate (P0);
\path (0:1.1cm) coordinate (P0););
\path (75:1.3cm) coordinate (P1););
\path (19:.75cm) coordinate (P2);
\path (310:1.3cm) coordinate (P3);
\path (55:.75cm) coordinate (P4);
\path (-30:.75cm) coordinate (P5);
\path (90:.8cm) coordinate (Q0);
\draw (origin) -- (P0) (origin) -- (P3) (origin) -- (P4) (origin) -- (P1) (origin) -- (P2) ;
\fill[white] (0,0) circle (1.5pt) ;
\node at (.5,.5){$\vdots$};
\node at (.6,.8){\tiny{$-\frac{5}{6}$}};
\node at (0,1.3){\tiny{$-\frac{1}{2}$}};
\node at (.4,-1.1){\tiny{$-\frac{1}{2}$}};
\node at (1.5,0){\tiny{$-\frac{2}{3}$}};
\node at (-.2,-.2){\tiny{$-1$}};

\node at (.3,-2){\tiny{$H=E_6$}};

\end{tikzpicture}
\quad\quad
\begin{tikzpicture}[scale=1.3]
\path (0,0) coordinate (origin);
\draw (0,0) circle (2pt) ;
\path (0:1.7cm) coordinate (P0);
\path (0:1.1cm) coordinate (P0););
\path (75:1.1cm) coordinate (P1););
\path (19:.75cm) coordinate (P2);
\path (310:1.5cm) coordinate (P3);
\path (55:.75cm) coordinate (P4);
\path (-30:.75cm) coordinate (P5);
\path (90:.8cm) coordinate (Q0);
\draw (origin) -- (P0) (origin) -- (P3) (origin) -- (P4)  (origin) -- (P1) (origin) -- (P2) ;
\fill[white] (0,0) circle (1.5pt) ;
\node at (.5,.5){$\vdots$};
\node at (.6,.8){\tiny{$-\frac{11}{12}$}};
\node at (0,1.3){\tiny{$-\frac{5}{6}$}};
\node at (.4,-1.1){\tiny{$-\frac{2}{3}$}};
\node at (1.5,0){\tiny{$-\frac{3}{4}$}};
\node at (-.2,-.2){\tiny{$-1$}};

\node at (.3,-2){\tiny{$H=E_7$}};

\end{tikzpicture}
\quad\quad
\begin{tikzpicture}[scale=1.3]
\path (0,0) coordinate (origin);
\draw (0,0) circle (2pt) ;
\path (0:1.7cm) coordinate (P0);
\path (0:1.1cm) coordinate (P0););
\path (75:1.1cm) coordinate (P1););
\path (19:.75cm) coordinate (P2);
\path (310:1.5cm) coordinate (P3);
\path (55:.75cm) coordinate (P4);
\path (-30:.75cm) coordinate (P5);
\path (90:.8cm) coordinate (Q0);
\draw (origin) -- (P0) (origin) -- (P3) (origin) -- (P4) (origin) -- (P1) (origin) -- (P2) ;
\fill[white] (0,0) circle (1.5pt) ;
\node at (.5,.5){$\vdots$};
\node at (.6,.8){\tiny{$-\frac{29}{30}$}};
\node at (0,1.3){\tiny{$-\frac{14}{15}$}};
\node at (.4,-1.1){\tiny{$-\frac{5}{6}$}};
\node at (1.5,0){\tiny{$-\frac{9}{10}$}};
\node at (-.2,-.2){\tiny{$-1$}};

\node at (.3,-2){\tiny{$H=E_8$}};

\end{tikzpicture}
\caption{The set $\Vg$. Each spoke of the diagram corresponds to a point of $\P^1/H$. The center corresponds to the valuation $\nu_0(\alpha w+\beta x)=-1$ for all $[\alpha,\beta]\in\P^1$. In the caption of $H=A_n$, we denote by $m$ the least common multiple of $2$ and $n$; in other words, $m=n$ if $n$ is even, and $m=2n$ if $n$ is odd.}
\label{Vg diagram}
\end{figure}
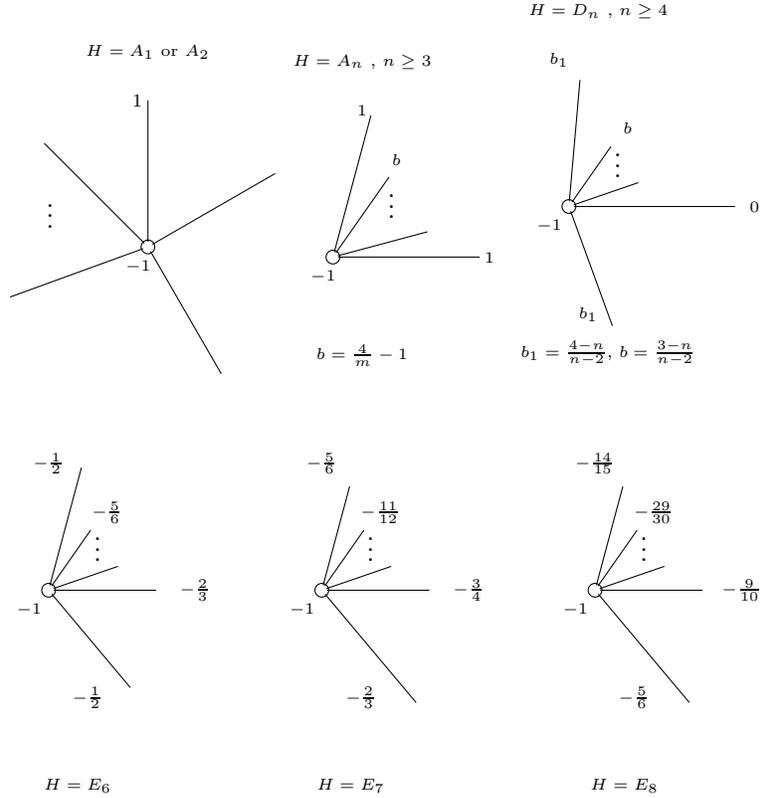

\subsection{Colored data and facets of orbits}\label{sec:B3} 
We now give a description of all possible colored data (see Definition~\ref{def:colored data of G-orbits}) of the $G$-orbits of the $G$-equivariant embeddings of $G/H$. We will also describe the {\it facet} of each orbit, i.e.~the set of valuations in $\Vone$  which dominate the local ring of the closure of the given orbit.

As in \cite{LV83}, we divide the (non-open) $G$-orbits into six types, and we describe the facet and give the colored data  $(\V_Y^G,\DD_Y^B)$ with $\V_Y^G\subseteq \V^G$ and $\DD_Y^B\subseteq \DD^B \simeq \P^1/H$ of a $G$-orbit $Y$ for each type.

\begin{itemize}[leftmargin=7mm]
\item[$\bullet$] {\bf Type $\CC(j,r)$} ($2$-dimensional orbit): $j\in\P^1/H$ and $r\in\Q\cap]-1,b(j)]$.\\ For this case, $\V_Y^G=\{\nu(j,r)\}$ is a single element of $\Vone$, 
 $\DD_Y^B=\varnothing$, and the facet of $Y$ is $\{\nu(j,r)\}$.

\item[$\bullet$]{\bf Type $\AA_N(j_1,\ldots,j_N,r_1,\ldots,r_N)$} (orbit isomorphic to $\P^1$):  $j_1,\ldots,j_N$ are  $N$ distinct elements of $\P^1/H$, and $r_i\in\Q\cap]-1,b(j_i)]$ for each $1 \leq i \leq N$.\\
Suppose also that $\sum_{i=1}^N\frac{1}{1+r_i}\ge1$.  Then $\V_Y^G=\{\nu(j_1,r_1),\ldots,\nu(j_N,r_N)\}$ and $\DD_Y^B=\P^1/H\setminus\{j_1,\ldots,j_N\}$. The facet of $Y$ is 
\[
\bigcup_{i=1}^N\{\nu(j_i,s_i)\ |\ -1<s_i<r_i\}\cup \{\nu(j,s)\ |\ j\not\in\{j_1,\ldots,j_N\},-1<s\le b(j)\}.
\]

\item[$\bullet$]{\bf Type $\AB(j,r_1,r_2)$} (orbit isomorphic to $\P^1$): $j\in\P^1/H$, and $r_1,r_2$ are two rational numbers satisfying $-1\le r_1<r_2\le b(j)$. Then $\V_Y^G=\{\nu(j,r_1), \nu(j,r_2)\}$ and $\DD_Y^B=\varnothing$. The facet of $Y$ is $\{\nu(j,s)\ |\ r_1<s<r_2\}$.

\item[$\bullet$]{\bf Type $\BB_+(j,r)$} (orbit isomorphic to $\P^1$): $j\in \P^1/H$, and $r\in\Q\cap[-1,b(j)[$. Then $\V_Y^G=\{\nu(j,r)\}$ and $\DD_Y^B=\{j\}$. The facet of $Y$ is $\{\nu(j,s)\ |\ r<s\leq b(j)\}$.
\item[$\bullet$]{\bf Type $\BB_-(j,r)$} (orbit isomorphic to $\P^1$): $j\in \P^1/H$, and $r\in\Q\cap]0,b(j)[$. Then $\V_Y^G=\{\nu(j,r))\}$ and $\DD_Y^B=\P^1/H\setminus\{j\}$. The facet of $Y$ is $\{\nu(j,s)\ |\ r<s\leq b(j)\}$.
\item[$\bullet$] {\bf Type $\BB_0(j,r)$} (fixed point): $j\in \P^1/H$, and $r\in\Q\cap]0,b(j)[$. Then $\V_Y^G=\{\nu(j,r))\}$ and $\DD_Y^B=\P^1/H$. The facet of $Y$ is $\{\nu(j,s)\ |\ r<s\leq b(j)\}$.

\end{itemize}

An orbit can be represented on the set $\Vg$ by marking its facet. Since the facets for orbits of types $\BB_-$, $\BB_0$ and $\BB_+$ are the same, one indicates the different types of orbits with a sign $0$, $+$ or $-$; see \S\S~\ref{sec:almost homogeneous SL2 threefolds}-\ref{sec:minimal models} for some examples.

\subsection{Conditions to meet to have an equivariant embedding}\label{sec:B4}
Using the facets, the conditions given in Appendix~\ref{sec:A.2} for determining which  set of orbits form a $G$-equivariant embedding $G/H \hookrightarrow X$ of $G/H$ simplifies as follows. 
A union of orbits form such a variety $X$ if and only if the following conditions are satisfied:
\begin{enumerate}[$(i)$]
\item The facets of the different orbits are disjoint.
\item The union of all the facets is closed in $\Vone$, where the topology is the one induced by the Euclidean topology on $\P^1 \times ([-1,1] \cap \Q)$.
\item If $X$ has an infinite number of orbits, then $X$ contains a $G$-stable hypersurface whose associated valuation is given by $\nu_0\in\Vg$. This means that $X$ contains all but finitely many orbits of type $\BB_+(j,-1)$. Also, in this case, the number of orbits of other types is finite.
\item If $X$ has a finite number of $G$-orbits, then it has no orbit of type $\BB_+(j,-1)$. 
\end{enumerate}

Moreover, the $G$-variety $X$ is complete if and only if the union of the facets is $\Vone$. It is quasiprojective if and only if it has at most one orbit of type $\BB_-$ or $\BB_0$. Note that type $\BB_-$ and $\BB_0$ orbits only occur if the subgroup $H$ is cyclic, since in all other cases, we have $b(j)\le 0$ for all $j \in \P^1/H$. In particular, if $H$ is non-cyclic, then any $G$-equivariant embedding $X$ of $G/H$ is quasiprojective, and so any descent datum on $X$ is effective (see \S~\ref{sec:descent data for algebraic varieties with group action}).

\begin{landscape}
\section{Equivariant real structures on \texorpdfstring{$\SL_2(\C)/H$}{SL2(C)/H}}\label{ap:table2}
\bigskip
\hspace{-15mm}
\scalebox{0.91}{
$
\renewcommand{\arraystretch}{1.2}
    \begin{array}{|c||c|c||c|c|}
\hline 
\text{Subgroup } H \subseteq \SL_2(\C) & \mu \text{ for } \sigma=\sigma_s &\text{real locus of }\mu& \mu \text{ for } \sigma=\sigma_c&\text{real locus of }\mu\\
\hline 
\hline 
  A_1    &  g \mapsto\sigma_s(g) &\SL_2(\R)&  g \mapsto \sigma_c(g)&\SU_2(\C)  \\
       &&     &   g \mapsto -\sigma_c(g)& \varnothing \\
  \hline
  A_2    &  \modH{g} \mapsto \modH{\sigma_s(g)} & \PSL_2(\R)\modH{I_2} \sqcup \PSL_2(\R) \modH{\w_4} \ (=\PGL_2(\R))  &\modH{g} \mapsto \modH{\sigma_c(g)} & \PSU_2(\C) \ (\simeq \SO_3(\R))   \\
      &\modH{g} \mapsto \modH{\sigma_s(g)e} & \varnothing &  \modH{g} \mapsto \modH{\sigma_c(g)e} & \varnothing \\
      \hline
A_n, n\ge 3, n \text{ odd} & \modH{g} \mapsto \modH{\sigma_s(g)}& \SL_2(\R) \modH{I_2} & \modH{g} \mapsto \modH{\sigma_c(g)}&\SU_2(\C)/A_n \\
 & \modH{g} \mapsto \modH{\sigma_s(g)f}& \SL_2(\R)\modH{d} 
  & \modH{g} \mapsto \modH{-\sigma_c(g)}& \varnothing \\
  \hline
 A_n, n\ge 4, n \text{ even} & \modH{g} \mapsto \modH{\sigma_s(g)}& \PSL_2(\R)\modH{I_2} \sqcup \PSL_2(\R) \modH{\w_{2n}} &\modH{g} \mapsto \modH{\sigma_c(g)}& \SU_2(\C)/A_n \\
 &\modH{g} \mapsto \modH{\sigma_s(g)e}&  \left\{
    \begin{array}{cl}
        \varnothing & \text{if } n \equiv 2\ [4] \\
        \PSL_2(\R)\modH{d} \sqcup  \PSL_2(\R) \modH{d^{-1}} & \text{if } n \equiv 0\ [4]
    \end{array}
\right. &\modH{g} \mapsto \modH{\sigma_c(g)e}&\varnothing \\
  &\modH{g} \mapsto \modH{\sigma_s(g)e\w_{2n}}&  \left\{
    \begin{array}{cl}
       \PSL_2(\R)\modH{d} \sqcup  \PSL_2(\R) \modH{d^{-1}}  & \text{if } n \equiv 2\ [4] \\
        \varnothing & \text{if } n \equiv 0\ [4]
    \end{array}
\right.  & \modH{g} \mapsto \modH{\sigma_c(g)\w_{2n}} &\varnothing \\
   \hline
           D_{n}, n\ge 4    & \modH{g} \mapsto \modH{\sigma_s(g)}& 
       \left\{
    \begin{array}{cl}
       \PSL_2(\R)\modH{I_2} \sqcup  \PSL_2(\R)\modH{\w_{4n-8}} \sqcup   \PSL_2(\R) \modH{d}  & \text{if } n \equiv 0\ [2] \\
      \PSL_2(\R)\modH{I_2}  \sqcup   \PSL_2(\R) \modH{d}  & \text{if } n \equiv 1\ [2]
    \end{array}
\right.     & \modH{g} \mapsto \modH{\sigma_c(g)} & \SU_2(\C)/D_n \\
                & \modH{g} \mapsto \modH{\sigma_s(g)\w_{4n-8}}&
                 \left\{
    \begin{array}{cr}
       \PSL_2(\R)\modH{\w_{8n-16}}  & \hspace{10mm}\text{if } n \equiv 0\ [2]\\
         \PSL_2(\R)\modH{\w_{8n-16}} \sqcup \PSL_2(\R) \modH{\w_{8n-16}^{-1}} & \text{if } n \equiv 1\ [2]
    \end{array}
\right.                    & \modH{g} \mapsto \modH{\sigma_c(g)\w_{4n-8}}& \varnothing \\
       \hline
        E_6    &  \modH{g} \mapsto \modH{\sigma_s(g)}&  \PSL_2(\R) \modH{I_2} & \modH{g} \mapsto \modH{\sigma_c(g)}& \SU_2(\C)/E_6  \\
                &  \modH{g} \mapsto \modH{\sigma_s(g)\w_8}&  \PSL_2(\R) \modH{\w_{16}} &  \modH{g} \mapsto \modH{\sigma_c(g)\w_8} & \varnothing  \\
          \hline
              E_7   &  \modH{g} \mapsto \modH{\sigma_s(g)} &  \PSL_2(\R) \modH{I_2} \sqcup \PSL_2(\R) \modH{\w_{16}} &  \modH{g} \mapsto \modH{\sigma_c(g)} & \SU_2(\C)/E_7 \\
                \hline
                    E_8   &  \modH{g} \mapsto \modH{\sigma_s(g)} &   \PSL_2(\R) \modH{I_2}  & \modH{g} \mapsto \modH{\sigma_c(g)} & \SU_2(\C)/E_8 \\
                    \hline
\end{array} 
$}
\end{landscape}

 \addtocontents{toc}{\protect\setcounter{tocdepth}{1}}
\bibliographystyle{alpha}

\def\cprime{$'$}

\end{document}